\documentclass[11pt]{article}
\usepackage{lipsum}
\usepackage[colorlinks]{hyperref}
\usepackage{mathrsfs,amsfonts}
\usepackage[dvips]{color}
\usepackage{amsmath}
\usepackage{amsthm}
\usepackage{comment}
\usepackage[final]{pdfpages}
\usepackage{graphicx}
\usepackage{float}
\usepackage{epstopdf}
\usepackage[english]{babel}
\usepackage{dsfont}
\usepackage{subfigure}
\usepackage{epsf}
\usepackage{enumitem}
\usepackage{comment}
\usepackage{amssymb}
\usepackage{tabularx}
\usepackage{color}

\hypersetup{
    pdfmenubar   = true,      
    colorlinks   = true,      
    linkcolor    = black,     
    citecolor    = red,       
    filecolor    = blue,      
    urlcolor     = blue       
}

 \newtheorem{theorem}{Theorem}[section]
 \newtheorem{lemma}[theorem]{Lemma}
 \newtheorem{corol}[theorem]{Corollary}
 \newtheorem{prop}[theorem]{Proposition}

 \theoremstyle{definition}
 \newtheorem{definition}[theorem]{Definition}
 \newtheorem{definitionAndLemma}[theorem]{Definition and Lemma}

 \def\Arg{\text{Arg}}
 
 \def\t{\tilde}
 \def\wt{\widetilde}
 \def\tbh{\tilde{\bar{h}}}

 \newcommand{\RN}[1]{%
 \textup{\uppercase\expandafter{\romannumeral#1}}%
}

 \def\qed{\hfill$\Box$\medskip}

\usepackage{tikz}
\newcommand{\lozengeincreasing}[1]{
\begin{tikzpicture}[#1]
\draw (0,0) -- (0,1ex);
\draw (0,0) -- (0.87ex,0.5ex);
\draw (0,1ex) -- (0.87ex,1.5ex);
\draw (0.87ex,0.5ex) -- (0.87ex,1.5ex);
\end{tikzpicture}
}

\newcommand{\lozengedecreasing}[1]{
\begin{tikzpicture}[#1]
\draw (0,1.5ex) -- (0,0.5ex);
\draw (0,1.5ex) -- (0.87ex,1ex);
\draw (0,0.5ex) -- (0.87ex,0ex);
\draw (0.87ex,1ex) -- (0.87ex,0);
\end{tikzpicture}
}

\newcommand{\lozengeconstant}[1]{
\begin{tikzpicture}[#1]
\draw (0,0.5ex) -- (0.87ex,1ex);
\draw (0,0.5ex) -- (0.87ex,0ex);
\draw (0.87ex,1ex) -- (1.74ex,0.5ex);
\draw (0.87ex,0ex) -- (1.74ex,0.5ex);
\end{tikzpicture}
}

\begin{document}
\selectlanguage{english}

\title{\large\bf DIMER MODEL, BEAD MODEL AND STANDARD YOUNG TABLEAUX: FINITE CASES AND LIMIT SHAPES}

\author{Wangru Sun \footnotemark[1]}

\date{}

\renewcommand{\thefootnote}{\fnsymbol{footnote}}

\maketitle

\footnotetext[1]{Laboratoire de Probabilit\'es et Mod\`eles Al\'eatoires, UMR 7599, Universit\'e Pierre et Marie Curie,
4 place Jussieu, 75005 Paris, France. wangru.sun@etu.upmc.fr}

{\narrower

\noindent{\bf Abstract.} The bead model is a random point field on $\mathbb{Z}\times\mathbb{R}$ which can be viewed as a scaling limit of dimer model. We prove that, in the scaling limit, the normalized height function of a uniformly chosen random bead configuration lies in an arbitrarily small neighborhood of a surface $h_0$ that maximizes some functional which we call as entropy. We also prove that the limit shape $h_0$ is a scaling limit of the limit shapes of a properly chosen sequence of dimer models. There is a map from bead configurations to standard tableaux of a (skew) Young diagram, and the map preserves uniform measures, and our results of the bead model yield the existence of the limit shape of a random standard Young tableau.
\smallskip

\smallskip




\bigskip

\section{Introduction}

The \textit{bead model} is a random point field on $\mathbb{Z}\times\mathbb{R}$ or a subset of it. A bead configuration is composed of a collection of parallel vertical threads, and on each thread there is a collection of points which we call the \emph{beads}. We furthermore ask a local finiteness and an interlacing relation on the vertical positions of the beads: for two consecutive beads on a thread, on each of its neighboring thread there is exactly one bead whose vertical position is between them. Figure~\ref{FigureBead} shows a typical configuration.

\begin{figure}[H]
\centering
\includegraphics[width=0.3\textwidth]{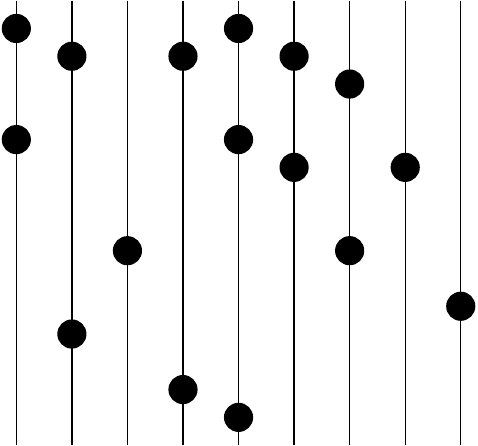}
\caption{A bead configuration.\label{FigureBead}}
\end{figure}

Boutillier \cite{BeadModel} considers this model on the infinite plane and constructs a family of ergodic Gibbs measures. This measure
is constructed as a limit of the dimer model measures on a bipartite graph when some weights degenerate. The author proves that under this measure the beads form a determinantal point process whose marginal is the sine process.

This paper focuses on the finite case. Authors of \cite{BeadFinitization} have done many works in the cases corresponding to $abc$-hexagons, and in this paper we focus on more general cases via approaches mainly inherited from \cite{CEP,CKP01}.

In Section 2, we describe the general setting of a bead model (Section~\ref{sqifmlumz}), define the height function (Section~\ref{qesfihjzemrau}), precisely define the boundary conditions (Section~\ref{qseliyrzame}) and define the uniform measure of the bead model (Section~\ref{sect5.2}). We also show that the bead model in such cases can be viewed as a limit of the dimer models (Section~\ref{eqsscifmjlkqsxfw} and~\ref{sqdfhyzqmurei}). Section~3 shows that every Young diagram (which can be skew) corresponds to one specific bead model and constructs a measure-preserving map from bead configurations to standard tableaux when considering the uniform measure.

We then consider the scaling limits of the bead models. As the bead model is some kind of limit of the dimer model, we can expect a result similar to that known in the dimer model \cite{CKP01}, that is, for a fixed asymptotic boundary condition, when the size of the domain tends to infinity, the normalized random surface converges in probability to a surface maximizing some functional called entropy. Once proved, by the measure preserving map in Section~3, this result directly yields the existence of the limit shape of standard (skew) Young tableaux, which generalizes the results of \cite{PR} and \cite{Sni}.

Since it is already very interesting, and also due to some technical reasons, for scaling limits we mainly consider the bead models corresponding to (skew) Young diagrams, which means with constant boundary height function on the left and right sides of the domains. Here below is an outline, where many ideas and technics come from \cite{CKP01} and \cite{CEP}.

In Section~\ref{erdsfdhzilmhfq}, we define the (adjusted) combinatorial entropy $S(.)$ of the bead model. The definition may appear not natural, but we will show that this gives a good order in the limit. In Section~\ref{sect5.3}, we consider the toroidal bead model and compute its free energy and the local entropy function $ent(.,.)$. We postpone the proof of the relation between the local entropy function $ent$ and the combinatorial entropy $S$ to Section~\ref{sectionProveVP}. In Sections~\ref{sect5.4} and~\ref{sectionProveVP}, we define the functional $Ent(.)$ which is almost the integral of $ent$ on the unit square $D=[0,1]\times[0,1]$, and define the space of admissible functions as the complete space of normalized height functions on $D$. We prove the following variational principle:

\begin{theorem}\label{sdfqsdfgeger}
For any given asymptotic boundary height function $h^{\partial}$ defined on $\partial D$ and being constant on $\{x=0\}$ and $\{x=1\}$, there is a unique function $h_0$ among the space of admissible functions that maximizes $Ent(.)$.
\end{theorem}

\begin{theorem}\label{thmmainbead}
Consider a given asymptotic boundary height function $h^{\partial}$ defined on  $\partial D$ and being constant on $\{x=0\}$ and on $\{x=1\}$. For any ${n\in\mathbb{N}^*}$, consider the bead model on $D$ with $n$ threads. For any admissible function ${h:D\rightarrow\mathbb{R}}$ such that $Ent(h)>-\infty$, when $n$ tends to infinity, the probability that the normalized surface of a random bead configuration lies within a $\delta$ neighborhood of $h$ is proportional to $e^{(Ent(h)+o(1))n^2}$ when $\delta\rightarrow 0$.
\end{theorem}

A more detailed version is given by Theorems~\ref{theoremexistence} and~\ref{thm-ent}. Please pay attention to the different uses of the same terminology ``entropy" in this paper:
\begin{itemize}
\item the adjusted combinatorial entropy $S$ of the bead model, see Section~\ref{erdsfdhzilmhfq}.
\item the local entropy function $ent$ as a function of the slope, see Section~\ref{surfacetensionlocalentropyfunction}.
\item the entropy function $Ent$ as a functional on the space of admissible functions, see Section~\ref{sect5.4part1}.
\end{itemize}

Note that the large deviation property (Theorem~\ref{thmmainbead}) particularly yields that when the size of the bead model is big, the random surface converges to $h_0$, which is the maximizer of the functional $Ent$ (see Theorem~\ref{Thm-Bead-convergenceinproba}).

As the bead model is a limit of the dimer model, it is natural to consider the following question: is the limit shape of the bead model a limit of the limit shapes of the dimer model? We give a positive answer, see below or Theorem~\ref{thm-hmcvg2h0} for details.

\begin{theorem}
The limit shape of the bead model $h_0$ is a properly normalized limit of the limit shapes of the lozenge tilings for the corresponding sequence of domains.
\end{theorem}

This theorem proves the commutativity of the following two limits: the limit from dimer models to a bead model when the heights tend to infinity, and the asymptotic limit for the dimer models on an increasing sequence of graphs with given asymptotic boundary condition, see Commutative Diagram~(\ref{commutativediagram}). Authors of \cite{KO1} provide a way to find the limit shape of the dimer model, especially for that of the hexagon lattice on domains with an asymptotic boundary condition piecewise linear in the direction of the edges of the hexagons. By the commutative diagram, their result implies directly a way to find the limit shape of the bead model.

As an example, if we consider a bead model on the unit square with the boundary condition given by ${h^{\partial}:\partial ([0,1]\times[0,1])\rightarrow\mathbb{R}},$
\begin{eqnarray}\label{sqlilfjrgeg}
h^{\partial}:(x,y)\mapsto
\begin{cases}
\frac{1}{4}-\frac{1}{2}|x-\frac{1}{2}| \ \text{ if }y\leq 0,\\
-\frac{1}{4}+\frac{1}{2}|x-\frac{1}{2}| \text{ if }y> 0,
\end{cases}
\end{eqnarray}
then Figure~\ref{BeadSquareu} is the expected density (which we will show is the vertical partial derivative of $h_0$), and Figure~\ref{BeadSquared} is a simulation of 256 beads.

\begin{figure}[H]
\centering
\subfigure[The estimated density of beads.\label{BeadSquareu}]{
\includegraphics[clip, trim=6.5cm 10cm 6.5cm 10cm, width=0.4\textwidth]{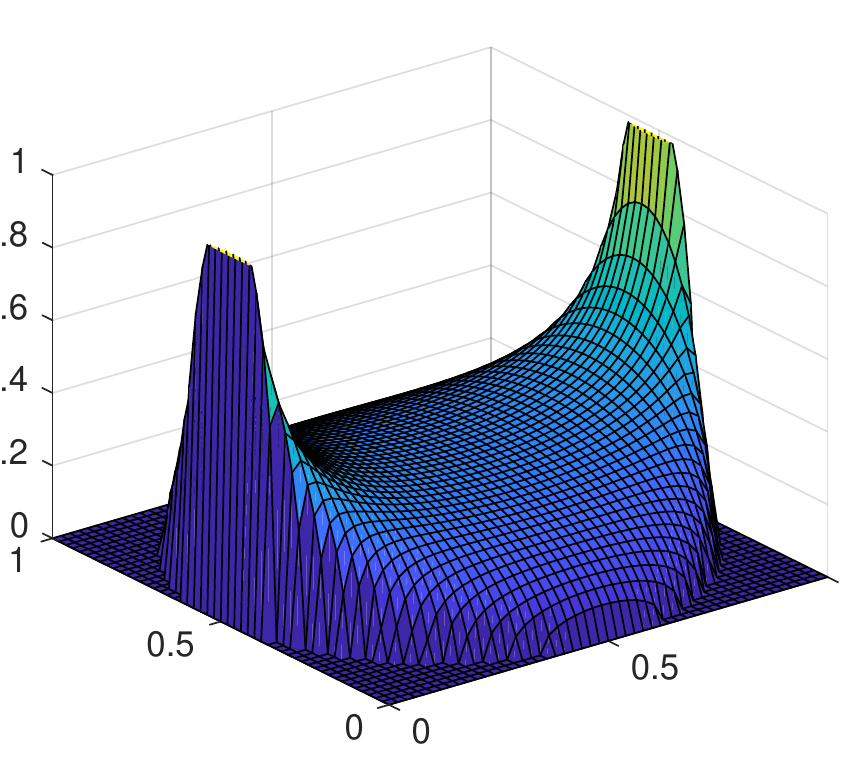}}
\subfigure[A simulation of 256 beads.\label{BeadSquared}]{
\includegraphics[clip, trim=7cm 10cm 7cm 10cm, width=0.4\textwidth]{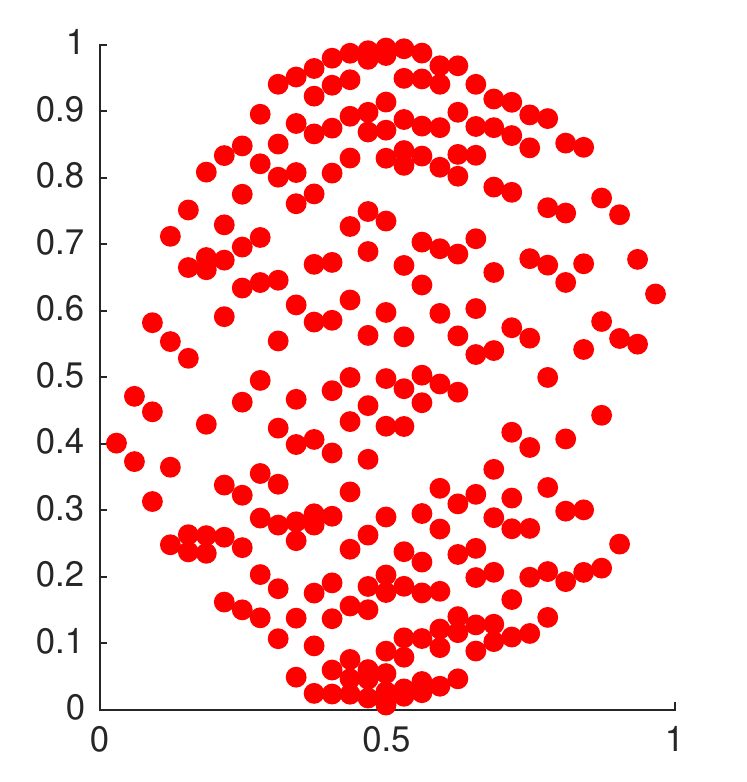}}
\caption{Theoretical and empirical density of beads.}
\end{figure}

In Section~\ref{sect6.4}, we apply the results on the bead model to random standard Young tableaux with a given asymptotic shape. We prove a surface version (Theorem~\ref{convergenceYoungSurface}) and a contour line version (Theorem~\ref{boundarymeasureconvergence}) of convergence of the tableaux, which generalize the results of \cite{PR} and \cite{Sni}, notably containing also the skew shapes. 

\vspace{0.5cm}

\noindent\textbf{Acknowledgements.} We would like to thank C\'edric Boutillier and B\'eatrice de Tili\`ere for their directions, comments and references. We would also like to thank Michel Pain for the helpful discussion.

\section{Presentation of the bead model}\label{The bead model}

\subsection{General setting of a bead configuration}\label{sqifmlumz}

Denote a bead configuration by $\mathbf{B}$. In this paper we focus on the case where the number of threads and that of beads are finite (but can be very large). Denote by $(i,y)$ the coordinate of a bead. We suppose that there are $n$ threads for some $n\in\mathbb{Z}$. Without loss of generality we suppose that the threads are ${\{i=1,2,...,n\}}$. For the vertical coordinates $y$ of the beads, we always suppose that $y$ takes value in $[0,1]$.

We consider the bead model on finite, planar, simply connected domains and on the torus. More precisely,
\begin{itemize}
\item \textit{The case of a finite planar simply connected domain.} Consider a planar simply connected domain ${R\subset]0,n+1[\times[0,1]}$. A bead configuration on the domain $R$ means that the coordinates of the beads $(i,y)$ take value in ${R\cap\left(\mathbb{Z}\times[0,1]\right)}$.
\item \textit{The case of a torus.} We suppose that ${(i,y)\in(\mathbb{Z}/n\mathbb{Z})\times(\mathbb{R}/\mathbb{Z})}$, so we can writes ${i\in\{1,...n\}}$ in the sense of modulo $n$ and $y\in[0,1[$ in the sense of modulo $1$.
\end{itemize}

Among the simply planar domains we are particularly interested in the rectangular case where $R=[1,n]\times[0,1]$, but due to some technical reasons we will also consider the case that $R$ is a right triangle.

\subsection{Bead configurations as limit of lozenge tilings: a first view}\label{eqsscifmjlkqsxfw}

The bead model can be viewed as a limit of lozenge tilings \cite{BeadModel,BeadFinitization}, which is equivalent to the dimer model on the hexagonal lattice. Throughout this paper we consider the following three types of lozenges:
\begin{itemize}
\item $\lozengeconstant{}$\ , generated by the vectors $(1,-\frac{1}{2})$ and $(1,\frac{1}{2})$,
\item $\lozengedecreasing{}$\ , generated by the vectors $(1,-\frac{1}{2})$ and $(0,1)$,
\item $\lozengeincreasing{}$\ , generated by the vectors $(1,\frac{1}{2})$ and $(0,1)$.
\end{itemize}

View the horizontal lozenges as ``beads", naturally located on the threads passing their centers, then such particles automatically verify the interlacing property. By let the vertical size of the domain or that of the torus tend to infinity and then vertically scale the domain into $[0,1]$ or $\mathbb{R}/\mathbb{Z}$, the discrete tiling model tends to continuous bead model \cite{BeadModel,BeadFinitization}.

\begin{figure}[H]
\centering
\includegraphics[width=0.35\textwidth]{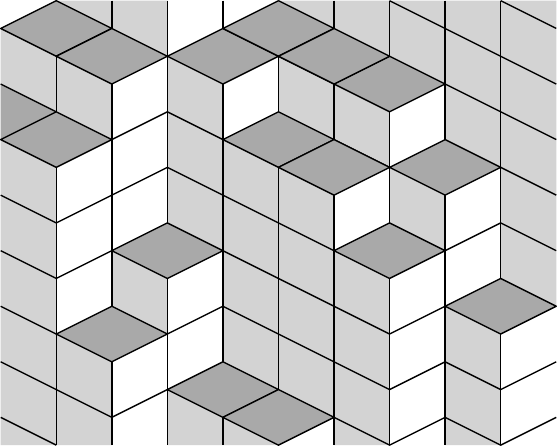}
\caption{A lozenge tiling corresponding to the bead configuration in Figure~\ref{FigureBead}.}
\end{figure}

\subsection{Height function}\label{qesfihjzemrau}

We first define the height function $H$ of lozenge tilings and then introduce the definition on bead configurations as an analogue. For a horizontal lozenge $\lozengeconstant{}$, the upper vertex is $1$ higher than the lower vertex, and the other two are equal to the average. For $\lozengeincreasing{}$ or $\lozengedecreasing{}$, vertices along the same vertical edge have the same height, and going right-up or left-up one step will raise the height by $\frac{1}{2}$.

\begin{figure}[H]
\centering
\includegraphics[width=0.25\textwidth]{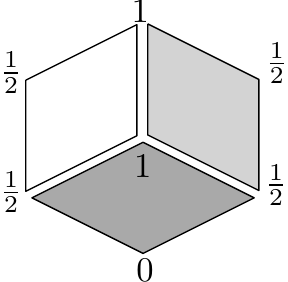}
\caption{Discrete height function $H$.}\label{heightfunctionlozengetiling}
\end{figure}


We consider the height function of the bead model as the following natural limit of the above definition when the vertical step size tends to $0$, where the vertices become a finite subset of points on threads ${\{1,2,...,n\}\times[0,1]}$. We still use the same letter $H$ which normally won't cause ambiguity.

\begin{definition}\label{unrenormalizedheightfunctionbeadmodel}
Consider a finite planar simply connected domain ${R\subset]0,n+1[\times[0,1]}$. Given a bead configuration $\mathbf{B}$ of the domain $R$, the \emph{height function} $H=H^{\mathbf{B}}$ (for convenience we omit $\mathbf{B}$) is the function
$$H:R\cap(\mathbb{Z}\times[0,1])\rightarrow\mathbb{R}$$
unique up to a constant which verifies the following conditions. The constant is fixed once we fixe the height of any point of $R\cap(\mathbb{Z}\times[0,1])$.
\begin{itemize}
\item The function $H$ is up-continuous, \emph{i.e.}, for any point $(i_0,y_0)\in D$,
      $$\lim_{y\rightarrow y_0^+}H(i_0,y)=H(i_0,y_0).$$
\item For any $i\in\{1,...,n\}$, and for any $y_1,y_2\in[0,1]$, $y_1<y_2$, ${H(i,y_2)-H(i,y_1)}$ is equal to the number of beads on the $i^{th}$ thread between $y_1$ and $y_2$.
\item If there is a bead at some point $(i_0,y_0)$, then on the neighboring threads ${i=i_0\pm1}$, we have
      $$\lim_{y\rightarrow y_0^-}H(i_0\pm 1,y)=H(i_0,y_0)-\frac{1}{2}.$$
\end{itemize}
\end{definition}



\vspace{0.5cm}

The toroidal case is an analogue, where a bead configuration defines a unique (up to a constant) multivalued height function $H$ defined on $(\mathbb{Z}/n\mathbb{Z})\times(\mathbb{R}/\mathbb{Z})$.

\subsection{Fixed and periodic boundary conditions}\label{qseliyrzame}

We begin by defining the fixed boundary conditions for a bead model in a simply connected region $R$.

\begin{definition}\label{qelisurqmzhfn}
For any planar simply connected domain ${R\subset]0,n+1[\times[0,1]}$, the bead model on it is said to have \emph{fixed boundary condition} if, given a fixed exterior bead configuration $\mathbf{B}^{ext}$ on $(\mathbb{Z}\times\mathbb{R})\backslash R$, the union of any bead configuration of $R$ and $\mathbf{B}^{ext}$ is a bead configuration of $\mathbb{Z}\times\mathbb{R}$.
\end{definition}

In other words, a fixed boundary condition is uniquely determined by the exterior bead configurations $\mathbf{B}^{ext}$ on ${(\mathbb{Z}\times\mathbb{R})\backslash R}$ modulo an equivalence relation (two exterior configurations are equivalent if they give the same restriction on the beads inside $R$).  These restriction are given by numbers of beads on every thread and the inequalities on the vertical coordinates of beads (other than that asked by the interlacing property).

In some cases it is easier to describe the boundary condition by fixing the height function on the boundary. For example, it is simple to verify that when $R=[1,n]\times[0,1]$, the following definition of a fixed boundary condition is reduced to Definition~\ref{qelisurqmzhfn}.

\begin{definition}\label{esqrfsdfhql}
For the bead model on $R=[1,n]\times[0,1]$, a function
$$H^{\partial}:(\{0,n+1\}\times[0,1])\cup(\{1,2,...,n\}\times\{0,1\})\rightarrow\mathbb{R}$$
is called \emph{boundary height function} if
\begin{itemize}
\item $H^{\partial}$ takes value in $\frac{1}{2}\mathbb{Z}$ up to a constant.
\item Restricted to $\{0\}\times[0,1]$ or $\{n+1\}\times[0,1]$, $H^{\partial}$ viewed as a function of $y$ is non-decreasing, piecewise constant and every jump is equal to $1$.
\item For every $i\in\{0,1,...,n+1\}$, $H^{\partial}(i,1)-H^{\partial}(i,0)\in\mathbb{N}$.
\item For every $i\in\{0,1,...,n\}$, $H^{\partial}(i+1,1)-H^{\partial}(i,1)$ and $H^{\partial}(i+1,0)-H^{\partial}(i,0)$ take values in $\{\pm\frac{1}{2}\}$.
\end{itemize}
A bead model on $R$ is said to have \emph{fixed boundary condition given by} $H^{\partial}$ if every bead configurations $\mathbf{B}$ can be extended to a bead configuration of ${\{0,...n+1\}\times[0,1]}$, and the height function of the extended configuration coincides with $H^{\partial}$ where $H^{\partial}$ is defined.
\end{definition}

Clearly, the number of beads is fixed by the height function, and it is equal to
$$\sum_{i=1}^{n}\big(H^{\partial}(i,1)-H^{\partial}(i,0)\big).$$

It is not hard to adapt the above definition into a more general shape of $R$.

We end this part by the following definition.
\begin{definition}
Let $R\subset]0,n+1[\times[0,1]$ be a simply connected domain, and let $\mathcal{U}$ be a subset of fixed boundary conditions of the bead model on $R$. The bead model is said to have the \emph{$\mathcal{U}$-boundary condition} if it contains all the configurations with fixed boundary conditions taken from $\mathcal{U}$.
\end{definition}

Especially, if $\mathcal{U}$ has only one element, the $\mathcal{U}$-boundary condition is just a fixed boundary condition, and if $\mathcal{U}$ contains all possible fixed boundary conditions, we say that the bead model has free boundary conditions.

\vspace{0.5cm}

Now consider the toroidal case. A toroidal bead configuration gives rise to a configuration in $\mathbb{Z}\times\mathbb{R}$, $n$-periodic in $i$ and $1$-periodic in $y$. As in the dimer model, for any $(i_0,y_0)\in\mathbb{Z}\times\mathbb{R}$, define the horizontal height change as
$$H_x=H(i_0+n,y_0)-H(i_0,y_0),$$
and the vertical height change as
$$H_y=H(i_0,y_0+1)-H(i_0,y_0).$$

It is not hard to see that when the number of beads is not $0$, $(H_x,H_y)$ takes value in
$$\left\{-\frac{n}{2}+1,-\frac{n}{2}+2,...,\frac{n}{2}-2,\frac{n}{2}-1\right\}\times\mathbb{N}^*,$$
independent of the choice of $(i_0,y_0)$.
\begin{definition}
For every given pair
$$(a,b)\in\left\{-\frac{n}{2}+1,-\frac{n}{2}+2,...,\frac{n}{2}-2,\frac{n}{2}-1\right\}\times\mathbb{N}^*,$$
we say that a toroidal model has \emph{periodic boundary condition} $(a,b)$ if its height change $(H_x,H_y)$ is equal to $(a,b)$.
\end{definition}

Clearly, the number of beads is fixed by the periodic conditions and equal to ${n H_y=nb}$.

\subsection{The uniform measure of the bead model}\label{sect5.2}

Consider a bead model with fixed boundary condition or periodic condition, which fixes the number of beads in the model. Denote the number of beads by $N$. The vertical coordinates can be viewed as a subset of $[0,1]^N$ or $\mathbb{T}^N$ (the $N$-dimensional torus). Moreover, the fixed boundary condition is equivalent to a collection of inequalities, so the set of the vertical coordinates is a convex set. The meaning of inequality is not clear for the toroidal case, but it is not hard to verify that the periodic condition also gives a convex subset of $\mathbb{T}^N$. In both cases, it makes sense to talk about the Lebesgue measure of the set of the vertical coordinates. Thus, we can define the uniform bead measure:

\begin{definition}\label{defunifrommeasurebeadmodel}
For a fixed, resp. periodic, boundary condition of the bead model with $N$ beads, the \emph{uniform bead measure} is the uniform probability measure of the vertical coordinates on the convex set determined by the fixed, resp. periodic, boundary condition, viewed as a subspace of $[0,1]^N$, resp. $\mathbb{T}^N$, equipped with the Lebesgue measure.
\end{definition}

In particular, under the uniform measure, the event that any two beads have the same vertical coordinate is a subspace of the convex of coordinates with lower dimension. So with probability $1$, the vertical coordinates of the beads are all different.

\subsection{Bead configuration as limit of lozenge tilings: a second view}\label{sqdfhyzqmurei}

Now that we have defined the fixed and periodic periodic conditions of a bead model and the uniform measure, the argument that ``the bead model is a limit of the lozenge tiling model" in Section~\ref{eqsscifmjlkqsxfw} can be described in a more detailed way. Although looks natural, the explicit construction in this section is to be used to make the proof of the variational principle (Section~\ref{sectionProveVP}) more rigorous.

We begin by simply connected planar domains. To simplify, we suppose that the simply connected planar domain is $R=[1,n]\times[0,1]$ where $n$ as usual is the number of threads. Given a boundary condition $H^{\partial}$ as in Definition~\ref{esqrfsdfhql}, for any $l\in\mathbb{N}^*$ big enough, we construct a very tall polygon $R_{l,H^{\partial}}$ tileable by lozenges as follows.

We first construct two piecewise linear paths $p_0$ and $p_1$. The path $p_0$ is a piecewise linear continuous path defined on $[0,n+1]$, which is a linear extension of $H^{\partial}(x,0)-H^{\partial}(0,0)$ on every interval $x\in[i,i+1]$, $i\in\{0,1,...,n\}$. Define analogously $p_1$ on $[0,n+1]$ as a piecewise linear extension of $H^{\partial}(x,1)-H^{\partial}(0,1)+l$. The paths
$$p_0,\ \ p_1,\ \ \{0\}\times[0,l],$$
$$\{n+1\}\times[H^{\partial}(n+1,0)-H^{\partial}(0,0),H^{\partial}(n+1,1)-H^{\partial}(0,1)+l]$$
enclose a region of $\mathbb{R}^2$ when $l$ is big enough so that $p_0$ and $p_1$ do not intersect. The paths $p_0$ and $p_1$ correspond
to the upper and lower boundary conditions of $R$, and we still need to remove some tiny triangles from this region so that it corresponds to the left and right boundary condition.

For any $j\in\mathbb{N}$, define $\Delta_j^0$ as the triangle defined by the three vertices
$$(0,j),\ (0,j+1),\ (1,j+\frac{1}{2})$$
and for any $j'\in\mathbb{N}$ define $\Delta_{j'}^1$ as the triangle defined by
$$(n+1,H^{\partial}(n+1,0)-H^{\partial}(0,0)+j'),\ \ (n+1,H^{\partial}(n+1,0)-H^{\partial}(0,0)+j'+1),$$ $$(n,H^{\partial}(n+1,0)-H^{\partial}(0,0)+j'+\frac{1}{2}).$$

Suppose that the jumps of $H^{\partial}(0,y)$ (resp. $H^{\partial}(n+1,y)$) are at $(0,y_k)$ (resp. $(n+1,y_{k'})$), we remove the triangles $\Delta_{\lfloor l y_k \rfloor}^0$ and $\Delta_{\lfloor l y_{k'} \rfloor}^1$ (when $l$ is large enough, these triangles are all different) from the region defined above, and we define $R_{l,H^{\partial}}$ as the new domain. A removed triangle is called a \emph{crack} on the left or on the right boundary of $R_{l,H^{\partial}}$. It is not hard to check that $R_{l,H^{\partial}}$ is tileable.

For some reason that will be clear later, we are particularly interested in the case where there are no cracks, \emph{i.e.} the function $H^{\partial}$ restricted to $i=0$ or on $i=n+1$ is constant. This domain is tileable in the following way: consider the case $l=0$, the region $R_{0,H^{\partial}}$ is tileable and only tileable by all $\lozengeconstant{}$. Now for $l>0$, the region $R_{l,H^{\partial}}\backslash R_{0,H^{\partial}}$ is enclosed by two pairs of parallel paths, and it is easy to see that this difference is tileable by $\lozengeincreasing{}$ and $\lozengedecreasing{}$.

Figure~\ref{R0nANDRln} gives an illustration of a bead model of $9$ threads and boundary condition $H^{\partial}$. On the left, the grey region is $R_{0,H^{\partial}}$, tiled in the only possible way. It is enclosed in a bigger polygon $R_{7,H^{\partial}}$, where on the right is a general tiling.  Readers can think of a pile of boxes in $\mathbb{R}^3$, and the height function $H$ is given by the projection of the pile on $\mathbb{R}^2$ in the direction $(1,1,0)$. The number of horizontal lozenges in a tiling is the projection in the direction $(0,0,1)$, thus independent of the exact pile of boxes and $l$ (the height of that pile).

\begin{figure}[H]
\centering
\includegraphics[width=0.7\textwidth]{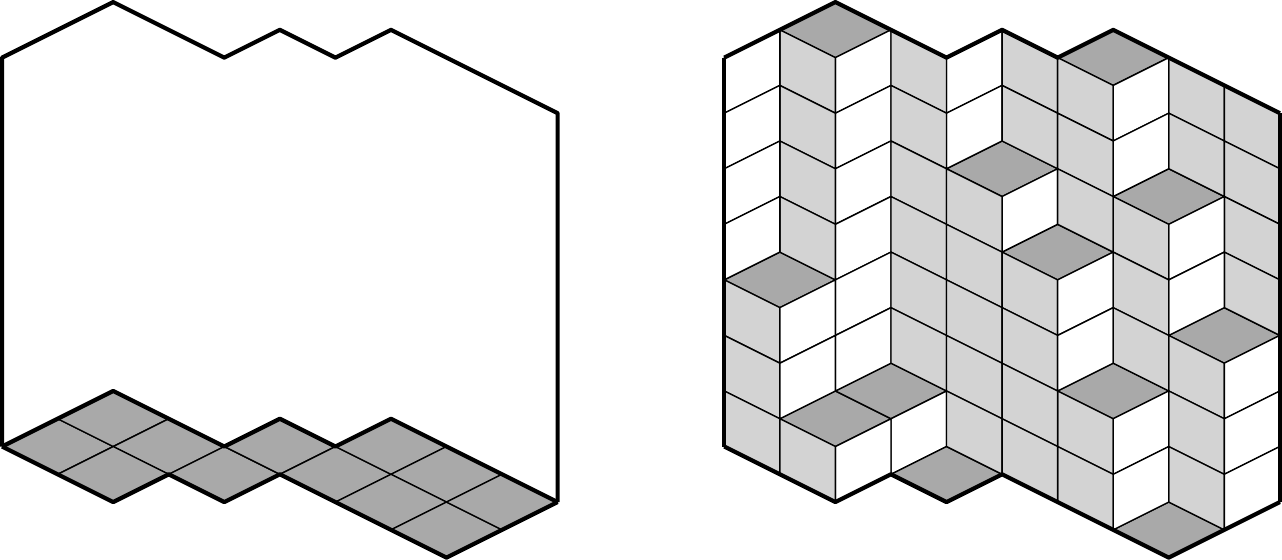}
\caption{Tiling $R_{0,H^{\partial}}$ and $R_{7,H^{\partial}}$.}\label{R0nANDRln}
\end{figure}

If we consider the uniform measure on the tilings, it is not hard to check that when $l\rightarrow\infty$, the joint Dirac measure of the positions of the horizontal lozenges $\lozengeconstant{}$ in a uniform tiling of $R_{l,H^{\partial}}$ and vertically normalized by $l$ converges weakly to that of the uniform bead measure with boundary condition $H^{\partial}$.

\vspace{0.5cm}

The torus is much simpler.  We consider $T_{l,n}$ as a torus of size $n\times l$ where $l$ is big enough. Its height change $(H_x,H_y)$ can take value in $$\left\{-\frac{n}{2}+1,-\frac{n}{2}+2,...,\frac{n}{2}-2,\frac{n}{2}-1\right\}\times\mathbb{N}^*\cup\left\{(\pm\frac{n}{2},0)\right\},$$
where $(\pm\frac{n}{2},0)$ correspond to the cases that there are only $\lozengeincreasing{}$ or $\lozengedecreasing{}$, so they should not be taken into consideration. If we fix $(H_x,H_y)$, then the number of $\lozengeconstant{}$ is fixed and equal to $n H_y$. When $l\rightarrow\infty$ the joint Dirac measure of the positions of $\lozengeconstant{}$ in a uniform tiling of $T_{l,n}$ and vertically normalized by $l$ converges weakly to that of the uniform bead measure with periodic condition $(H_x,H_y)$.


\section{Standard Young tableaux and bead model}\label{sect6.4.1map}

Throughout this paper we use the Russian convention of Young diagrams, skew diagrams and Young tableaux, see Figure \ref{qdsfqizfyhmlqrg}. Denote the number of boxes of a (skew) diagram $\lambda$ by $|\lambda|$. A column here means boxes with same horizontal position.

\begin{figure}[H]
\centering
\includegraphics[width=0.65\textwidth]{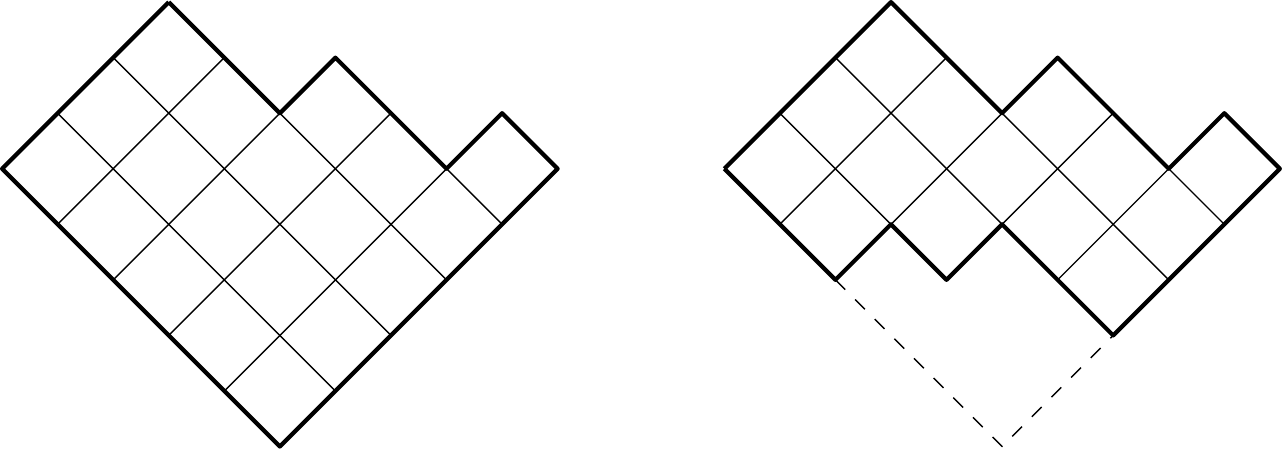}
\caption{A Young diagram and a skew diagram under the Russian convention.}\label{qdsfqizfyhmlqrg}
\end{figure}

In Figure~\ref{R0nANDRln}, $R_{0,H^{\partial}}$ is a skew Young diagram if we view every horizontal lozenge in the tiling of $R_{0,H^{\partial}}$ as a box. More generally, any (skew) Young diagram (of $n$ columns) corresponds to a bead model (with $n$ threads).
For any such bead model, if we take the domain as $R=[1,n]\times[0,1]$, the boundary function $H^{\partial}$ is constant if restricted on $i=0$ or $i=n+1$. The path $p_0$ (resp. $p_1$) constructed as in Section~\ref{sqdfhyzqmurei} is exactly the lower (resp. upper) boundary of the Young diagram, and every box of the diagram is encoded with a bead (so the number of beads is equal to $|\lambda|$).

For any bead configuration, let $y_{i,j}$ be the vertical coordinate of the $j^{th}$ bead on the $i^{th}$ thread. We sort them in a non-decreasing order:\label{sqdflfhriluhreiugtruetgrkjfjfdfv}
$$y_{i_1,j_1}\leq y_{i_2,j_2}\leq...\leq y_{i_{|\lambda|},{j_{|\lambda|}}}.$$

As under the uniform measure, the probability that any two coordinates coincide is equal to $0$, with probability $1$ we can rewrite the inequalities above as
\begin{eqnarray}\label{sdlfqkhsddmlfhdsrg}
y_{i_1,j_1}< y_{i_2,j_2}<...< y_{i_{|\lambda|},{j_{|\lambda|}}}.
\end{eqnarray}

For the given diagram $\lambda$, define $\mathcal{T}_{\lambda}$ as the set of standard tableaux of $\lambda$ and $\mathcal{B}_{\lambda}$ as the space of bead configurations with same constraint. Any inequality on the vertical coordinates of a pair of beads on neighboring threads interprets itself to be an inequality relation between the ranks of neighboring boxes, which is exactly the inequality relation of the neighboring entries in the definition of a standard Young tableau. So conditioned to that all vertical coordinates $y_{i,j}$ are different, if we define the following map $\mathcal{Y}$ as:
\begin{eqnarray*}
\mathcal{Y}: & \mathcal{B}_{\lambda} & \rightarrow \mathcal{T}_{\lambda},\\
& \mathbf{B} &\mapsto T,
\end{eqnarray*}
where $T=\mathcal{Y}(\mathbf{B})$ is a filling of $\lambda$ such that $T(i_k,j_k)=k$ for any $k\leq|\lambda|$ (define $T(i,j)$ as the number in the cell $(i,j)$ of $T$). Then for every configuration $\mathbf{B}\in\mathcal{B}_{\lambda}$, $T=\mathcal{Y}(\mathbf{B})$ is a tableau whose entries are all different and verify the constraint of a Young tableau, thus $T\in\mathcal{T}_{\lambda}$.

In short, the map $\mathcal{Y}$ just turns the continuous coordinates $y$ to its total rank among all the coordinates. If we take the uniform measure of the bead model, the measure induced by $\mathcal{Y}$ on the standard (skew) tableaux is the uniform measure. This fact is known in \cite{MR2733067} and \cite{MR2001148}. In fact, for any $T\in\mathcal{T}_{\lambda}$, the induced probability measure is by definition proportional to the Lebesgue measure of its preimage $\mathcal{Y}^{-1}(T)$, \emph{i.e.}, the volume of the simplex
$$0<y_{i_1,j_1}<y_{i_2,j_2}<...<y_{i_{|\lambda|},j_{|\lambda|}}<1,$$
which is always equal to $\frac{1}{|\lambda|!}$ for any $T$.

\section{Entropy of the bead model}\label{erdsfdhzilmhfq}

Our first task is to define the combinatorial entropy. Once defined, we will use the letter $S$ to denote it.

In a classical case (dimer model for example) where a random variable $X$ takes a countable number of possible different values (or states) with $p_i$ be the corresponding probability, the \emph{combinatorial entropy} is defined as
$$S(X)=\sum_i -p_i\ln p_i.$$
In particular, if every state has the same probability, then
$$S(X)=\ln Z,$$
where $Z$ is the number of states, known as the ``partition function".

One significant difference between the bead model and the dimer model is that rather than considering the ``number" of dimer configurations in a state, here we should consider the volume of similar bead configurations. Moreover, in practice we will adjust it by adding an additional term to let the entropy be of the good order. We give the definition here below.

\begin{definition}
Consider a bead model with fixed number of beads. Let $N$ be the number of beads and $n$ be that of threads. Consider a random bead configuration as a random vector $X$ taking values in ${[0,1]^{N}}$, where every component of $X$ is the vertical coordinates of the corresponding bead (in the toroidal case the coordinates are in the sense of modulo $1$).

For any point ${\mathbf{y}=(y_1,y_2,...,y_N)\in[0,1]^{N}}$, define $\rho(\mathbf{y})$ as the \emph{density of the bead measure} $\mathbb{P}$ at the point $\mathbf{y}$ with respect to the Lebesgue measure of $[0,1]^{N}$ whenever it exists, \emph{i.e.},
$$\rho(\mathbf{y})=\lim_{\varepsilon\rightarrow 0}\frac{\mathbb{P}(X\in\prod_{i=1}^{N}[y_i-\varepsilon,y_i+\varepsilon])}{(2\varepsilon)^{N}}$$
whenever this limit exists.
\end{definition}

If we consider the uniform bead measure, and if we define $V$ as the $N$-dimensional Lebesgue measure of the convex set of coordinates, then
\begin{align*}
\rho(\mathbf{y})=
\begin{cases}
\frac{1}{V}&\text{ if }h\text{ is an inner point of the convex set of the admissible coordinates,}\\
0&\text{ otherwise,}
\end{cases}
\end{align*}
and the undefined points are negligible.

We use the same letter $S$ to denote the adjusted combinatorial entropy of the bead model.

\begin{definition}\label{ezrlkhmaqdsfqe}
If the density $\rho$ is well defined almost everywhere, then for the bead model with a fixed boundary condition or periodic condition, we define the \emph{(adjusted) combinatorial entropy} $S$ associated to the random variable $X$ of the bead model as
$$S(X)=\int_{[0,1]^{N}}-\rho(\mathbf{y})\ln\rho(\mathbf{y})dy_1...dy_N+N\ln n,$$
where $N$ is the number of beads and $n$ is the number of threads.
\end{definition}

The term $N\ln n$ may seem not natural, but soon we will see that this term helps to adjust the entropy so that it is of a proper order if we consider a sequence of bead models where $n\rightarrow\infty$ and $N$ is of order $n^2$.

In particular, if we consider the uniform measure, we have
$$S(X)=\ln V+N\ln n.$$

The following lemma is a general result for entropies.
\begin{lemma}\label{lemmaSdecomposition}
Suppose $E=\{E_1,E_2,...\}$ is a countable partition of the state space, and $I_E$ is a random variable that tells $X$ is in which $E_i$, and $X_i$ is the variable equipped with the conditional law of $X$ restricted on $E_i$. We have
\begin{eqnarray}\label{dsfmeqrzlruihrgqedq}
S(X)=S(I_E)+\sum_i\mathbb{P}(X\in E_i)S(X_i).
\end{eqnarray}
\end{lemma}

The proof is straightforward.

We want to remark that the decomposition~(\ref{dsfmeqrzlruihrgqedq}) allows us to define the entropy $S$ for a union of conditions that not necessarily have the same number of beads once we have defined the probability of taking different number of beads $N$:
\begin{definition}\label{sqdsdfkufhrulgerkzzer}
For a random bead configuration $X$ that with probability $p_i$ to be in the state of $N_i$ beads, define
\begin{eqnarray*}
S(X)=-\sum_i p_i\ln p_i +\sum_i p_i S(X_i),
\end{eqnarray*}
where $X_i$ is the random configuration equipped with the induced probability measure conditioning to have $N_i$ beads.
\end{definition}

The following proposition proves that the entropy of the bead model under the uniform measure is the limit of the entropies of the corresponding lozenge tiling models. This discrete approximation is useful in the remaining part of this paper.


Consider a bead model with $n$ threads and $N$ beads, with a fixed boundary condition $H^{\partial}$ or a given periodic condition ${(H_x,H_y)}$. Consider the corresponding lozenge tiling model, where for $l$ sufficiently large we tile a simply connected domain $R_{l,H^{\partial}}$ or a toroidal region $T_{l,n}$. In each of the cases, we define $Z_{l,n}$ as the partition function of the lozenge tilings of the region, and $V$ as the volume of the convex set in ${[0,1]^N}$ or $(\mathbb{R}/\mathbb{Z})^N$ formed by the vertical coordinates of the beads.

\begin{prop}\label{beadanddimerpartitionfunction}
For either a fixed boundary condition $H^{\partial}$ or a given periodic condition ${(H_x,H_y)}$, we have the following relation between $Z_{l,n}$ and $V$:
\begin{eqnarray}\label{qsdifhmzeeraq}
\ln V=\lim_{l\rightarrow\infty}\big(\ln Z_{l,n}-N\ln l\big).
\end{eqnarray}
In particular, if we let $l=mn$, then for fixed $n$ and $N$, $l\rightarrow\infty$ is equivalent to $m\rightarrow\infty$, and we have that the entropy of the bead model is equal to
\begin{eqnarray}\label{limitofpartition}
S(X)=\lim_{m\rightarrow\infty}\big(\ln Z_{mn,n}-N\ln m\big).
\end{eqnarray}
\end{prop}
\begin{proof}
Consider the convex set of the vertical coordinates of the beads. For any $l\in\mathbb{N}^*$ big enough, $Z_{l,n}$ is approximately equal to the number of points on the lattice $\left(\frac{1}{l}\mathbb{Z}\right)^2$ inside the convex set, so we have
\begin{eqnarray*}
\lim_{l\rightarrow\infty}\frac{Z_{l,n}}{l^N}=V,
\end{eqnarray*}
and by taking logarithm we get Equation~(\ref{qsdifhmzeeraq}) in the proposition. Replacing $l$ by $mn$, we obtain Equation~(\ref{limitofpartition}).
\qed
\end{proof}

We now explain why the combinatorial entropy $S$ defined in Definition~\ref{ezrlkhmaqdsfqe} is adjusted by $N\ln n$, and why we use the substitution of $l$ by $mn$ in Proposition~\ref{beadanddimerpartitionfunction}. As mentioned, we are interested in the asymptotic behavior of the bead model, \emph{i.e.} in the limit $n\rightarrow\infty$. If the boundary function $H^{\partial}$ of the bead model has an asymptotic limit when $n\rightarrow\infty$, then $N$ is asymptotically proportional to $n^2$. We write ${H^{\partial}=H^{\partial}(n)}$ and ${N=N(n)}$ to emphasize their dependances on $n$.

For every given $m$ and $n$, consider $R_{mn,H^{\partial}(n)}$ as in Section~\ref{sqdfhyzqmurei}. Since we fix the asymptotic shape of $H^{\partial}(n)$ in the remaining part of this paper, from now on we simply write $R_{mn,n}$ instead of $R_{mn,H^{\partial}(n)}$ to simplify the notation. Consider
\begin{eqnarray}\label{sqdlifhyzitz}
\frac{\ln Z_{mn,n}-N(n)\ln m}{n^2}.
\end{eqnarray}
If we fix $m$ and let $n\rightarrow\infty$ (we pretend to forget that $m$ should be chosen large enough depending on $n$), the boundary condition of $R_{mn,n}$ has an asymptotic limit, so by \cite{CKP01,KOS06},~(\ref{sqdlifhyzitz}) converges when $m$ is fixed and $n\rightarrow\infty$. Meanwhile, Proposition~\ref{beadanddimerpartitionfunction} proves that~(\ref{sqdlifhyzitz}) converges to $\frac{S(X)}{n^2}$ when $m\rightarrow\infty$ for fixed $n$.

For this reason, it is natural to ask the following questions:
\begin{itemize}
\item In~(\ref{sqdlifhyzitz}), can we take the limit $m\rightarrow\infty$ first and then the limit ${n\rightarrow\infty}$?
\item If this limit exists, does it have a good order?
\item Can we exchange the order of the limits in $m$ and in $n$?
\end{itemize}

We give a positive answer to each of them in Sections~\ref{sectionProveVP} and~\ref{solution}, but before that we want to give some discussion.

We first give an intuitive explanation for the first and the second questions in a specific case. When the boundary condition of the bead model corresponds to a square Young diagram (Section~\ref{sect6.4.1map}), then $N=\frac{(n+1)^2}{4}$. The volume $V$ is equal to the number of possible total ranking of the vertical coordinates (which is the number of standard Young tableaux for a $\frac{n+1}{2}\times\frac{n+1}{2}$ square diagram) times the volume of the convex set of the coordinates totally ranked (which is equal to $\frac{1}{N!}$). Thus, by the hook formula (\cite{FRT}), we have
\begin{eqnarray*}
\frac{S(X)}{n^2}
&=&\frac{\ln V+N(n)\ln n}{n^2}\\
&=&\frac{1}{n^2}\left(\ln\left(\frac{N(n)!}{\prod_{1\leq i,j\leq\frac{n+1}{2}}(i+j+1)}\frac{1}{N(n)!}\right)+ N(n)\ln n\right)\\
&=&\frac{1}{n^2}\left(\sum_{1\leq i,j\leq\frac{n+1}{2}}\ln\left(\frac{n}{i+j+1}\right)\right)\\
&\simeq&\frac{1}{4}\iint_{x,y\in[0,1]}\ln\frac{1}{x+y}dxdy.
\end{eqnarray*}
This double integral also appears in \cite{PR}, which studies the limit shape of a random square Young tableaux using hook formula.

For the third question, we show why a priori it is not obvious that we can exchange the order of the limits in $n$ and in $m$. In fact, in the proof of Proposition~\ref{beadanddimerpartitionfunction} we use an approximation of the volume of a convex set of dimension ${N(n)=O(n^2)}$ by a mesh of size $\frac{1}{mn}$, and this approximation is not uniform in $m$ and $n$ for whatever type of convex set. For example, if we consider a simplex
$$0\leq y_1\leq y_2...\leq y_{N(n)}\leq 1,$$
the volume of this simplex is $\frac{1}{N(n)!}$, while the number of lattice points of a $(\mathbb{Z}/mn)^{N(n)}$ mesh inside this simplex is equal to the number of choosing $N(n)+1$ non-negative ordered integers that sum to $mn$, which is equal to $\binom{mn+N(n)}{N(n)}$. Thus, the approximation has a relative error of order
\begin{eqnarray*}
&\ &\binom{mn+N(n)}{N(n)}\left(\frac{1}{mn}\right)^{N(n)}\left(\frac{1}{N(n)!}\right)^{-1}-1\\
&=&\left(1+\frac{1}{mn}\right)\left(1+\frac{2}{mn}\right)...\left(1+\frac{N(n)}{mn}\right)-1.
\end{eqnarray*}
We see that as in Proposition~\ref{beadanddimerpartitionfunction}, for fixed $n$, the relative error tends to $0$ when $m\rightarrow\infty$, while this is not uniform in $n$.

\section{Free energy and local entropy function of the bead model}\label{sect5.3}
\sectionmark{Free energy and local entropy function}

In this section, we consider a sequence of toroidal bead models with given asymptotic periodic condition (which means that the height change $(H_x,H_y)$ is proportional to the number of threads $n$), and we calculate its entropy when the size of the torus tends to infinity. In the computation we mainly use the discrete approximation given by Proposition~\ref{beadanddimerpartitionfunction}.

In \cite{BeadModel}, the author defines a family of ergodic Gibbs bead measures which are limits of the ergodic Gibbs measures of the dimer model on the hexagonal lattice when some weights degenerate. We begin by considering this parameterized weight setting of the dimer model, then apply a Legendre transform on the adjusted partition function of the dimer model to obtain the local entropy function $ent$. The proof of that $ent$ is equal to the normalized combinatorial entropy $S$ is postponed to Theorem~\ref{thm-ent} of Section~\ref{sectionProveVP}.

\subsection{Free energy}

Throughout this section, suppose that a lozenge $\lozengeconstant{}$ has weight $a$, $\lozengedecreasing{}$ has weight $b$ and $\lozengeincreasing{}$ has weight $c$, see Figure~\ref{ThreeLozenges}.
\begin{figure}[H]
\centering
\includegraphics[width=0.2\textwidth]{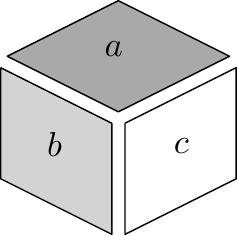}
\caption{The lozenges respectively weighted $a$, $b$ and $c$.}\label{ThreeLozenges}
\end{figure}

Consider a fundamental domain as Figure~\ref{fundom}. The characteristic polynomial is
$$a^2z-(b+cw)^2/w.$$
\begin{figure}[H]
\centering
\includegraphics[width=0.2\textwidth]{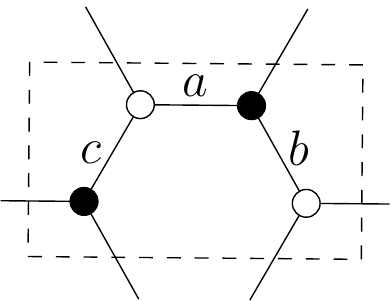}
\caption{The fundamental domain.}\label{fundom}
\end{figure}
The advantage of taking this domain is that it has an obvious horizontal-vertical decomposition. We consider $\mathcal{G}_{mn,n}$ as a toroidal graph which is $mn\times n$ this fundamental domain. The dimer model on $\mathcal{G}_{mn,n}$ corresponds to a toroidal lozenge tiling model defined in Section~\ref{sqdfhyzqmurei} where we use the substitution of $l$ by $mn$. But pay attention, $\mathcal{G}_{mn,n}$ corresponds to $T_{mn,2n}$ rather than $T_{mn,n}$.

We take two parameters $\alpha\in\mathbb{R}^+$, $\gamma\in ]-1,1[$, and let $a=\alpha/m$, $b=e^{\alpha\gamma /m}$, $c=1$. The author of \cite{BeadModel} proves that the ergodic Gibbs dimer measure under such setting, which is to first take $n\rightarrow\infty$ for a dimer model on $\mathcal{G}_{mn,n}$, converges to an ergodic Gibbs measure on the configurations of the beads on threads when $m\rightarrow\infty$, and the limiting measure is parameterized with respect to $\alpha$ and $\gamma$. We will take the reverse order, \emph{i.e.}, we first take $m\rightarrow\infty$ and then $n\rightarrow\infty$.

Denote the dimer partition function of $\mathcal{G}_{mn,n}$ by $Z_{mn,n}(\alpha,\gamma)$. The set of dimer configurations is denoted by $\mathcal{M}=\mathcal{M}(\mathcal{G}_{mn,n})$. Let $N_a$ (resp. $N_b$ and $N_c$) be the number of edges with weight $a$ (resp. $b$ and $c$), then the partition function is
\begin{eqnarray*}
Z_{mn,n}(\alpha,\gamma)=\sum_{M\in\mathcal{M}}a^{N_a(M)}b^{N_b(M)}c^{N_c(M)}=\sum_{M\in\mathcal{M}}(\alpha/m)^{N_a(M)}e^{\alpha\gamma N_b(M)/m}.
\end{eqnarray*}

To simplify the notation we denote the weight of a configuration $(\alpha /m)^{N_a(M)}e^{\alpha\gamma N_b(M)/m}$ by $w(M)$.

Here the order of $\ln Z_{mn,n}(\alpha,\gamma)$ is $n^2$ (while for fixed $a$, $b$ and $c$ the logarithm of the partition function should be of order $mn^2$). In fact, if we differentiate $\ln Z_{mn,n}(\alpha,\gamma)$ with respect to $\gamma$ or $\alpha$, we get:
\begin{eqnarray*}
\frac{\partial\ln Z_{mn,n}(\alpha,\gamma)}{\partial\gamma}&=&\frac{1}{Z_{mn,n}(\alpha,\gamma)}\frac{\partial Z_{mn,n}(\alpha,\gamma)}{\partial\gamma}=\frac{\sum_{M\in\mathcal{M}} \alpha N_b(M)w(M)}{m\sum_{M\in\mathcal{M}} w(M)}
=\frac{\alpha \mathbb{E}[N_b]}{m},\\
\frac{\partial\ln Z_{mn,n}(\alpha,\gamma)}{\partial\alpha}&=&\frac{1}{Z_{mn,n}(\alpha,\gamma)}\frac{\partial Z_{mn,n}(\alpha,\gamma)}{\partial\alpha}=\frac{\sum_{M\in\mathcal{M}} \big(\frac{1}{\alpha}N_a(M)+\gamma /m N_b(M)\big)w(M)}{\sum_{M\in\mathcal{M}} w(M)}\\
&=&\frac{1}{\alpha}\mathbb{E}[N_a]+\frac{\gamma}{m}\mathbb{E}[N_b].
\end{eqnarray*}

When divided by $n^2$, we get
\begin{eqnarray}
\frac{\partial}{\partial\gamma}\frac{\ln Z_{mn,n}(\alpha,\gamma)}{n^2}&=&\alpha \frac{\mathbb{E}[N_b]}{mn^2}.\label{1}\\
\frac{\partial}{\partial\alpha}\frac{\ln Z_{mn,n}(\alpha,\gamma)}{n^2}&=&\frac{1}{\alpha} \frac{\mathbb{E}[N_a]}{n^2}+\gamma\frac{\mathbb{E}[N_b]}{mn^2}.\label{2}
\end{eqnarray}

Since we expect that the number of edges $a$ is of order $n^2$ and that of edges $b$ and $c$ is of order $mn^2$, Equations~(\ref{1}) and~(\ref{2}) show that $\ln Z_{mn,n}(\alpha,\gamma)$ normalized by $n^2$ is of the good order. The aim of this section is to compute the limit of $\frac{\ln Z_{mn,n}(\alpha,\gamma)}{n^2}$, where we first take ${m\rightarrow\infty}$ and then ${n\rightarrow\infty}$.

\begin{prop}\label{prop1}
For any given $n\in\mathbb{N}^*$, when $m\rightarrow\infty$, $Z_{mn,n}(\alpha,\gamma)$ converges.
\end{prop}

Assuming Proposition \ref{prop1}, we define the partition function of the bead model of the torus of size $n$ and of parameters $\alpha$ and $\gamma$ as
$$\widetilde{Z}_{n}(\alpha,\gamma)=\lim_{m\rightarrow\infty}Z_{mn,n}(\alpha,\gamma).$$

\begin{prop}\label{prop2}
When $n\rightarrow\infty$, $\ln \widetilde{Z}_n(\alpha,\gamma)$ is of order $n^2$, and
$$\lim_{n\rightarrow\infty} \frac{\ln \widetilde{Z}_n(\alpha,\gamma)}{n^2}=
\frac{2\alpha}{\pi}\big(\gamma\arccos(-\gamma)+\sqrt{1-\gamma^2}\big).$$
\end{prop}

This limit is called the free energy of the bead model with parameters $\alpha$ and $\gamma$ per fundamental domain. This value depends on the choice of fundamental domain. The proof of Propositions~\ref{prop1} and~\ref{prop2} is put in Appendix A.

\vspace{0.5cm}

As a corollary, by~(\ref{1}) and~(\ref{2}), we get an estimate of the numbers of the different types of edges:
$$\mathbb{E}\left(\frac{N_a}{n^2}\right)=\frac{2\alpha}{\pi}\sqrt{1-\gamma^2},\ \mathbb{E}\left(\frac{N_b}{mn^2}\right)=\frac{2}{\pi}\arccos(-\gamma),\ \mathbb{E}\left(\frac{N_c}{mn^2}\right)=\frac{2}{\pi}\arccos(\gamma).$$

The following proposition allows us to take the limit $m,n\rightarrow\infty$ in an arbitrary way.

\begin{prop} \label{twolimits}
The limit $m\rightarrow\infty$ and the limit $n\rightarrow\infty$ can be exchanged when calculating the partition function, i.e.,
$$\lim_{m\rightarrow\infty}\lim_{n\rightarrow\infty} \frac{\ln Z_{mn,n}(\alpha,\gamma)}{n^2}=
\lim_{n\rightarrow\infty}\frac{\ln \widetilde{Z}_{n}(\alpha,\gamma)}{n^2}.$$
\end{prop}
\begin{proof}
This can be proved via direct computation.
\qed
\end{proof}

\subsection{Surface tension, local entropy function}\label{surfacetensionlocalentropyfunction}

We now turn to the uniform measure on the periodic bead model with given height change, which corresponds to the uniform bead measure with given periodic boundary condition introduced in Section~\ref{qseliyrzame}. Take the definition of height function of Section~\ref{qesfihjzemrau}, and let $N_a$, $N_b$ and $N_c$ respectively be the number of $\lozengeconstant{}$, $\lozengedecreasing{}$ and $\lozengeincreasing{}$ in a tiling of $T_{mn,n}$, then the height change $(H_x,H_y)$ is given by
\begin{eqnarray*}
N_a=n H_y ,\ N_b-N_c=-2mn H_x ,\ N_a+N_b+N_c=mn^2.
\end{eqnarray*}

Recall that $T_{mn,n}$ corresponds to $\mathcal{G}_{mn,n/2}$ (without loss of generality we suppose that $n$ is even). If we fix the height change $(H_x,H_y)$ and define $Z_{mn,n}^{H_x,H_y}$ as the partition function for a uniform tiling of $T_{mn,n}$, then the partition function of $\mathcal{G}_{mn,n/2}$ with parameters $\alpha$, $\gamma$ is given by
\begin{eqnarray}
Z_{mn,n/2}(\alpha,\gamma)&=&\sum_{H_x,H_y}Z_{mn,n}^{H_x,H_y}\left(\frac{1}{m}\right)^{n H_y}
e^{n^2\big(\ln\alpha\frac{H_y}{n}+\alpha\gamma(-\frac{H_x}{n}+\frac{1}{2})+o(1)\big)},
\label{relation}
\end{eqnarray}
where $o(1)$ is in $m$. Recall that in Proposition~\ref{beadanddimerpartitionfunction} we proved that for given $(H_x,H_y)$, the term
$$Z_{mn,n}^{H_x,H_y}\left(\frac{1}{m}\right)^{n H_y}$$
converges when $m\rightarrow\infty$. Later in Section~\ref{sectionProveVP} we prove that moreover the logarithm of this limit value divided by $n^2$ converges when $n\rightarrow\infty$ and the limits depends and is continuous on the average slope $(\frac{H_x}{n},\frac{H_y}{n})$, which by construction should be included in $[-\frac{1}{n}]\times[0,+\infty]$. The case where $\frac{H_y}{n}\rightarrow\infty$ (\emph{i.e.} $H_y$ is beyond the order $O(n)$) is possible, but this has a negligible contribution because otherwise the right hand side of~(\ref{relation}) explodes if we take an $\alpha$ bigger than $1$, which is not the case.

Following the idea of \cite{KOS06}, for $(s,t)\in[-\frac{1}{2},\frac{1}{2}]\times[0,\infty[$, and for any $m,n$, consider the lozenge tilings of $T_{mn,n}$ of height change
$$(H_x,H_y)=(\lfloor ns\rfloor,\lfloor nt\rfloor),$$
and by the discussion above we can define the surface tension as
$$\sigma(s,t)=-\lim_{n\rightarrow\infty}\lim_{m\rightarrow\infty}\left(
\frac{\ln Z^{\lfloor ns\rfloor,\lfloor nt\rfloor}_{mn,n}}{n^2}-t\ln m\right).$$

If we consider a fundamental domain as in Figure~\ref{fundomp} (which is half of that of Figure~\ref{fundom}), and consider the free energy per fundamental domain which is equal to
$$F(\alpha,\gamma)=\lim_{n\rightarrow\infty}\lim_{m\rightarrow\infty}\frac{1}{2}\frac{\ln Z_{mn,n}(\alpha,\gamma)}{n^2}= \frac{1}{\pi}\big(\alpha\gamma\arccos(-\gamma)+\alpha\sqrt{1-\gamma^2}\big),$$
(it is half of the limit in Proposition~\ref{prop2}), then equation~(\ref{relation}) implies that
$$F(\alpha,\gamma)=\max_{s,t}\big(-\sigma(s,t)+\ln\alpha t+\alpha\gamma(\frac{1}{2}-s)\big).$$

\begin{figure}[H]
\centering
\includegraphics[width=0.2\textwidth]{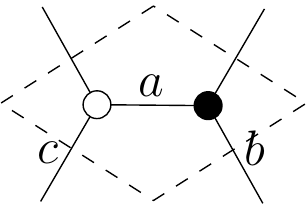}
\caption{Another fundamental domain.}\label{fundomp}
\end{figure}

Let $A=\ln\alpha$, $B=-\alpha\gamma$, then we get that
\begin{eqnarray}
\alpha=e^A,\ \gamma=-\frac{B}{e^A}.\label{changementvariable}
\end{eqnarray}
Replace $\alpha$ and $\gamma$ by $A$ and $B$, and define
$$\widetilde{F}(A,B)=F(\alpha,\gamma)-\frac{\alpha\gamma}{2}=\frac{1}{\pi}\big(-B\arccos(\frac{B}{e^A})+\sqrt{e^{2A}-B^2}\big)+\frac{B}{2}.$$
Its Hessian matrix is positive-definite so $\widetilde{F}$ is strictly convex. Since the $\sigma$ as a limit of strictly convex function (the surface tension in the dimer model) is convex, $\widetilde{F}$ and $\sigma(s,t)$ are Legendre duals, so we have
\begin{eqnarray*}\label{locallabelx}
\sigma(s,t)&=&\max_{A,B}\big(-\widetilde{F}(A,B)+At+Bs \big)\\
&=&-\left(1+\ln\left(\frac{\cos(\pi s)}{\pi t}\right)\right)t.
\end{eqnarray*}

Define the local entropy $ent$ of the bead model as $-\sigma$. More precisely,

\begin{definition}\label{EntDef}
For any slope $(s,t)\in[-\frac{1}{2},\frac{1}{2}]\times[0,\infty]$, define the \emph{local entropy function} $ent(s,t)$ of the bead model as the following function:
\begin{eqnarray*}
ent(s,t)=
\begin{cases}
0 &\text{ if }t=0,\\
-\infty &\text{ if }s=\pm\frac{1}{2},\ t\neq 0,\\
\left(1+\ln\left(\frac{\cos(\pi s)}{\pi t}\right)\right)t &\text{ otherwise}.
\end{cases}
\end{eqnarray*}
\end{definition}

The function $ent(s,t)$ is concave in $s$ and $t$ and strictly concave on any domain where $t>0$.
\begin{figure}[H]
\centering
\subfigure[$ent(s,t)$ as function of slope $(s,t)$.\label{entFig}]{
\includegraphics[clip, trim=6.5cm 10cm 6.5cm 10cm, width=0.43\textwidth]{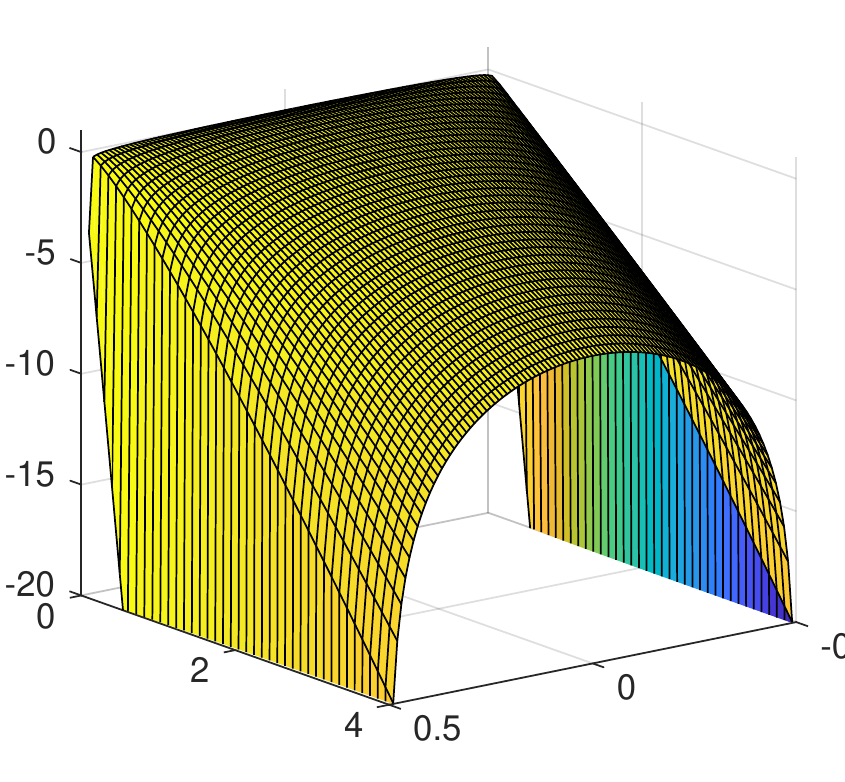}}
\hspace{0.1\textwidth}
\subfigure[Contour lines of $ent$.]{
\includegraphics[width=0.32\textwidth]{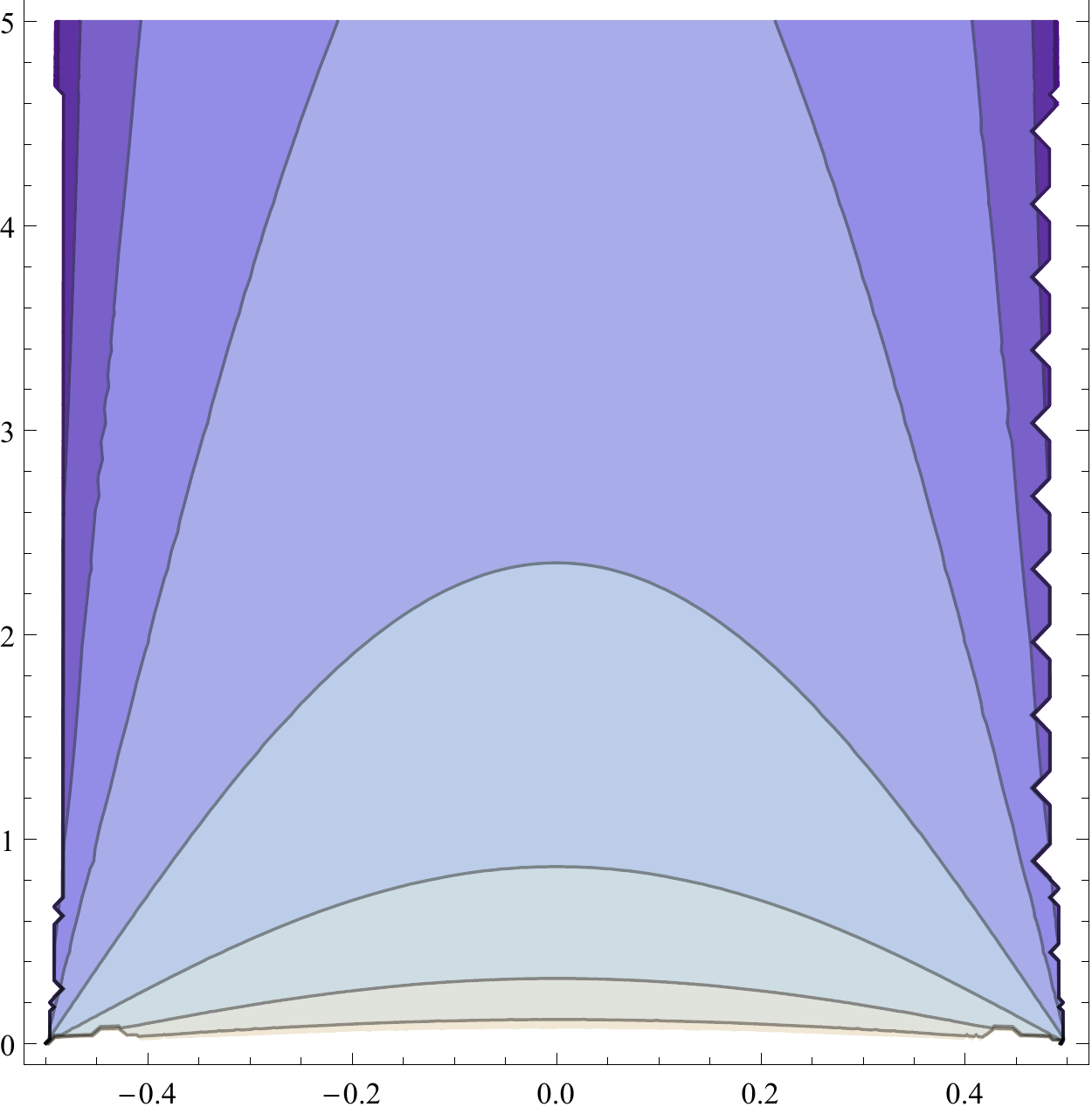}}
\caption{The local entropy function $ent$.}
\end{figure}

In Section \ref{sect5.4}, we consider the problem of maximizing (roughly speaking) the integral of $ent$ over $D$ where where $t$ will be taken to be $\frac{\partial h}{\partial y}$ and $s=\frac{\partial h}{\partial x}$. For later use, we here take a look at the expression of the local entropy $ent(s,t)$ (Definition~\ref{EntDef}):
$$\left(1+\ln\left(\frac{\cos(\pi s)}{\pi t}\right)\right)t,$$

The integral of any constant times $\frac{\partial h}{\partial y}$ on $D$ is fixed by the boundary condition. Also, $-u\ln u\leq\frac{1}{e}$. So maximizing the integral of $ent$ is equivalent to maximizing the integral of
$$\ln\left(\frac{\cos(\pi s)}{t}\right)t-\frac{1}{e},$$
which is non-positive and we call it the \emph{active part} of $ent(s,t)$. In Section~\ref{sect5.4}, without loss of generality, sometimes we consider the active part and assume that $ent$ is bounded above by $0$ for the sake of simplification.

Meanwhile, it is good to remark that $ent(s,t)$ tends to $-\infty$ when $t$ to infinity or when $s$ tends to $\pm{\frac{1}{2}}$ while $t$ not tends to $0$ fast enough.

\vspace{0.5cm}

We end this section by an analog of Proposition~\ref{twolimits} for entropy. Denote by $ent^{\diamond}$ the entropy of the dimer model on the honeycomb lattice.
\begin{prop}\label{entropyexchange}
For any $(s,t)\in[-\frac{1}{2},\frac{1}{2}]\times[0,+\infty[$, the entropy function of the bead model $ent(s,t)$ is the following limit of that of the dimer model on the honeycomb lattice:
\begin{eqnarray}\label{entropyexchangeequation}
ent(s,t)=\lim_{m\rightarrow\infty}m\ ent^{\diamond}(s,t/m)-\ln m t.
\end{eqnarray}
Also, we have the following properties concerning the convergence:\\
\emph{(a)} for any compact set of possible slopes that doesn't contain points where $s=\pm\frac{1}{2}$, the above convergence~(\ref{entropyexchangeequation}) is uniform.\\
\emph{(b)} for any $\varepsilon>0$, there exists $\delta<0$ and $M\in\mathbb{N}^*$ such that for all possible slopes $(s,t)$ such that $t\leq\delta$ and for all $m\geq M$, we have
$$m\ ent^{\diamond}(s,t/m)-\ln m t<\varepsilon.$$
\end{prop}

Note that in (b) we just claim an arbitrarily small upper bound and the lower bound is in fact $-\infty$.

\begin{proof}
By \cite{CKP01,LN_K}, the dimer entropy of slope $(s,t)$ is
$$ent^{\diamond}(s,t)=\frac{1}{\pi}\big(\mathcal{L}(\pi t)+\mathcal{L}(\pi(\frac{1}{2}-s-\frac{t}{2}))+\mathcal{L}(\pi(\frac{1}{2}+s-\frac{t}{2}))\big),$$
where $\mathcal{L}$ is the Lobachevsky function defined by
$$\mathcal{L}(\theta)=-\int^{\theta}_0 \ln|2\sin t|dt.$$
So we have
\begin{eqnarray}\label{clkjdssfsdqs}
m\ ent^{\diamond}(s,\frac{t}{m})=-\frac{m}{\pi}\int_0^{\pi\frac{t}{m}}\ln|2\sin t|dt
+\frac{m}{\pi}\int_{\pi(\frac{1}{2}-s-\frac{t}{2m})}^{\pi(\frac{1}{2}-s+\frac{t}{2m})}\ln|2\sin t|dt.
\end{eqnarray}

The first term of~(\ref{clkjdssfsdqs}) is equal to
$$\big(1+\ln m-\ln2-\ln\pi-\ln t\big)t+o(1)$$
where the $o(1)$ tends to $0$ when $m\rightarrow\infty$ and this is uniform on any set where $t$ is bounded.

For the second term of~(\ref{clkjdssfsdqs}), for any $(s,t)$, when $m$ is big this is equal to
\begin{eqnarray}\label{dfsqfezrgt}
\ln|2\sin (\pi(\frac{1}{2}-s))|t+o(1),
\end{eqnarray}
so we have proved the pointwise convergence in the lemma:
\begin{eqnarray*}
m\ ent^{\diamond}(s,t/m)-\ln m t=\left(1+\ln\left(\frac{\cos(\pi s)}{\pi t}\right)\right)t+o(1)=ent(s,t)+o(1).
\end{eqnarray*}

Clearly, the convergence~(\ref{dfsqfezrgt}) is uniform on any compact set of slopes that excludes the points where $s=\pm\frac{1}{2}$, which finish the proof of (a). The convergence is not uniform on a bounded set containing $(\pm\frac{1}{2},0)$, as for all $m$ the function $ent^{\diamond}_m$ is continuous on any possible point of slopes, while $ent$ is not continuous at $(\pm\frac{1}{2},0)$.

To prove (b), we have
\begin{align*}
&m\ ent^{\diamond}(s,t/m)-\ln m t\\
=&\left[-\frac{m}{\pi}\int_0^{\pi\frac{t}{m}}\ln|2\sin t|dt+(-\ln m +\ln 2)t\right]
+\frac{m}{\pi}\int_{\pi(\frac{1}{2}-s-\frac{t}{2m})}^{\pi(\frac{1}{2}-s+\frac{t}{2m})}\ln|\sin t|dt.
\end{align*}

For any $\delta>0$, the term in the bracket converges uniformly to $(1-\ln\pi-\ln t)t$ on the set $t\leq \delta$ when $m\rightarrow\infty$ and $(1-\ln\pi-\ln t)t$ converges to $0$ when $t\rightarrow 0$, so it suffices to choose a $\delta$ small enough and $M$ large enough so that for all $m>M$ this term is less than $\varepsilon$. Meanwhile, the second term is always negative. Thus we have finished the proof.
\qed
\end{proof}

\section{Entropy-maximizing problem}\label{sect5.4}

Our main aim is to establish a variational principle for the bead model as in \cite{CKP01}. This mainly consists of three parts: giving an entropy function, proving that there exists a unique maximizer and proving that there is a large deviation type behavior around that maximizer. In this section we focus on the first two parts, \emph{i.e.} raise a functional $Ent$ and prove that there exists a unique maximizer of it.

Since the bead model is just some kind of limit of that of the dimer, there are a lot of similarities between our case and that in \cite{CKP01}. It is natural to think of defining a global entropy function $Ent(h)$ as the integral of $ent\circ\nabla h$ on some domain with given boundary condition, where $h$ is the normalized height function. However, some delicate differences make the proof in the case of the bead model not a trivial and direct corollary of the dimer model. As we will see, the most remarkable difference is the unboundness of $ent$.

We consider a bead model normalized into the unit square $D=[0,1]\times[0,1]$, define a normalized height function $h$ and define the space of admissible functions $\mathcal{H}$ in Section~\ref{eoizruzeauoh}. In Section~\ref{sect5.4part1} we define the functional $Ent$ on the admissible functions, and in Section~\ref{sect5.4part2} we prove that there is a unique admissible function that maximizes the entropy $Ent$.

\subsection{Bead model normalized into unit square}\label{eoizruzeauoh}

We normalize the bead configuration into the unit square $D=[0,1]\times[0,1]$. Consider a bead model with $n$ threads as in Section~\ref{The bead model} but take the threads as
$$\left\{\left(x=\frac{i-1}{n-1},y\right):i=1,2,...,n,\ y\in[0,1]\right\},$$
and we normalize the bead height function $H$ defined in Section~\ref{qesfihjzemrau} by $n-1$. Moreover, we extend this function to the whole of $D$ in a piecewise linear way.

\begin{definition}\label{heightfunctionbeadmodel}
Given a bead configuration $\mathbf{B}$ on threads $x\in\left\{\frac{i-1}{n-1},i=1,2,...,n\right\}$, ${y\in[0,1]}$, the \emph{normalized height function} $h=h^{\mathbf{B}}$ (again for convenience we omit $\mathbf{B}$) is defined as
$$h(x,y)=\frac{1}{n-1}H\big((n-1)x+1,y\big)$$
along every thread
$$\left\{\frac{i-1}{n-1},i=1,2,...,n\right\}\times[0,1],$$
then extended to the whole unit square $D$ in the following way: for every $y\in[0,1]$, the height function $x\mapsto h(x,y)$ viewed as a function of $x$ is taken to be the piecewise linear extension of $h(\frac{i-1}{n-1},y)$.
\end{definition}

For a toroidal bead model, we can also define a normalized multivalued height function $h$ on $\left(\mathbb{R}/\mathbb{Z}\right)^2$ analogously to Definition~\ref{heightfunctionbeadmodel}.

\vspace{0.5cm}

From now on, when we speak of the bead model with $n$ threads defined on the unit square $D$, we mean a normalized bead model as above, with normalized height function extended to $D$.

Clearly the function $h$'s horizontal partial derivative is equal to $\pm\frac{1}{2}$ almost everywhere and its vertical partial derivative equals to $0$ almost everywhere. To describe the boundary condition by the normalized height function $h$, compare this to Section~\ref{qseliyrzame}, we should introduce two imaginary threads ${x=-\frac{1}{n-1}}$ and ${x=1+\frac{1}{n-1}}$ where we define the boundary height function. So we sometimes consider a bead model on ${[-\frac{1}{n-1},1+\frac{1}{n-1}]\times[0,1]}$ if necessary. Denote a boundary height function by $h^{\partial}_n$ since its domain of definition depends on $n$.

\begin{definition}\label{defbeadboundarydcondition}
Consider a bead model with $n$ threads defined on $D$. The normalized height function ${h:D\rightarrow \mathbb{R}}$ is said to have
one of the boundary conditions below if there exists a normalized bead model height function ${h':[-\frac{1}{n-1},1+\frac{1}{n-1}]\times[0,1]\rightarrow\mathbb{R}}$ such that $h'|_D=h$ and\\
(a) if $\mathcal{U}$ is a subspace of the functions ${h^{\partial}_n:\partial\big([-\frac{1}{n-1},1+\frac{1}{n-1}]\times[0,1]\big)\rightarrow\mathbb{R}}$, then we say $h$ has a
\emph{boundary condition lying in} $\mathcal{U}$ if
$$h'|_{\partial\big([-\frac{1}{n-1},1+\frac{1}{n-1}]\times[0,1]\big)}\in \mathcal{U}.$$\\
(b) we say $h$ has a \emph{fixed boundary condition} ${h^{\partial}_n:\partial\big([-\frac{1}{n-1},1+\frac{1}{n-1}]\times[0,1]\big)\rightarrow\mathbb{R}}$ if
$${h'|_{\partial\big([-\frac{1}{n-1},1+\frac{1}{n-1}]\times[0,1]\big)}=h^{\partial}_n},$$
\emph{i.e.}, $\mathcal{U}$ has only one element.\\
\end{definition}

When $n\rightarrow\infty$, the domain $[-\frac{1}{n-1},1+\frac{1}{n-1}]\times[0,1]$ tends to the unit square $D$, so if we talk about an \emph{asymptotic boundary condition}, it means a function $h^{\partial}$ defined on $\partial D$, non-decreasing in $y$ and $\frac{1}{2}$-Lipschitz in $x$.

Given such a $h^{\partial}$, for any $n\in\mathbb{N^*}$, we want to consider a boundary height function $h^{\partial}_n$ close to $h^{\partial}$. However, as we will see later in Section~\ref{sectionProveVP}, the dependence of the entropy of the bead model on the boundary condition is delicate, so the meaning of ``close to $h^{\partial}$" should be clarified with attention. We postpone this problem to Section~\ref{sectionProveVP}, and in this section we focus on analytic results.

For every given asymptotic boundary condition $h^{\partial}$, we define the space of \textit{admissible functions} as the closure of the normalized height function, \emph{i.e.},
\begin{definition}\label{admissible}
Given the unit square $D$ and a boundary condition $h^{\partial}$ defined on $\partial D$, a function $h$ is called \emph{admissible} if it is horizontally $\frac{1}{2}$-Lipschitz, vertically non decreasing, and when restricted on $\partial D$ it is equal to $h^{\partial}$. Denote by $\mathcal{H}$ the space of admissible functions.
\end{definition}

We have the following generalization of Dini's theorem, which will be used later:
\begin{lemma}\label{dini}
For any sequence of admissible functions $(h^i)_{i=1,2,...}$, if they converge pointwise to some continuous function, then the convergence is uniform.
\end{lemma}
\begin{proof}
Denote the limiting function by $h^{\infty}$. For any $x$, as $h^i(x,.)$ is non-decreasing and $h^{\infty}(x,.)$ is continuous, the convergence of $h^i(x,.)$ to $h^{\infty}(x,.)$ is uniform on $y\in[0,1]$ by Dini's theorem.

For all $\varepsilon>0$ and for all $x$, there exists $I_x$ such that for all $i>I_x$, $$\sup_{y\in[0,1]}|h^i(x,y)-h^{\infty}(x,y)|<\frac{\varepsilon}{2}.$$
By the Lipschitz condition on $x$, for fixed $x_0$, for every $x$ in the interval $[x_0-\frac{\varepsilon}{2},x_0+\frac{\varepsilon}{2}]$ we have that
$$\sup_{y\in[0,1]}|h^i(x,y)-h^{\infty}(x,y)|<\varepsilon$$
for all $i>I_{x_0}$ and for all $y$. By compactness of $[0,1]$, we can choose $I$ such that for $i>I$ we have
$$\sup_{y\in[0,1]}|h^i(x,y)-h^{\infty}(x,y)|<\varepsilon$$
for all $(x,y)\in D$.
\qed
\end{proof}

\subsection{Statement of the entropy-maximizing problem}\label{sect5.4part1}

From now on we consider a specific case: the bead models on the unit square $D$ with fixed asymptotic boundary condition
$$h|_{\partial D}=h^{\partial},$$
where the left and right boundary conditions are given by a constant function. More precisely, we have $h(0,y)=C_0$, $h(1,y)=C_1$ for $C_0,C_1\in\mathbb{R}$, $h(x,0)$ and $h(x,1)$ are $\frac{1}{2}$-Lipschitz and $h(x,1)\geq h(x,0)$ for $x\in[0,1]$. Recall that any such boundary condition corresponds to an asymptotic shape of (skew) Young diagram. The shape of the diagram is given by the projection of ${h(x,0)}$ and ${h(x,1)}$ in the direction of $y$ to the same plane:
\begin{eqnarray}\label{QKJFQSDKFQ}
\lambda=\{(x,z):2h(x,0)\leq z \leq 2h(x,1)\}.
\end{eqnarray}
In the case where $h(x,0)=h(x,1)$ for some $x\in]0,1[$, the diagram can be decomposed into two independent regions, so without loss of generality we can always suppose that $h(x,1)>h(x,0)$ for all $x\in]0,1[$.

We want to define a global entropy function $Ent(.)$ on the space of admissible functions. According to the definition, an admissible function is differentiable almost everywhere so $ent\circ\nabla h$ is well defined almost everywhere too.

Naturally we can define the entropy of a function $h\in \mathcal{H}$ as the integral of $ent\circ\nabla h$ in $D$, but in fact it is not reasonable. If we take a discontinuous function as
\begin{eqnarray*}
h(x,y)=
\begin{cases}
h(x,0)\text{ if }y\leq\frac{1}{2},\\
h(x,1)\text{ if }y>\frac{1}{2},
\end{cases}
\end{eqnarray*}
then the integral of $ent\circ\nabla h$ is equal to $0$. Meanwhile, in the bead model this corresponds to a phenomenon where almost all the beads are located on one horizontal segment, which should be very rare. Even if a surface is continuous, it can have $\frac{\partial h}{\partial y}$ being $0$ almost everywhere but the height changes (such a function can be constructed via Cantor set).

The solution is to think of a new space where we fix the $x-$axis and turn the space in the $y-z$ plane by $\frac{\pi}{4}$ so that the vertically monotonicity converts to a $1-$Lipshitz condition. This turning map is denoted by $\sim$, and the new coordinate system is denoted by $\t{x}$, $\t{y}$ and $\t{z}$. By properly choosing the $0$ of the coordinates of the system, we have\\
\noindent\begin{minipage}{.5\linewidth}
\begin{eqnarray*}
\begin{cases}
\t{x}=x,\\
\t{y}=\frac{\sqrt{2}}{2}(y+z),\\
\t{z}=\frac{\sqrt{2}}{2}(y-z),
\end{cases}
\end{eqnarray*}
\end{minipage}
\begin{minipage}{.5\linewidth}
\begin{eqnarray*}
\begin{cases}
x=\t{x},\\
y=\frac{\sqrt{2}}{2}(\t{y}-\t{z}),\\
z=\frac{\sqrt{2}}{2}(\t{y}+\t{z}).
\end{cases}
\end{eqnarray*}
\end{minipage}

Consider the surface of any admissible function $h$ as subset of $\mathbb{R}^3$ containing the points $(x,y,z)$ where for any $x,y$ it contains such $z$ that
$$\lim_{\delta\rightarrow0^-}h(x,y+\delta)\leq z\leq\lim_{\delta\rightarrow0^+}h(x,y+\delta).$$
The surface under new coordinates is given by $\t{h}:\t{D}\rightarrow\mathbb{R}$, where the domain of definition $\t{D}$ is uniquely determined the boundary condition of $\mathcal{H}$. Denote by $\t{\mathcal{H}}$ the space $\{\t{h}:h\in \mathcal{H}\}$. We will still call the functions in $\t{\mathcal{H}}$ as admissible function, but under the coordinates $(\t{x},\t{y})$.

One advantage of this change is that the new space $\tilde{\mathcal{H}}$ is compact under the uniform metric, as we see that for fixed $x=\t{x}$, the monotonicity in $y$ of ${h(x,y)}$ turns to $1$-Lipschitz in $\t{y}$ of ${\t{h}(\t{x},\t{y})}$. More precisely we have the following relations between $(s,t)$ and $(\t{s},\t{t})$ being the partial differentials corresponding to the same point $(x,y)$, $(\t{x},\t{y})$:\\
\noindent\begin{minipage}{.5\linewidth}
\begin{eqnarray*}
\begin{cases}
\t{s}=\frac{s}{t+1},\\
\t{t}=\frac{t-1}{t+1},
\end{cases}
\end{eqnarray*}
\end{minipage}
\begin{minipage}{.5\linewidth}
\begin{eqnarray*}
\begin{cases}
t=\frac{1+\t{t}}{1-\t{t}},\\
s=\frac{2\t{s}}{1-\t{t}}.
\end{cases}
\end{eqnarray*}
\end{minipage}

Now consider in the double integral of $ent\circ\nabla h$ in $D$ the change of variable $(x,y)$ to $(\t{x},\t{y})$. As the Jacobian is equal to
\begin{eqnarray}\label{Jacobianjacobian}
J=\frac{\partial x}{\partial \t{x}}\frac{\partial y}{\partial \t{y}}-\frac{\partial x}{\partial \t{y}}\frac{\partial y}{\partial \t{x}} =\frac{\sqrt{2}}{2}(1-\t{t}),
\end{eqnarray}
so the new entropy under the variable change is
$$\wt{ent}(\t{s},\t{t})=\frac{\sqrt{2}}{2}\ln\left(1+\frac{(1-\t{t})
\cos(\frac{2\pi\t{s}}{1-\t{t}})}{\pi (1+\t{t})}\right)(1+\t{t}).$$
Readers can verify that the new entropy $\wt{ent}$ is also strictly concave in the interior of its domain of definition.


We see that the entropy function $\wt{ent}$ is equal to $-\infty$ when the slope $\t{t}=\frac{\partial \t{h}}{\partial \t{y}}$ is equal to $1$, the case corresponding to the discontinuity or quasi-discontinuity in $\mathcal{H}$. If we check the examples we considered as the typical cases that the integral of $ent$ doesn't reflect the entropy, in the new integral they both give an integral equal to $-\infty$. So we take the following definition.

\begin{definition}\label{sdlflhsmdgqedgedeer}
For any admissible function $\t{h}\in\t{\mathcal{H}}$, its \emph{entropy} is defined as
$$\wt{Ent}(\t{h})=\iint_{\t{D}}\wt{ent}\left(\frac{\partial\t{h}}{\partial\t{x}},\frac{\partial\t{h}}{\partial\t{y}}\right)d\t{x}d\t{y},$$
and for any admissible function $h\in \mathcal{H}$, its \emph{entropy} is defined as
$$Ent(h)=\wt{Ent}(\t{h}).$$
\end{definition}

We can announce the main theorem of Section~\ref{sect5.4} now:
\begin{theorem}\label{theoremexistence}
There exists a unique $h_0\in \mathcal{H}$ (resp. $\t{h}_0\in\t{\mathcal{H}}$) which maximizes $Ent(.)$ (resp. $\wt{Ent}(.)$) among all admissible functions of $\mathcal{H}$ (resp. $\t{\mathcal{H}}$).
\end{theorem}

Later, Definition and Lemma~\ref{DL1} will show that there exists at least one admissible function whose entropy is not $-\infty$, so this set $\mathcal{H}$ (resp. $\t{\mathcal{H}}$)is not empty.

\subsection{Proof of the existence and uniqueness of entropy-maximizer}\label{sect5.4part2}
\sectionmark{Existence and uniqueness of the entropy-maximizer}

We prove Theorem~\ref{theoremexistence} in this section. For some technical reason that we will see soon, we still hope to calculate directly $Ent(h)$ by integrating $ent\circ\nabla h$ on $D$ in the normal sense of Lebesgue.

\begin{definition}
Define the following subspace of the admissible functions:
$$\mathcal{H}_0=\{h\in \mathcal{H}:\ Ent(h)=\iint_D ent\circ\nabla h dxdy\},$$
where $Ent(h)=\wt{Ent}(\t{h})$ (Definition~\ref{sdlflhsmdgqedgedeer}). We define $\t{\mathcal{H}}_0$ as the image of $\mathcal{H}_0$ in $\t{\mathcal{H}}$.
\end{definition}

Check the Jacobian in Equation~(\ref{Jacobianjacobian}), we directly get that
\begin{lemma}
The space $\t{\mathcal{H}}_0$ is the subspace of $\t{\mathcal{H}}$ where for every function $\t{h}\in\t{\mathcal{H}}$, the Lebesgue measure of the set $$\{(\t{x},\t{y}):\frac{\partial\t{h}}{\partial\t{y}}=1\}$$
is equal to 0.
\end{lemma}

In particular, any function in $\mathcal{H}_0$ is continuous, so a pointwise convergence of any sequence of admissible functions to a function in $\mathcal{H}_0$ is uniform (Lemma~\ref{dini}).

Obviously for any function $h\in\t{\mathcal{H}}\backslash\t{\mathcal{H}}_0$, $Ent(h)=\wt{Ent}(\t{h})=-\infty$ (while its converse is false). As in Theorem~\ref{theoremexistence} we are only interested in finding the maximizer of the entropy, we can restrict ourselves to any subspace which excludes only some functions whose entropy is $-\infty$, so it suffices to consider Theorem~\ref{theoremexistence} in $\mathcal{H}_0$ and $\t{\mathcal{H}}_0$.

\vspace{0.5cm}

To prove the existence, although with the new coordinates we have compactness, the proof of the semicontinuity in \cite{CKP01} does not apply here because the local entropy $\wt{ent}$ here is no longer bounded from below, and these singularities should be taken into consideration because typically the slope explodes at some boundary points, which is shown in the explicit examples given later.

The method we use is to give a way to construct for every admissible function a good approximation. However, as the construction highly relies on the boundary condition, and it is hard to describe the boundary condition in $\t{\mathcal{H}}$ in a simple and clear way, we choose to do the construction still in $\mathcal{H}$. So in the remaining part of this section, we will often switch between $\t{\mathcal{H}}$ and $\mathcal{H}$. We hope that this inconvenience will not cause too many difficulties to the readers.

\vspace{0.5cm}

We introduce successive technical constructions which will be used in the proof of Theorem~\ref{theoremexistence}.

\begin{definitionAndLemma}\label{DL1}
Given a rectangular domain $[0,1]\times[a,b]$ with a boundary condition where $h(0,y)$ and $h(1,y)$ are constant, $h(x,a)$ and $h(x,b)$ are $\frac{1}{2}$-Lipschitz and $h(x,b)>h(x,a)$ for $x\in]0,1[$, then we can construct an admissible function
$$h^t:[0,1]\times[a,b]\rightarrow\mathbb{R},$$
whose vertical partial derivative $\frac{\partial h^t}{\partial y}$ only take two possible values, and ${ent\circ\nabla h^t}$ is bounded on $D$.

\end{definitionAndLemma}
\begin{proof}
Since $h(x,b)>h(x,a)$ for all $x\in]0,1[$, there exists a $(\frac{1}{2}-\varepsilon)$-Lipschitz function $\bar{h}(x)$ for some $\varepsilon>0$ such that $h(x,b)\geq\bar{h}(x)\geq h(x,a)$ for all $x\in[0,1]$. Let $A=\max_{x\in[0,1]}(\bar{h}(x)-h(x,a))$ and $B=\max_{x\in[0,1]}(h(x,b)-\bar{h}(x))$, consider $D'$ as a subdomain of $[0,1]\times[a,b]$ given by
$$D'=\left\{(x,y)\in[0,1]\times[a,b]:\frac{b-a}{A+B}(A+h(x,a)-\bar{h}(x))\leq y-a\leq\frac{b-a}{A+B}(A+h(x,b)-\bar{h}(x))\right\}.$$

We take
\begin{eqnarray*}
h^t(x,y)=
\begin{cases}
h(x,b)\text{ if }y\geq a+\frac{b-a}{A+B}(A+h(x,b)-\bar{h}(x)),\\
h(x,a)\text{ if }y\leq a+\frac{b-a}{A+B}(A+h(x,a)-\bar{h}(x)),\\
\lambda_a h(x,a)+\lambda_b h(x,b)\text{ otherwise},
\end{cases}
\end{eqnarray*}
where $\lambda_a=\frac{a+\frac{b-a}{A+B}(A+h(x,b)-\bar{h}(x))-y}{\frac{b-a}{A+B}(h(x,b)-h(x,a))}$ and $\lambda_b=1-\lambda_a$. In short, what we do is inscribing into the domain $[0,1]\times[a,b]$ a domain $D'$ of shape corresponding to the boundary condition of $D$. Outside $D'$ we have $\frac{\partial h^t}{\partial y}=0$ so the height function on the boundary of $D$ extends vertically to the boundary of $D'$, and inside $D'$ we construct the surface of $h^t$ by linking every pair of points with the same horizontal coordinate by line segment.

Outside $D'$ we have $\frac{\partial h^t}{\partial y}=0$ so $ent\circ\nabla{h^t}$ is equal to $0$. Inside $D'$ we always have
\begin{eqnarray*}
\frac{\partial h^t}{\partial y}(x,y)&=&\frac{A+B}{b-a}\\
\frac{\partial h^t}{\partial x}(x,y)&=&\frac{\partial\bar{h}}{\partial x}(x),
\end{eqnarray*}
so $h^t$ satisfies the conditions in the statement.
\qed
\end{proof}

The construction of $h^t$ is not unique: it depends on the choice of $\bar{h}$. Figure \ref{sdqlfmhsqgedmgnkqeer} is an illustration of an example $h^t$ where the domain of definition is taken to be the unit square $D$ and $h|_{\partial D}$ corresponds to the square Young diagrams and $\bar{h}$ is taken to be constant. In this example, $h^t$ is piecewise linear.
\begin{figure}[H]
\centering
\includegraphics[width=0.35\textwidth]{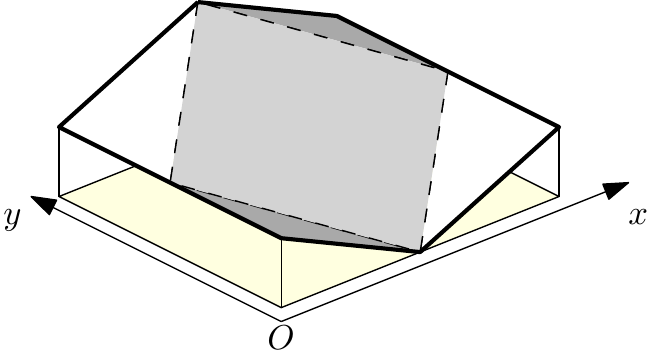}
\caption{An example of $h^t$.\label{sdqlfmhsqgedmgnkqeer}}
\end{figure}

The following corollary is direct.
\begin{corol}
For all $h^{\partial}$ as given asymptotic boundary function on $\partial D$ satisfying that $h(0,y)$ and $h(1,y)$ are constant, $h(x,0)$ and $h(x,1)$ are $\frac{1}{2}$-Lipschitz, and $h(x,1)\geq h(x,0)$ for $x\in]0,1[$, there exists at least an admissible function whose entropy is not $-\infty$.
\end{corol}

We give a series of definitions for technical reasons. Although they are long and redundant, some of them may be used more than one time, so we decide to list them here rather than putting them separately into the proofs. Reader may skip this part and go back once they are needed.
\begin{definition}\label{defs}
\ \\
\begin{itemize}
\item For any $\delta\in]-\infty,\frac{1}{2}[$, $D_{\delta}$ is defined as the domain $[\delta,1-\delta]\times[\delta,1-\delta]$. Attention, when $\delta<0$, the new domain is bigger than the unit square, see Figure \ref{Ddelta} where the dashed square is the unit square $D$.
    \begin{figure}[H]
    \centering
    \includegraphics[width=0.5\textwidth]{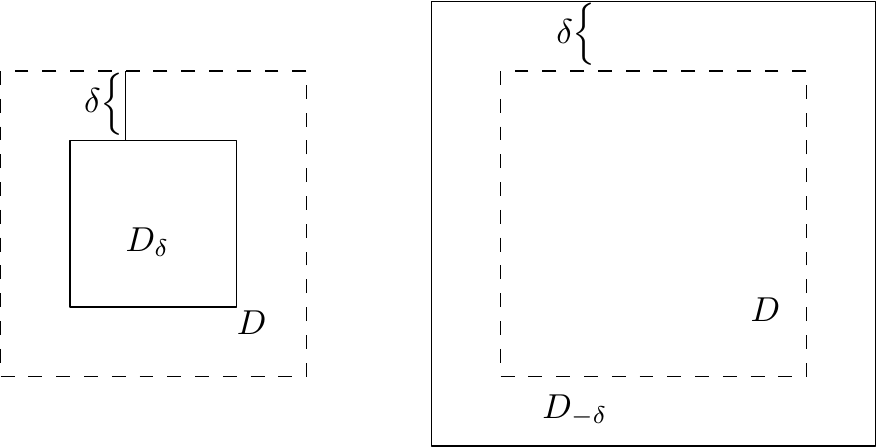}
    \caption{Examples of $D_{\delta}$ and $D_{-\delta}$ for $\delta>0$.\label{Ddelta}}
    \end{figure}
\item For any $\delta_1,\delta_2\in]-\infty,\frac{1}{2}[$ and $z_0\in\mathbb{R}$, define the operator of contraction
    $$P_{\delta_1,\delta_2,z_0}:\{h:D_{\delta_1}\rightarrow\mathbb{R}\}\rightarrow\{h:D_{\delta_2}\rightarrow\mathbb{R}\},$$
    where for any bounded function $h:D_{\delta_1}\rightarrow\mathbb{R}$, we apply on the surface given by $h$ the following map:
    $$(x,y,z)\rightarrow\left(\frac{\delta_2-\delta_1}{1-2\delta_1}+\frac{1-2\delta_2}{1-2\delta_1}x, \frac{\delta_2-\delta_1}{1-2\delta_1}+\frac{1-2\delta_2}{1-2\delta_1}y, z_0+\frac{1-2\delta_2}{1-2\delta_1}z\right).$$
    In other words, the surface of $P_{\delta_1,\delta_2,z_0}(h)$ is taken to be geometrically similar to that of $h$ and to fit the domain $D_{\delta_2}$.
\item For any boundary function $h^{\partial}$ which is respectively constant on the left and right boundaries of $D$, and for $\delta>0$ and any function $h:D_{\delta}\rightarrow\mathbb{R}$ which is respectively constant on the left and right boundaries of $D_{\delta}$, define ${T_{\delta}(h):D\rightarrow\mathbb{R}}$ as the following extension of $h$ on $D$ whenever possible:
    \begin{itemize}
    \item Respectively on $[0,\delta]\times[\delta,1-\delta]$ and on $[1-\delta,1]\times[\delta,1-\delta]$, $T_{\delta}(h)$ is taken to be flat and fitting the boundary conditions on $\partial D$ and $\partial D_{\delta}$.
    \item If for the function constructed in the last step respectively we have
        \begin{eqnarray*}
        h(x,1)&> T_{\delta}(h)(x,1-\delta) &\text{ for }x\in]0,1[,\\
        h(x,0)&< T_{\delta}(h)(x,\delta)\ \ \ \ \ &\text{ for }x\in]0,1[,
        \end{eqnarray*}
        then extend $T_{\delta}(h)$ respectively on these two domains by using the function $h^t$ in Definition and Lemma~\ref{DL1}.
    \end{itemize}
\item For any admissible function $h\in \mathcal{H}$ and any $\delta>0$, its flat extension on $D_{-\delta}=[-\delta,1+\delta]\times[-\delta,1+\delta]$ is the function
      $$F_{\delta}(h):D_{-\delta}\rightarrow\mathbb{R}$$
      which is the unique extension of $h$ on $D_{-\delta}$ whose vertical partial derivative is equal to $0$ everywhere outside $D$ and the horizontal derivative is equal to $0$ if $x\not\in[0,1]$.
\end{itemize}
\end{definition}

\begin{definitionAndLemma}\label{DL2}
Fix a family of non-negative functions ${U_{\delta}\in C^{\infty}(\mathbb{R}^2)}$ parameterized by $\delta>0$, whose integral is equal to $1$ and whose support is contained in a disc centered at the origin and of radius $\delta$. Define the operator
$$C_{\delta}: \mathcal{H}\rightarrow\{h:D_{-\delta}\rightarrow\mathbb{R}\}$$
such that for any admissible $h\in \mathcal{H}$, we extend $h$ to
$$D_{-2\delta}=[-2\delta,1+2\delta]\times[-2\delta,1+2\delta]$$
by using the flat extension defined above, and $C_{\delta}(h)$ is taken to be the convolution of $U_{\delta}$ and the extended $h$ on $D_{-\delta}$.

This new function $C_{\delta}(h)$ is horizontally $\frac{1}{2}$-Lipschitz and vertically non-decreasing. Moreover,\\
(a) the integral of $ent\circ\nabla C_{\delta}(h)$ on $D_{-\delta}$ is bigger than that of $ent\circ\nabla h$ on $D$.\\
(b) if $h\in \mathcal{H}_0$, then the integral of $ent\circ\nabla\big(C_{\delta}(h)\big)$ on $D_{-\delta}$ tends to $Ent(h)$ when $\delta\rightarrow 0$.
\end{definitionAndLemma}
Remark: the letter $C$ stands for convolution.

\begin{proof}
The function $C_{\delta}(h)$ is obviously horizontally $\frac{1}{2}$-Lipschitz and vertically non-decreasing.

To prove (a), it suffices to note that $ent$ is concave and outside $D$ the local entropy of the extended $h$ is always $0$, so for any admissible function $h$ the value of $ent\circ\nabla(U_{\delta}*h)$ at any point is bigger than $U_{\delta}*ent\circ\nabla h$ (define by default $ent\circ\nabla h=0$ on $D_{-2\delta}\backslash D$).

If $h\in \mathcal{H}_0$, then $Ent(h)$ is equal to the integral of $ent\circ\nabla h$ on $D$. As we have already proved (a), to prove the convergence it suffices to prove that
$$\limsup_{\delta\rightarrow 0}\int_{D_{-\delta}}ent\circ\nabla\big(C_{\delta}(h)\big)dxdy \leq \int_{D}ent\circ\nabla h dxdy.$$

Let $\delta_i$ be any sequence tending to $0$, as $\nabla h\in L^1$, by property of the convolution we have the convergence in $L^1$ of $U_{\delta}*(\nabla h)$ to $\nabla h$. We can take a subsequence $\delta_{i_j}$ of $\delta_i$ such that the convergence is almost everywhere, so $ent\circ\nabla C_{\delta}(h)$ tends to $ent\circ\nabla h$ almost everywhere for this subsequence of $\delta$.

Since $ent$ is a non-positive function, apply Fatou's lemma for this subsequence we have
$$\limsup_{j\rightarrow\infty}\int_{D_{-\delta_{i_1}}}\mathds{1}_{D_{-\delta_{i_j}}}ent\circ\nabla\big(C_{\delta_{i_j}}(h)\big)dxdy \leq \int_{D_{-\delta_{i_1}}}ent\circ\nabla h dxdy,$$
while by construction the right hand side is equal to $Ent(h)$. As the sequence of $\delta_i$ can be chosen arbitrarily, so the inequality above is true for any $\delta\rightarrow 0$.
\qed
\end{proof}

We remark that the convolution provides us a function of better regularity but breaks the boundary condition. Using the convolution technique here above, in Lemma~\ref{lemmaapproximation} we construct an admissible function of good enough regularity.

We also remark that, restricted on the left (resp. right) boundary of $D_{-\delta}$, it is constant and equal to the value of $h$ on the left (resp. right) boundary of $D$, while its restriction on the upper and lower boundaries are functions that only depend on the boundary condition of $h$ on the upper and lower boundaries of $D$.



\begin{lemma}\label{lemmaapproximation}
For any given boundary condition, any $\delta<\frac{1}{2}$, and for any $\delta'$ small enough (depending on $\delta$ and the boundary condition), then there exists a family of operators
$$A_{\delta,\delta'}:\mathcal{H}\rightarrow \mathcal{H},$$
parameterized by $\delta$ and $\delta'$, verifying that\\
\emph{(a)} for all functions $h\in \mathcal{H}$,
\begin{eqnarray*}
Ent\big(A_{\delta,\delta'}(h)\big)\geq(1-2\delta)^2 Ent(h)+O(\delta\ln\delta),
\end{eqnarray*}
where the function $O(.)$ only depends on $\delta$ and the boundary condition.\\
\emph{(b)} for any $h\in \mathcal{H}_0$, then for $\delta'$ sufficiently small depending on $h$, we have
\begin{eqnarray*}
Ent\big(A_{\delta,\delta'}(h)\big)=(1-2\delta)^2 Ent(h)+O(\delta\ln\delta).
\end{eqnarray*}.\\
\emph{(c)} if $Ent(h)>-\infty$, then $ent\circ\nabla A_{\delta,\delta'}(h)$ is bounded from below on $D$ by some constant depending on $\delta$, $\delta'$ and $Ent(h)$ (and not on the precise $h$).
\end{lemma}

The letter $A$ stands for approximation.

\begin{proof}
Technically we will limit ourselves to the case that the domain
$$\lambda=\{(x,z):x\in[0,1],z\in[2h(x,0),2h(x,1)]\}$$
is \emph{star convex}, \emph{i.e.} there exist points $(x_0,z_0)$ such that for any $\alpha\in]0,1[$ and $x\in[0,1]$, we have
\begin{eqnarray*}
\alpha(h(x,1)-z_0)&<&h(x_0+\alpha(x-x_0),1)-z_0,\\
\alpha(h(x,0)-z_0)&>&h(x_0+\alpha(x-x_0),0)-z_0.
\end{eqnarray*}
Moreover, we ask that every straight line passing $(x_0,z_0)$ is not tangent to the boundary of this domain at any point. This assures a distance of order $1-\alpha$ between the domain and the one multiplied by $\alpha$, $\alpha$ close to $1$.

For example, the function $h(x,1)=\frac{1}{2}|x-\frac{1}{2}|$ and $h(x,0)=-\frac{1}{2}|x-\frac{1}{2}|$ verify the condition above for $(x_0,y_0)=\frac{1}{2}$. In the case that the domain does not verify such condition, by the Lipschitz condition and compactness we can cut the domain vertically into disjoint parts such that every part verifies this condition. The following procedure still works with small modifications if we treat each part simultaneously.

Without loss of generality throughout the remainder of this section we will always suppose the star convexity with respect to $(x_0,z_0)=(\frac{1}{2},0)$. In this case, for every $\delta<\frac{1}{2}$, we take $\delta'$ small enough (to be specified below) so that we can define
$$T_{\delta}P_{-\delta',\delta,0}C_{\delta'}(h),$$
where $T$ and $P$ are operators defined in Definition~\ref{defs}. In this case, we just let
$$A_{\delta,\delta'}=T_{\delta}P_{-\delta',\delta,0}C_{\delta'}.$$

We will prove that such defined $A_{\delta,\delta'}$ verifies the conclusion in the Lemma. For $h$ such that $Ent(h)=-\infty$, the results (a) and (b) are automatical. So without loss of generality we consider the $h$ such that $Ent(h)>-\infty$, so these $h$ are in $\mathcal{H}_0$, and
$$Ent(h)=\int_D ent\circ\nabla h\ dxdy.$$

As proved in Definition and Lemma~\ref{DL2}, for any $h\in \mathcal{H}$,
\begin{eqnarray}\label{sjflwsjreowiwu}
\int_{D_{-\delta'}}ent\circ\nabla\big(C_{\delta'}(h)\big) dxdy \geq Ent(h),
\end{eqnarray}
and for any $h\in \mathcal{H}_0$, when $\delta'\rightarrow 0$,
\begin{eqnarray}\label{sjflkfdgjdlwiwu}
\lim_{\delta'\rightarrow 0}\int_{D_{-\delta'}}ent\circ\nabla\big(C_{\delta'}(h)\big) dxdy = Ent(h).
\end{eqnarray}

The map $P_{-\delta',\delta,0}$ keeps gradient, so
\begin{eqnarray}\label{djoingondoinxioorgni}
\int_{D_{\delta}}ent\circ\nabla\big(P_{-\delta',\delta,0}C_{\delta'}(h)\big) dxdy
=\frac{(1-\delta)^2}{(1+\delta')^2}\int_{D_{-\delta'}}ent\circ\nabla\big(C_{\delta'}(h)\big) dxdy.
\end{eqnarray}

Consider the function $P_{0,\delta,0}(h):D_{\delta}\rightarrow\mathbb{R}$, and we claim that $T_{\delta}P_{0,\delta,0}(h)$ is well defined for $\delta$ small enough, \emph{i.e.}, it is possible to fill in $D\backslash D_{\delta}$ piecewisely by $h^t$ constructed in Definition and Lemma~\ref{DL1}. Moreover, we will prove that the filling function will give a contribution of order ${O(\delta\ln\delta)}$ when $\delta\rightarrow 0$ in the integral of $ent$.

On $[0,\delta]\times[\delta,1-\delta]$ and $[1-\delta,1]\times[\delta,1-\delta]$, it is always possible to define the extended function $T_{\delta}P_{0,\delta,0}(h)$, which is of $0$ vertical slope so gives $0$ contribution in $Ent$. In the following, without loss of generality we only treat the region near the upper boundary of $D$, that is $[0,1]\times[1-\delta,1]$.

By the star convex hypothesis, there exists a $(\frac{1}{2}-K\delta)$-Lipschitz function $\bar{h}$ for some constant $K>0$ not depending on $\delta$ such that
\begin{eqnarray*}
h(x,1)\geq\bar{h}(x)\geq
P_{0,\delta,0}(h)(x,1-\delta),
\end{eqnarray*}
since $P_{0,\delta,0}(h)(x,1-\delta)$ is already defined on $[0,1]\times[\delta,1-\delta]$. By construction (see Definition and Lemma~\ref{DL1}), the local entropy $ent$ is at most of order $O(\ln\delta)$, thus the contribution of this region in $Ent$ is at most of order $O(\delta\ln\delta)$. The function $O(.)$ only depends on the constant $K$ we just mention.

Now consider $P_{-\delta',\delta,0}C_{\delta'}(h)$ instead of $P_{0,\delta,0}(h)$. Clearly for $\delta'$ small enough depending only on $\delta$ and the boundary condition, all the results above are still true.

Combining this with~(\ref{sjflwsjreowiwu}),~(\ref{sjflkfdgjdlwiwu}) and~(\ref{djoingondoinxioorgni}), we prove (a) and (b).

To prove (c), it suffices to note that on $D_{\delta}$, $A_{\delta,\delta'}(h)$ is constructed via convolution. By concavity of $ent$,
$$ent\circ\nabla\big(A_{\delta,\delta'}(h)\big)(x,y)\geq \frac{(1-\delta)^2}{(1+\delta')^2} U_{\delta'}*\left(ent\circ\nabla h(x,y)\right),$$
and the right hand side has a trivial lower bound depending on $Ent(h)$ and $U_{\delta'}$. Outside $D_{\delta}$, $A_{\delta,\delta'}(h)$ is constructed via $h^t$ in Definition and Lemma~\ref{DL1}, whose local entropy function $ent$ has also a lower bound only depending on the boundary condition and $\delta$. Thus we prove (c).
\qed
\end{proof}

For any $\lambda\in\mathbb{R}^+$, let $\mathcal{H}^{\lambda}$ be the subspace of the admissible functions such that ${ent\circ\nabla{h}(x,y)\geq -\lambda}$ for $(x,y)\in \mathring{D}$ almost everywhere. Clearly $\mathcal{H}^{\lambda}$ is an increasing sequence of space of functions in $\lambda$. Definition and Lemma~\ref{DL1} constructs a function $h^t$ whose entropy ${ent\circ\nabla h^t}$ is bounded, so $\mathcal{H}^{\lambda}$ is not empty for all $\lambda$ bigger than some $\Lambda\in\mathbb{R}^+$.

We furthermore have the following two lemmas.

\begin{lemma}\label{triangulation}
For $\delta$ and $\delta'$ such that $A_{\delta,\delta'}(h)$ verifies Lemma~\ref{lemmaapproximation}, for any $\varepsilon>0$, there exists $l>0$ such that we can construct a function $h'$ as below:\\
\emph{(a)} $h'$ agrees with $A_{\delta,\delta'}(h)$ on $D\backslash D_{\delta}$.\\
\emph{(b)} on $D_{\delta}$ it is piecewise linear on a triangle mesh of size $O(l)$.\\
\emph{(c)} the sup norm between $h'$ and $A_{\delta,\delta'}(h)$ is less than $\varepsilon$, and
$$|Ent(A_{\delta,\delta'}(h)) -Ent(h')|<\varepsilon.$$
\end{lemma}

We need the triangulation to avoid the possible explosion of $ent$ near the singularity $(s,t)=(\pm\frac{1}{2},0)$. Readers will see later that the lemma above plays the same role as Lemma 2.2 of \cite{CKP01} where the authors give an approximation by triangulation, and it is interesting to compare them. In \cite{CKP01}, the main problem the authors deal with is the lack of smoothness, and thanks to Lemma~\ref{lemmaapproximation} this is not the main focus in our case.

\begin{proof}
Define respectively for $+$ and $-$ the set of frozen points as
$$K^{\pm}:=\left\{(x,y)\in D_{\delta}:\frac{\partial h}{\partial x}(x,y)=\pm\frac{1}{2}\right\},$$
and for all $\eta$, define respectively for $+$ and $-$ the set
$$K^{\pm}_{\eta}:=\left\{(x,y)\in D_{\delta}:\text{dist}((x,y),D\backslash K^{\pm})\geq\eta\right\}.$$

For any $\delta'$ small enough we have the following equality:
$$K^{\pm}_{\delta'}=\left\{(x,y)\in D_{\delta}:\frac{\partial A_{\delta,\delta'}(h)}{\partial x}(x,y)=\pm\frac{1}{2}\right\}.$$

The Lebesgue measure of $K^{+}_{\eta}\cup K^{-}_{\eta}$ is a decreasing function in $\eta$ so it is continuous almost everywhere. From now on we take $\delta''$ close enough to $\delta'$ such that $\delta''$ is a continuous point of the measure of $K^{+}_{\eta}\cup K^{-}_{\eta}$, the sup norm between $A_{\delta,\delta''}(h)$ and $A_{\delta,\delta'}(h)$ is less than $\frac{\varepsilon}{2}$, and $|Ent(A_{\delta,\delta'}(h))-Ent(A_{\delta,\delta''}(h))|<\frac{\varepsilon}{2}$ (by the arguments we used in the proof of Lemma~\ref{lemmaapproximation} it is possible to do so).

The function $A_{\delta,\delta''}(h)$ belongs to some $\mathcal{H}^{\lambda}$.

For any $l$ small enough and dividing $1-2\delta$, consider a $l$-grid on $D_{\delta}$, and consider the mesh of isosceles right triangles constructed by linking the northeast and southwest vertices of every $l$-square of the grid. Consider the only function $h'$ which is a linear function on every triangle of the mesh and agrees with $A_{\delta,\delta''}(h)$ on vertices of triangles and outside $D_{\delta}$ we just take $h'=A_{\delta,\delta''}(h)$.

As $A_{\delta,\delta''}(h)$ is $C^{\infty}$ on $D_{\delta}$, for $l$ sufficiently small, $h'$ agrees with $h$ within $\frac{\varepsilon}{2}$, so the first approximation in (c) is direct.

We also conclude that there exists $\lambda'\in\mathbb{R}^+$ such that $h'\in \mathcal{H}^{\lambda'}$ (\emph{i.e.} $ent\circ\nabla h'$ is bounded from below by $-\lambda'$ almost everywhere) and $\lambda'$ is independent of $l$. The boundness on $D\backslash D_{\delta}$ is trivial. Inside $D_{\delta}$, consider any triangle of the $l$-mesh. Its slope is equal to the integrals of the partial differentials of $A_{\delta,\delta''}(h)$ along the edge which is the diagonal of the $l$-square then normalized by the length $\sqrt{2}l$. Since on any point the partial differentials $(\partial A_{\delta,\delta''}(h)/\partial x,\partial A_{\delta,\delta''}(h)/\partial y)$ is within the set
$$\{(s,t):ent(s,t)\geq-\lambda\},$$
the normalized integral of the partial differentials is within the convex hull of this set, which is easy to be shown to be included in $\{(s,t):ent(s,t)\geq-\lambda'\}$ for some finite $\lambda'$, so we get the wanted property.

We will respectively treat the case where $s=\pm\frac{1}{2}$ (note that whenever this is true we have also $\{t=0\}$) and where $s\neq\pm\frac{1}{2}$. We show that for $l$ small enough, $ent\circ\nabla h'$ approximates well $ent\circ\nabla h$ on most points in the interior of these sets, and the rest points have a contribution arbitrarily small.

As the Lebesgue measure of $K^{\pm}_{\eta}$ is continuous at $\delta''$, for the $\varepsilon$ given, find $d$ small enough such that the Lebesgue measure of $(K^{+}_{\delta''}\cup K^{-}_{\delta''})\backslash(K^{+}_{\delta''+d}\cup K^{-}_{\delta''+d})$ is less than $\frac{\varepsilon}{8 \lambda'}$.

For any $l<\frac{d}{\sqrt{2}}$, the $l$-squares respectively intersecting $K^{\pm}_{\delta''+d}$ gives a cover of $K^{\pm}_{\delta''+d}$, and they are contained in $K^{\pm}_{\delta''}$. In other words, we give an inner approximation of $K^{\pm}_{\delta''}$ by disjoint squares of length $l$. The measure of the difference set is less than that of $(K^{+}_{\delta''}\cup K^{-}_{\delta''})\backslash(K^{+}_{\delta''+d}\cup K^{-}_{\delta''+d})$, which is less than
$\frac{\varepsilon}{8 \lambda'}$.

Note that $ent\circ\nabla h'(x,y)=0$ on $K^{\pm}_{\delta''+d}$, so the local entropy of the function $h'$ is equal to that of $A_{\delta,\delta''}(h)$ there.


On the other hand, for any point $(x_0,y_0)\in D_{\delta}$ such that $\frac{\partial h}{\partial x}(x_0,y_0)\neq\pm\frac{1}{2}$, there exists $r$ (which depends on $(x_0,y_0)$) such that within the $r$-neighborhood of $(x_0,y_0)$, $\nabla h$ is within a small convex neighborhood of $\nabla h(x_0,y_0)$ where
$$|ent\circ\nabla h(x,y)-ent\circ\nabla h(x_0,y_0)|<\frac{\varepsilon}{4}.$$

This gives an open cover of the points that $\frac{\partial h}{\partial x}\neq\pm\frac{1}{2}$. We may choose $\rho$ small enough such that the area of the union of the balls of diameters bigger than $\rho$ is bigger than the measure of
$$\{(x,y)\in D_{\delta}:\frac{\partial h}{\partial x}(x,y)\neq\pm\frac{1}{2}\}$$
minus $\frac{\varepsilon}{8 \lambda'}$.

Now take $l<\min\{\frac{d}{\sqrt{2}},\rho\}$, for the piecewise linear function $h'$, we have
$$|ent\circ\nabla h(x,y)-ent\circ\nabla h(x_0,y_0)|<\frac{\varepsilon}{4}$$
on all points except a set of measure $\frac{\varepsilon}{2 \lambda'}$, thus
$$|Ent\big(A_{\delta,\delta''}(h)\big)-Ent(h')|<\frac{\varepsilon}{2},$$
thus we have finished the proof.
\qed
\end{proof}

Figure \ref{lMesh} gives an illustration of several notions used in the proof above: the $l$-mesh, the frozen point set $K^{\pm}$ and the set $K^{\pm}_{\delta''}$.
\begin{figure}[H]
\centering
\includegraphics[width=0.3\textwidth]{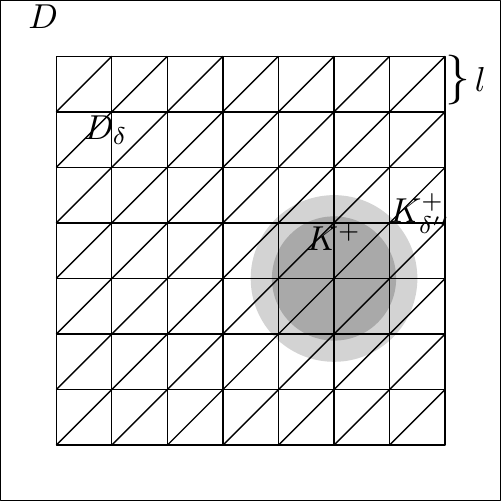}
\caption{The $l$-mesh, the frozen point set $K^{\pm}$ and the set $K^{\pm}_{\delta''}$.}\label{lMesh}
\end{figure}

\begin{lemma}\label{lemmaHlambda}
The space $\mathcal{H}^{\lambda}$ is compact and semicontinuous with respect to the sup norm of $H$, and there exists a unique function $h_{\lambda}$ that maximizes $Ent(.)$ among all functions of $\mathcal{H}^{\lambda}$.
\end{lemma}
\begin{proof}
When $\frac{\partial h}{\partial y}$ tends to plus infinity, the entropy $ent(\frac{\partial h}{\partial x},\frac{\partial h}{\partial y})$ tends to minus infinity. So for any ${\lambda}$, the boundness condition implies that $\frac{\partial h}{\partial y}$ is bounded from above by a constant depending on ${\lambda}$. Thus we have a Lipshitz condition on $y$ for $\mathcal{H}^{\lambda}$, so $\mathcal{H}^{\lambda}$ is relative compact in $\mathcal{H}$ under the uniform norm.

To prove the compactness we should also prove that $\mathcal{H}^{\lambda}$ is closed in $\mathcal{H}$. This is because the constraint ${ent\circ\nabla h(x,y)\geq -\lambda}$ can be interpreted as a constraint on the horizontal and vertical slope. By the argument of the local convexity near the points $(\pm\frac{1}{2},0)$, a space verifying such geometric constraint is closed.

As for the semicontinuity of $Ent$, thanks to Lemma~\ref{triangulation} and the boundness of $ent$ on $\mathcal{H}^{\lambda}$, the proof is nothing different from Lemma 2.3 in \cite{CKP01}.


All these imply the existence of a maximizer: taking a sequence of height functions whose entropies tend to $\sup\{Ent(h):h\in \mathcal{H}^{\lambda}\}$, then there exists a converging subsequence. By semicontinuity, the entropy of the limit height function is $\sup\{Ent(h):h\in \mathcal{H}^{\lambda}\}$.
\qed
\end{proof}

\begin{corol}\label{corolsup}
We have
$$\sup_{h\in \mathcal{H}}Ent(h)=\lim_{\lambda\rightarrow\infty}\sup_{h\in \mathcal{H}^{\lambda}}Ent(h).$$
\end{corol}
\begin{proof}
This is a direct corollary of Lemma~\ref{lemmaapproximation} and Lemma~\ref{lemmaHlambda}.
\end{proof}

\vspace{0.5cm}

\noindent\textit{Proof of Theorem }\ref{theoremexistence}. Choose $I\in\mathbb{N}$ such that $\mathcal{H}^{I}$ is not empty. For any $i\geq I$, consider the function $h_i$ as the maximizer of $Ent$ among all functions of $\mathcal{H}^{i}$. We first take the turned coordinates $(\t{x},\t{y})$. By compactness of $\t{\mathcal{H}}$, the sequence $\t{h}_i$ given by Corollary~\ref{corolsup} has a converging subsequence $\t{h}_{i_j}$ under the uniform norm. Denote the limit by $\t{h}_0$, and denote its preimage by $h_0\in \mathcal{H}$. The function $h_0$ is well defined except for the discontinuous points, and on these points we will take $h_0(x,y)=\limsup_{(a,b)\rightarrow(x,y)}h(a,b)$. It is easy to verify that the convergence of $h_{i_j}$ to $h_0$ is pointwise except on the discontinuous points.


To simplify the notation, here rather than a subsequence of $\t{h}_i$, we suppose that the sequence $\t{h}_i$ converges. This simplification does not lose generality: we prove below that the limit of the subsequence is the unique function that maximizes $Ent(.)$, so by the uniqueness of the limit of the subsequence, the sequence itself converges to the same limit.

We now prove that $\t{h}_0\in \t{\mathcal{H}}_0$. Otherwise, the set $\Delta$ defined by
$$\Delta=\{(\t{x},\t{y})\in\t{D}:\frac{\partial\t{h}_0}{\partial\t{y}}=1\}$$
has a positive measure $\mu>0$.

For all $\varepsilon$, there exists an open set $U\subset\mathbb{R}^2$ such that $\Delta\subset U$ and the $\mathbb{R}^2$-Lebesgue measure of $U$ is less than $\mu(1+\varepsilon)$. Denote the $\mathbb{R}^2$-Lebesgue measure by $|\ .\ |$. The set $U$ is the union of at most countable open discs, and we choose finite discs such that the $2-$Lebesgue measure of their union is bigger than $\mu$, and we can cut their union into finite disjoint convex parts.

Let $\Delta_1$, $\Delta_2$,...,$\Delta_M$ respectively be their closures. The sum of $|\Delta_k|$ for ${k=1,2,...,M}$ is less than $\mu(1+\varepsilon)$ and bigger than $\mu$, and the sum of $|\Delta_k\cap\Delta|$ is bigger than $\mu(1-\varepsilon)$. Thus, there exists a subset $K$ of $\{1,2,..M\}$ such that for every $k\in K$, $\frac{|\Delta_k\cap\Delta|}{|\Delta_k|}\geq 1-3\varepsilon$ and $\sum_{k\in K}|\Delta_k|\geq\frac{1}{3}\mu$.

For any $\t{h}\in\t{\mathcal{H}}$ and $k\in K$, define the average vertical slope on $\Delta_k$ as
$$av_y^{\Delta_k}(\t{h}):=\frac{1}{|\Delta_k|}\iint_{\Delta_k}\frac{\partial\t{h}}{\partial\t{y}}d\t{x}d\t{y}.$$

As $\frac{\partial\t{h}}{\partial\t{y}}\geq-1$, for all $k\in K$, we have
$$av_y^{\Delta_k}(\t{h}_0)\geq \frac{|\Delta_k\cap\Delta|}{|\Delta_k|}-\left(1-\frac{|\Delta_k\cap\Delta|}{|\Delta_k|}\right)=1-6\varepsilon.$$

As the convergence of $\t{h}_i$ to $\t{h}_0$ is uniform, there exists $J$ such that for all $i\geq J$, $\t{h}_i$ is within a $\varepsilon\min_{k\in K}\{\text{diam}\Delta_k\}$-neighborhood of $\t{h}_0$, then $av_y^{\Delta_k}(\t{h}_i)\geq 1-8\varepsilon$ for all $k\in K$ and $i\geq J$.

By concavity of $\wt{Ent}$, negativity of $\wt{ent}$, we get
\begin{eqnarray}\label{ineqtilde}
\wt{Ent}(\t{h}_i)\leq\sum_{k\in K} \iint_{\Delta_k}\wt{ent}\left(\frac{\partial\t{h}_i}{\partial\t{x}},\frac{\partial\t{h}_i}{\partial\t{y}}\right)d\t{x}d\t{y}
\leq\sum_{k\in K}\iint_{\Delta_k}\wt{ent}(av_x^{\Delta_k},av_y^{\Delta_k}) d\t{x}d\t{y},
\end{eqnarray}
where the average horizontal height change $av_x^{\Delta_k}$ is defined as analogue of $av_y^{\Delta_k}$. When $\varepsilon\rightarrow 0$, the average vertical slopes on all $\Delta_k$ uniformly tend to $1$ so $\wt{ent}(av_x^{\Delta_k},av_y^{\Delta_k})$ uniformly tend to $-\infty$. As the sum of the $2-$Lebesgue measure of $\Delta_k$ is bigger than $\frac{1}{3}\mu$, we prove that $\wt{Ent}(\t{h}_i)$ tend to $-\infty$. However, by definition it should be finite and increasing, thus we get a contradiction and prove that $\t{h}_0\in\t{\mathcal{H}}_0$, so $h_0\in \mathcal{H}_0$.

By Lemma~\ref{lemmaapproximation} for all $h_i$ and $h_0$, for all $\varepsilon>0$, there exist $\delta$ and $\delta'$ such that the function $O(\delta\ln\delta)$ depending only on the boundary condition is less than $\varepsilon$, and
\begin{align}
\label{twoinequality}
\begin{split}
Ent\big(A_{\delta,\delta'}(h_i)\big)&\geq(1-2\delta)^2 Ent(h_i)-\varepsilon,\ i=I,I+1,...\\
Ent\big(A_{\delta,\delta'}(h_0)\big)&\in[(1-2\delta)^2 Ent(h_0)-\varepsilon,(1-2\delta)^2 Ent(h_0)+\varepsilon].\\
\end{split}
\end{align}

By construction, for all $i$,
$$Ent\big(A_{\delta,\delta'}(h_i)\big)\geq (1-2\delta)^2 Ent(h_I)-\varepsilon,$$
and by Lemma~\ref{lemmaapproximation} (c) there exists $\Lambda(\delta,\delta')\in\mathbb{R}$ such that for all $i=I,I+1,...$, $A_{\delta,\delta'}(h_i)\in \mathcal{H}^{\Lambda(\delta,\delta')}$.

The convergence of $h_i$ to $h_0$ when $i\rightarrow\infty$ implies the convergence of $A_{\delta,\delta'}(h_i)$ to $A_{\delta,\delta'}(h_0)$. By semicontinuity of $Ent$ for functions of $\mathcal{H}^{\Lambda(\delta,\delta')}$, we have
$$Ent\big(A_{\delta,\delta'}(h_0)\big)\geq\limsup_{i\rightarrow\infty}Ent\big(A_{\delta,\delta'}(h_i)\big).$$

Compare this inequality to~(\ref{twoinequality}), we prove that $Ent(h_0)\geq\limsup_{i\rightarrow\infty}Ent(h_i)$, and by Lemma~\ref{corolsup}, $h_0$ maximizes $Ent(.)$ among functions of $\mathcal{H}$. It is the unique maximizer because of the strict concavity of $ent$ in the interior of its domain of definition.
\qed

\section{The variational principle}\label{sectionProveVP}

In this section, we establish a variational principle and prove a large deviation result for the bead model, which reveals a limit shape of the model when the size tends to infinity. A lot of notations appear. So as not to confuse the readers, we give here, in the very beginning of this section, a general convention on the use of notation.
\begin{itemize}
\item The number of beads or the number of horizontal lozenges is denoted by $N(n)$, where $n$ emphasizes its dependence on the number of threads $n$.
\item For any upper index $u$ and lower index $l$,
\begin{itemize}
\item $\mathcal{H}^u_l$ will denote the space of configurations.
\item $X^u_l$ is a uniformly chosen random configuration of $\mathcal{H}^u_l$.
\end{itemize}
\item If the lower index $l$ is $n$, we mean the bead model with $n$ threads, and if $l$ is $mn,n$, that means we are considering a discrete approximation of the bead model by lozenge tilings.
\item The letter $h$ is used for the normalized bead height function. $H$ with a lower index $n$ means the un-normalized bead height function, and with a lower index $mn,n$ means a height function of lozenge tiling. 
\item We use $h^{\partial}_n$ to denote a fixed boundary condition of the normalized bead height function, while that of the lozenge tiling is given by the domain $R_{mn,n}$.
\end{itemize}

The main result of this section is the following theorem. Consider a fixed normalized boundary condition as in Section~\ref{sect5.4part1}, \emph{i.e.}, a normalized boundary function $h^{\partial}$ defined on $\partial D$ where $D$ is the unit square, and we furthermore ask that $h^{\partial}(0,y)$, $h^{\partial}(1,y)$ viewed as functions of $y$ are constant functions, while ${h^{\partial}(x,0)<h^{\partial}(x,1)}$ as functions of $x$ are $\frac{1}{2}$-Lipschitz.

Fix the asymptotic boundary function $h^{\partial}$. For any $n\in\mathbb{N}^*$, consider the following boundary height function $h^{\partial}_n$ of a bead model with $n$ threads and normalized into $D$:
$$h^{\partial}_n:\left([-\frac{1}{n-1},1+\frac{1}{n-1}]\times[0,1]\right)\rightarrow\mathbb{R}.$$
We furthermore ask that the function $h^{\partial}_n$ is constant if restricted to ${x=-\frac{1}{n-1}}$ or ${x=1+\frac{1}{n-1}}$ and that for any $x\in[0,1]$,
$$\max\left\{|h^{\partial}_n(x,0)-h^{\partial}(x,0)|,
|h^{\partial}_n(x,1)-h^{\partial}(x,1)|\right\}<\frac{1}{n-1}.$$

\begin{theorem}\label{thm4.3CKP01}
For any admissible function $h:D\rightarrow\mathbb{R}$ that agrees with $h^{\partial}$ on $\partial D$ and $Ent(h)>-\infty$, define ${\mathcal{V}_{\delta}(h)}$ as the $\delta$-neighborhood of $h$ under the supremum norm. For any $n\in\mathbb{N^*}$, consider the normalized bead model on $D$ with $n$ threads and fixed boundary condition $h^{\partial}_n$. Let $\mathcal{H}_n^{\mathcal{V}_{\delta}(h)}$ be the space of configurations whose normalized height function is in ${\mathcal{V}_{\delta}(h)}$ and let ${X_n^{\mathcal{V}_{\delta}(h)}}$ be a random configuration uniformly chosen in this space, then we have
\begin{eqnarray}\label{dfserfew}
\lim_{\delta\rightarrow 0}\lim_{n\rightarrow\infty}\frac{S(X_n^{\mathcal{V}_{\delta}(h)})}{n^2}=Ent(h),
\end{eqnarray}
where we recall that $S(X)$ is the adjusted combinatorial entropy of the random variable $X$.
\end{theorem}

Note that, in order to differ with $\mathcal{U}$ (the letter we use to a set of boundary condition), here we use $\mathcal{V}$ to denote the neighborhood of an admissible function $h$ which is a subspace of the height functions of the bead configurations.

Here we give an outline of the proof of this theorem. Recall that with the help of Lemma~\ref{beadanddimerpartitionfunction}, we are able to approach the entropy $S$ of a bead model via dimers, well studied in \cite{CKP01}.

The main idea of the proof is to use triangulations. Although most results are given for the bead model on the unit square, it is not hard to generalize the definition to other shapes, particularly the isosceles right triangles by  Proposition~\ref{prop-ent}.

We first study a bead model on the unit square $D$ with almost planar boundary condition. Unlike in \cite{CKP01}, in the bead model, close boundary condition (under the uniform norm on $\partial D$) does not imply close entropy: by manipulating the distance between little jumps on the boundary, in any neighborhood of a nearly planar boundary condition we can always find two boundary conditions that give entropies arbitrarily far apart.

To solve this problem, we give a more precise definition that clarify what is an ``almost planar" boundary condition (Definitions~\ref{definitionalmostplanar} and~\ref{definitionalmostplanardiscrete}). We prove that the normalized (and adjusted) entropy of a bead model with almost planar boundary condition converges (Lemma~\ref{Lemmaentexist}).

Then we prove a series of technical lemmas (Lemma~\ref{LemmaProp20CEP},~\ref{lemmaazuma},~\ref{lemmaProp3.4CKP}) which the readers can find their origins in \cite{CKP01} and \cite{CEP}. They serve to prove Lemma~\ref{prop3.6CKP0}, which concludes that among all nearby boundary conditions, the almost planar ones have the biggest entropy. Theorem~\ref{thm-ent} proves that it is equal to $ent(.,.)$, and the entropy of the bead model whose boundary condition is the union of nearby functions in a neighborhood of an almost planar function has the same limit when the radius of the neighborhood tends to $0$.

Lemma~\ref{triangulation} and Lemma~\ref{triangulationprime} give two triangulations of a surface of bead configurations respectively for a lower bound and an upper bound of entropy. This proves Theorem~\ref{thm4.3CKP01}, a variational principle of the bead model.

As a corollary of the variational principle, in the end of this section, we prove a limiting behavior for the bead model with fixed boundary condition: when $n\rightarrow\infty$, the random surface of the bead configuration converges in probability to the function $h_0$ that maximizes $Ent(h)$.

\vspace{0.5cm}

In the beginning we give a lemma which is somehow independent of the others. It proves that the combinatorial entropy $S$ is bounded from above by a constant. This is reminiscent of the fact that the local entropy function $ent$ is also bounded from above. This lemma will be used when we prove the upper bound of the entropy in Theorem~\ref{thm4.3CKP01}.

\begin{lemma}\label{Slessthan0}
We have a global constant $C$ such that uniformly for any asymptotic fixed or periodic boundary condition of bead model, when $n\rightarrow\infty$, we have
$$\limsup_{n\rightarrow\infty}\frac{S(X)}{n^2}\leq C.$$
\end{lemma}
\begin{proof}
We use the discretization based on Proposition~\ref{beadanddimerpartitionfunction} to prove this lemma. For any fixed or toroidal boundary condition of $N(n)$ beads, consider its discrete version of lozenge tiling of a rectangle-like region $R_{mn,n}$, so we should study
$$\limsup_{n\rightarrow\infty}\lim_{m\rightarrow\infty}\frac{1}{n^2}\Big(\ln Z_{mn,n}-\ln m N(n)\Big).$$

We consider the tiling of $R_{mn,n}$ column by column from left to right. Suppose that the number of horizontal lozenges on the $i^{th}$ column is $N_i$. Given the positions of the $i^{th}$ column, for the $i+1^{th}$ column, the number of possible positions of the columns are the product of the length of the intervals between the neighboring horizontal lozenges on the $i^{th}$ thread, thus at most $\left(\frac{mn}{N_i}\right)^{N_{i+1}}$. Thus we have
$$Z_{mn,n}\leq\prod_{i=1}^n\left(\frac{mn}{N_i}\right)^{N_{i+1}}.$$

We derive that
\begin{eqnarray}\label{zerfhzeqfhkfsdgvd}
\frac{1}{n^2}\Big(\ln Z_{mn,n}-\ln mN(n)\Big)
\leq\frac{1}{n^2}\Big(\sum_{i=1}^n N_i(\ln n -\ln N_{i+1})\Big)
\end{eqnarray}

If $\min_{i=1,...,n}N_i=0$, then the model can be decomposed into independent models (one can be trivial, \emph{i.e.} no bead at all). Without loss of generality we suppose that $\min_{i=1,...,n}N_i\geq 1$.

For any $\varepsilon>0$, we consider the following partition of ${\{1,2,...,n\}}$:
\begin{itemize}
\item set $I_{\leq\varepsilon n}:=\{i:N_i\leq\varepsilon n\}$,
\item set $I_{>\varepsilon n}:=\{i:N_i>\varepsilon n\}$.
\end{itemize}

By construction, the numbers of beads on neighboring threads differ at most by $1$, so for ${i\in I_{\leq\varepsilon n}}$, we have
$${\ln\left(\frac{N_i}{N_{i+1}}\right)\leq\ln 2},$$
so
$$N_i(\ln n -\ln N_{i+1})\leq N_i(\ln n -\ln N_{i})+\ln 2.$$

By the same reason, for ${i\in I_{>\varepsilon n}}$, we have
$$N_i(\ln n -\ln N_{i+1})\leq N_i\left(\ln n -\ln N_{i}+\ln\left(\frac{1+\varepsilon n}{\varepsilon n}\right)\right)\leq
N_i\left(\ln n -\ln N_{i}+\left(\frac{1}{\varepsilon n}\right)\right).$$

Thus, the right hand side of Inequality (\ref{zerfhzeqfhkfsdgvd}) is less than
\begin{eqnarray*}
&\ &\frac{1}{n^2}\left(\sum_{i=1}^{n} N_i\big(\ln n -\ln N_{i})
+\sum_{i\in I_{\leq\varepsilon n}}\ln 2 N_i
+\sum_{i\in I_{>\varepsilon n}}\left(\frac{1}{\varepsilon n}\right)N_i\right)\\
&\leq&\frac{1}{n^2}\left(\sum_{i=1}^{n} N_i\big(\ln n -\ln N_{i})\right)
+ \varepsilon \ln 2+\frac{1}{\varepsilon}O\left(\frac{1}{n}\right),
\end{eqnarray*}
where we use the fact that $N(n)=O(n^2)$. Let $n\rightarrow\infty$, and by the fact that $\varepsilon$ is arbitrarily small, the right hand side of Inequality~(\ref{zerfhzeqfhkfsdgvd}) is less than
\begin{eqnarray*}
\frac{1}{n^2}\Big(\sum_i N_i(\ln n -\ln N_{i})\Big)+o(1),
\end{eqnarray*}
which takes maximum if $N_i$ are almost all equal, thus less than
\begin{eqnarray}\label{sqdfqjksdghsg}
\frac{1}{n^2}\left(n\frac{N(n)}{n}\left(\ln n-\ln\frac{N(n)}{n}\right)\right)+o(1)
={-\frac{N(n)}{n^2}\ln\frac{N(n)}{n^2}}+o(1).
\end{eqnarray}
Let $C$ be any constant bigger than
$\max_{x>0}(-x\ln x) =\frac{1}{e}$
and we have finished the proof.
\qed
\end{proof}

Readers may compare this lemma to the expression of $ent$, and the term ${-\frac{N(n)}{n^2}\ln\frac{N(n)}{n^2}}$ of Equation~(\ref{sqdfqjksdghsg}) corresponds to the term $-t\ln t$ in $ent(s,t)$.

\vspace{0.5cm}

As we have already seen, the dependence of the entropy of a random bead configuration on the boundary condition is more delicate than that of the dimer model. So rather than roughly speaking that one boundary condition is close to a plane, we need to have a more precise definition.

\begin{definition}\label{definitionalmostplanar}
Given a tilt $(s,t)\in]-\frac{1}{2},\frac{1}{2}[\times]0,+\infty[$, for any $n\in\mathbb{N}^*$, a fixed boundary condition $h^{\partial,0}_n$ of a bead model with $n$ threads is called \emph{almost planar} if there exists a plane of tilt $(s,t)$ and $h^{\partial,0}_n$ is chosen to give the best approximation of that plane.
\end{definition}

We remark that being chosen to give the best approximation of a plane implies that on the left and right boundaries of $[-\frac{1}{n-1},1+\frac{1}{n-1}]\times[0,1]$, the distances between neighboring jumps of $h^{\partial,0}_n$ are all equal.

We will equally need its discrete version:

\begin{definition}\label{definitionalmostplanardiscrete}
Given a tilt $(s,t)\in]-\frac{1}{2},\frac{1}{2}[\times]0,+\infty[$, for any $n\in\mathbb{N}^*$ and $m\in\mathbb{N}^*$ large, \emph{an almost planar region} $R_{mn,n}^{0}$ is a tall region tileable by lozenges corresponding to $h^{\partial,0}_n$ constructed as in Section~\ref{sqdfhyzqmurei}.
\end{definition}

According to the construction of $R_{mn,n}^{0}$, the distance between the neighboring cracks on the left and right boundaries of $R_{mn,n}^{0}$ are all equal except for an error smaller than $1$. We also remark that an equivalent way to describe this region is that there exists a parallelogram in $\mathbb{R}^3$ corresponding to the tilt, and the boundary height function of $R_{mn,n}^{0}$ is chosen to fit best to that parallelogram.

By construction, the almost planar bead boundary condition (Definition~\ref{definitionalmostplanar}) is the continuous limit of its discrete version (Definition~\ref{definitionalmostplanardiscrete}). It is also clear that for any tilt $(s,t)$, it is always possible to find at least one boundary function $h^{\partial}_n$ verifying Definition~\ref{definitionalmostplanar} and to find at least one region $R_{mn,n}^0$ verifying Definition~\ref{definitionalmostplanardiscrete}.

\begin{lemma}\label{Lemmaentexist}
For any tilt $(s,t)\in]-\frac{1}{2},\frac{1}{2}[\times]0,\infty[$,\\
\emph{(a)} let $\mathcal{H}_n^t(s,t)$ be the space of bead configurations on the torus with $n$ threads and the height change of $H$
is equal to $(\lfloor n s \rfloor,\lfloor n t \rfloor)$, and let $X_n^{t}(s,t)$ be a randomly chosen element of $\mathcal{H}_n^t(s,t)$, then
$$\liminf_{n\rightarrow\infty}\frac{S\big(X_n^{t}(s,t)\big)}{n^2}>-\infty.$$\\
\emph{(b)} let $\mathcal{H}_n^{0}(s,t)$ be the space of almost planar bead configurations on $D$ with $n$ threads best fitting a plane of tilt $(s,t)$, and let $X_n^{0}(s,t)$ be a randomly chosen element of $\mathcal{H}_n^{0}(s,t)$, then
$$\liminf_{n\rightarrow\infty}\frac{S\big(X_n^{0}(s,t)\big)}{n^2}>-\infty.$$
\end{lemma}
\begin{proof}
The existence of the $\liminf_{n\rightarrow\infty}$ is a result of subadditivity.
\qed
\end{proof}

We give a series of technical lemmas following \cite{CEP,CKP01}.

\begin{lemma}\label{LemmaProp20CEP}
For any two bead models having the same number of threads but with different boundary conditions, consider the un-normalized height function $H$. If the two boundary conditions differ by at most $\Delta$, then on every common vertex, the expected value of these two unnormalized height functions differ by at most $\Delta+2$ under the supremum norm.
\end{lemma}
\begin{proof}
This is a corollary of Proposition 20 of \cite{CEP}'s analog in the case of lozenge tilings.
\qed
\end{proof}

\begin{lemma}\label{lemmaazuma}
For a bead model with $n$ threads, if the average height function is vertically $A$-Liptshitz, then there exists a constant $C>0$ and $C'>0$ depending on $A$ such that for any simply connected region contained in $D$ with given boundary condition of the bead model and two points $(x_1,y_1)$ and $(x_2,y_2)$ in the this region, the probability that $h(x_1,y_1)-h(x_2,y_2)$ differs from its expected height change by more than $\alpha\sqrt{\frac{|x_1-x_2|+|y_1-y_2|}{n-1}}$ is less than $Ce^{-C'{\alpha}^2}$.
\end{lemma}
\begin{proof}
This is the bead-model-version of Theorem 21 and Proposition 22 of \cite{CEP}. To prove this, consider the path
$$\big((n-1)x_1,(n-1)y_1\big)\rightarrow\big((n-1)x_2,(n-1)y_1\big)\rightarrow\big((n-1)x_2,(n-1)y_2\big)$$
in the domain $[0,n-1]\times[0,n-1]$, and consider the height function $(n-1)h(x,y)$ (attention, this is not the un-normalized height function $H$).

Let $N_1=\lfloor (n-1)|x_1-x_2|\rfloor$ be the number of threads between these two points. For the first step (the horizontal step), consider the same martingale as in Theorem 21 of \cite{CEP}, and using Azuma's inequality \cite{AS}, the probability that the difference between the exact height change and expected height change is bigger than $\frac{1}{2}\alpha\sqrt{N_1}$ is less than $C_1e^{-C_2\alpha^2}$ for some constant $C_1,C_2>0$.

Now consider the vertical step. To do this, we turn the space by $\frac{\pi}{4}$ again so that the new space is vertically Lipshitz (see Section~\ref{sect5.4part1}) under the turned coordinates ${(\t{x},\t{y},\t{z})}$ and the surface is ${\t{z}=\t{H}(\t{x},\t{y})}$. Since the space is Lipschitz, we can apply Azuma's inequality again. Take a discretization in $\t{y}$ and for any two points on the same thread and with vertical coordinates $\t{y}_1$, $\t{y}_2$, the number of steps is equal to ${\lfloor (n-1)|\t{y}_1-\t{y}_2|\rfloor}$, and we get a result for $\t{y}$ similar to that for $x$. For any surface $(n-1)h$, turning ${(\t{x},\t{y},\t{z})}$ back into the original space ${(x,y,z)}$ will lead to another difference, but it is bounded by a constant depending on $A$ times the difference of the real and expected height change of $(n-1)h$. Thus, there exist constants $C_3,C_4>0$ such that ${(n-1)h(x,y)}$ changes by more than ${\frac{1}{2}\alpha\sqrt{N_2}}$ is less than $C_3e^{-C_4\alpha^2}$, where ${N_2=\lfloor (n-1)|y_1-y_2|\rfloor}$.

In conclusion, the probability that $h$ and its expected value differs by more than $\alpha\sqrt{\frac{|x_1-x_2|+|y_1-y_2|}{n-1}}$ is at most $Ce^{-C'a^2}$ for well chosen $C$ and $C'$.
\qed
\end{proof}

For any tilt $(s,t)\in]-\frac{1}{2},\frac{1}{2}[\times]0,+\infty[$, any
$\delta>0$ and $n\in\mathbb{N}^*$, define ${\mathcal{U}_{\delta,n}(s,t)}$ as the space of fixed boundary condition of the normalized bead model with $n$ threads, where the boundary functions $h^{\partial}_n$ are in the $\delta$-neighborhood of a plane of tilt $(s,t)$.

\begin{lemma}\label{lemmaProp3.4CKP}
Under the setting above, for $n$ sufficiently large, for any fixed boundary condition ${h^{\partial}_n\in\mathcal{U}_{\delta,n}(s,t)}$, the average normalized height function of the bead model is given within ${\delta +o(1)}$ by that plane, $o(1)$ tending to $0$ when $n\rightarrow\infty$.
\end{lemma}
\begin{proof}
This is a direct corollary of Proposition 3.4 of \cite{CKP01}.
\qed
\end{proof}

\begin{lemma}\label{prop3.6CKP0}
Under the same setting of Lemma~\ref{lemmaProp3.4CKP}, for any $\varepsilon>0$, if we let $X^{h^{\partial}_n}_n$ be any random bead model whose fixed boundary condition is given by ${h^{\partial}_n\in\mathcal{U}_{\delta,n}(s,t)}$, then for $\delta$ sufficiently small and $n$ sufficiently large, we have
$$\frac{S(X^{h^{\partial}_n}_n)}{n^2}\leq\frac{S\big(X_n^{0}(s,t)\big)}{n^2}+\varepsilon,$$
where recall that $X_n^{0}(s,t)$ is the random bead configuration with an almost planar fixed boundary condition.
\end{lemma}

This lemma corresponds to Proposition 3.6 of \cite{CKP01}, where the authors prove that the entropies of the dimer models of nearby boundary conditions are close. As already explained, this is no longer true for the bead model, and the lemma above tells that among all the boundary conditions near a planar, the one that is almost planar has almost the biggest entropy, with an error tending to $0$ when $n\rightarrow\infty$. 


We remark that in the proof below there is a technical assumption. We don't succeed to find a rigorous proof of this point but we have reason to believe that it is true.

\begin{proof}
To prove this lemma, we will look at the discrete version of the bead model. There we use the same idea of \cite{CKP01}, where the authors compare the entropy of two different boundary conditions by applying a coupling-like method between the surfaces. In our case this is more complicated since even a tiny region may have big negative contribution in the entropy, so some more detailed construction is needed.

Given a tilt $(s,t)\in]-\frac{1}{2},\frac{1}{2}[\times]0,+\infty[$, we define $h^{\partial,0}:\partial D\rightarrow\mathbb{R}$ as the linear function fitting a plane of tilt $(s,t)$. For any $n\in\mathbb{N}^*$, we denote by $h^{\partial,0}_n$ an almost planar boundary condition fitting best $h^{\partial,0}$, and for any $m\in\mathbb{N}^*$ large, consider an almost planar region $R_{mn,n}^{0}$ as in Section~\ref{sqdfhyzqmurei}.

Now given any other boundary condition ${h^{\partial}_n\in\mathcal{U}_{\delta,n}(s,t)}$, consider another region $R_{mn,n}$ that corresponds to $h^{\partial}_n$ as in Section~\ref{sqdfhyzqmurei}.

As rising the whole boundary by the same amount doesn't change the entropy, without loss of generality we can suppose that $h^{\partial}_n\geq h^{\partial,0}_n$. We superpose $R_{mn,n}^{0}$ and $R_{mn,n}$ in such a way that their left and right sides are on the same lines (except for the positions of cracks corresponding to the jumps of $h^{\partial,0}_n$ and $h^{\partial}_n$), and the upper and lower boundaries differ by at most $O(\delta n)$.

Define respectively $\mathcal{H}_{mn,n}$ and $\mathcal{H}_{mn,n}^0$ as the space of tilings of $R_{mn,n}$ and of $R_{mn,n}^0$, and a random tiling uniformly chosen respectively from $\mathcal{H}_{mn,n}$ and $\mathcal{H}_{mn,n}^0$ is denoted by $X_{mn,n}$ and $X_{mn,n}^0$. We want to compare the adjusted entropies of $X_{mn,n}$ and $X_{mn,n}^0$. To do this, we use a surface coupling method as in \cite{CKP01} but more delicate (in some sense).

For a given tilt $(s,t)$, we fix some $\rho>0$ such that the $\rho$-neighborhood of $(s,t)$ lies within $]-\frac{1}{2},\frac{1}{2}[\times]0,\infty[$. Define $h^{0,+}:D\rightarrow\mathbb{R}$ as the supremum of the admissible function fitting $h^{\partial,0}$ on $\partial D$ and of tilt within the $\rho$-neighborhood of $(s,t)$ almost everywhere. Such $h^{0,+}$ is a piecewise linear surface on $D$. For any $n$, we let $h^{0,+}_n$ be a bead height function that fits best to $h^{0,+}$ and let $H^{0,+}_{mn,n}$ be a height function of tiling of $R^0_{mn,n}$ which approximates $h^{0,+}$ best when normalized horizontally by $n$ and vertically by $mn$.

For any $r\in]0,\frac{1}{2}[$, for every $H_{mn,n}$, whenever possible, define $\gamma_r(H_{mn,n})$ to be the maximal curve made up by the points of the intersection of $H_{mn,n}$ and $H^{0,+}_{mn,n}$ and enclosing ${D_r=[r,1-r]\times[r,1-r]}$. Here the ``curve" means a path along the edges of lozenges, and ``maximal'' means having the biggest enclosed area.

We decompose the set of tilings of $R_{mn,n}$ by $\gamma_r(H_{mn,n})$. In case that $\gamma_r(H_{mn,n})$ doesn't exist, we just note $\gamma_r(H_{mn,n})=\emptyset$. According to Lemma~\ref{lemmaSdecomposition}, we have the following decomposition for $\gamma\in\{\gamma_r(H_{mn,n}):H_{mn,n}\in \mathcal{H}_{mn,n}\}\cup\{\emptyset\}$:
\begin{eqnarray}\label{eqnSdecomp}
S(X_{mn,n})-\ln m N(n)=\sum_{\gamma}p_{\gamma}\big(-\ln p_{\gamma}+
S(X_{mn,n}|_{\gamma_r(H_{mn,n})=\gamma})-\ln m N(n)\big),
\end{eqnarray}
where $N(n)$ is the number of horizontal tiles in a tiling of $R_{mn,n}$ and $p_{\gamma}$ is the probability that $\gamma_r(H_{mn,n})=\gamma$. We will compare this to $S(X_{mn,n}^0)-\ln m N^0(n)$, where $N^0(n)$ is the number of horizontal tiles in a tiling of $R_{mn,n}^0$.

We first treat the term $\gamma=\emptyset$ and fix $r=\frac{2\delta}{\rho}$ as a function of $\delta$. The probability that $\gamma=\emptyset$ is less than the probability that there is some point on $\partial D_r$ such that on the corresponding point in the discrete version we have $H_{mn,n}>H^{0,+}_{mn,n}$.

By Lemmas~\ref{LemmaProp20CEP} and~\ref{lemmaazuma}, on any such point this probability is exponentially small in $n$. Moreover, there are only $O(n)$ points that need to be checked: on the upper and lower sides there are only $O(n)$ point, and on the other two sides it suffices to check $O(n)$ with fixed distance between neighboring ones. The unit distance should be small enough depending on $\rho$ so that if two neighboring points verify the condition then on the whole interval the same condition is automatically verified. Thus, the total probability tends to $0$ when $\delta\rightarrow 0$ and $n\rightarrow\infty$, and the term $-p_{\gamma}\ln p_{\gamma}$ tends to $0$ too.

By Lemma~\ref{Slessthan0} the remaining part is bounded from above by a global constant times the area, so in conclusion, we can choose $\delta$ small enough so that this term is less than $\frac{\varepsilon}{4}$ for any $n$ large enough.

We now restrict ourselves to the case where $\gamma\neq\emptyset$. For any $\gamma$, denote respectively the number of horizontal lozenges on the curve by $N^{\gamma}$, the number of lozenges not enclosed by $\gamma$ by $N^{\gamma}_{out}$ and the number of lozenges enclosed by $\gamma$ by $N^{\gamma}_{in}$.
Conditioned to $\gamma$, the tiling of the regions inside and outside $\gamma$ are independent, so we can write every term (corresponding to $\gamma$) in the sum on the right hand side of~(\ref{eqnSdecomp}) as a sum of:\\
(a) $p_{\gamma}$ times the adjusted entropy of a tiling of the region enclosed by $\gamma$,\\
(b) that of a tiling of the region not enclosed by $\gamma$,\\
(c) $p_{\gamma}(-\ln p_{\gamma}-\ln m N^{\gamma}).$

We take the following technical assumption: we assume that for $\delta$ small enough, when $n\rightarrow\infty$ and $m\rightarrow\infty$ (depending on $n$), the term (c) summed over all $\gamma$ and normalized by $n^2$ will be finally smaller than $\frac{\varepsilon}{4}$. In fact, the sum over all $\gamma$ of (c) can be viewed as an expectation, and
we consider a typical boundary. If on the left piece there are $N_l$ horizontal lozenges, and we suppose that the winding contributes not too much so the left piece behaves as a lazy random walk with fixed number of moves, starting position and ending position. The way to take this piece is around $\dbinom{nm}{\frac{N_l}{2}}^2$, so typically the probability is of order $\dbinom{mn}{N^{\gamma}}^{-1}$, and
$$\frac{1}{n^2}\mathbb{E}\Big[\sum_{\gamma}p_{\gamma}(-\ln p_{\gamma}-\ln m N^{\gamma})\Big]$$
should be of order $\frac{\ln n}{n}$. By this argument, our assumption seems to be reasonable, but we wish to find a way to make this argument rigorous.

For terms (a) and (b), we consider the following subspaces of the tiling of $R_{mn,n}^0$: for every given $\gamma$, define
\begin{eqnarray}
\mathcal{H}^{0,+}_{mn,n}(\gamma)=\{H^0_{mn,n}\in \mathcal{H}^0_{mn,n} :H^0_{mn,n}|_{\gamma}=H^{0,+}_{mn,n}|_{\gamma}\}.
\end{eqnarray}
Denote by $X^{0,+}_{mn,n}(\gamma)$ a random tiling uniformly chosen in this space. We prove that the normalized and adjusted entropy  $$\frac{1}{n^2}\Big(S(X^{0,+}_{mn,n}(\gamma))-\ln m(N^0(n)-N^{\gamma})\Big)$$
is at least not much smaller than
$$\frac{1}{n^2}\Big(S(X_{mn,n}|_{\gamma_r(H_{mn,n})=\gamma}-\ln m (N(n)-N^{\gamma})\Big).$$


In fact, both of them can be written as a sum of the adjusted and normalized entropy on the region enclosed by $\gamma$ and that on the region not enclosed by $\gamma$. Obviously their contributions of the region enclosed by $\gamma$ in the entropy are equal. On the region not enclosed by $\gamma$, by Lemma~\ref{Slessthan0} we have
$$\frac{1}{n^2}\Big(S(X^{out}_{mn,n}|_{\gamma_r(H_{mn,n})=\gamma})-
\ln m N^{\gamma}_{out}\Big)$$
is less than a global constant $C$ times the area of region, which tends to $0$ when $\delta\rightarrow 0$ (so $r\rightarrow 0$). Here $X^{out}_{mn,n}|_{\gamma_r(H_{mn,n})=\gamma}$ is the conditioned random tiling outside the region enclosed by $\gamma$ and of boundary condition $\partial R_{mn,n}$.

Meanwhile, if let $X^{out,0,+}_{mn,n}|_{\gamma_r(H_{mn,n})=\gamma}$ be the conditioned random tiling outside the region enclosed by $\gamma$ and of boundary condition $\partial R_{mn,n}^0$, then for
$$\frac{1}{n^2}\Big(S(X^{out,0,+}_{mn,n}|_{\gamma_r(H_{mn,n})=\gamma})-\ln m N^{\gamma,0}_{out}\Big),$$
where $N^{\gamma,0}_{out}=N^{0}(N)-N^{\gamma}-N^{\gamma}_{in}$ is the number of horizontal lozenges outside $\gamma$, its boundary condition restricted on $\partial R^0_{mn,n}$ and $\gamma$ fits best to a piecewise linear function $h^{0,+}$. We can decomposed the region into a union of disjoint squares, and by Lemma~\ref{Lemmaentexist}, the adjusted normalized ventropy normalized entropy is bounded from below by some constant (depending on $(s,t)$ and $\rho$) times the area of this region when $n\rightarrow\infty$ and $m\rightarrow\infty$ depending on $n$. Since when $\delta\rightarrow 0$, the normalized area of the region between curve $\gamma$ and $\partial R_{mn}^0$ also tends to $0$, as conclusion, for $\delta$ small enough, for $n$ big enough and for $m$ big enough, we have
$$\frac{1}{n^2}\Big(S(X^{out,0,+}_{mn,n}|_{\gamma_r(H_{mn,n})=\gamma})-\ln m N^{\gamma,0}_{out}\Big)>-\frac{\varepsilon}{4}.$$

Finally, since the space $\mathcal{H}^{0+}_{mn,n}|_{\gamma_r(H_{mn,n})=\gamma}$ is a subspace of $\mathcal{H}^0_{mn,n}$, the normalized adjusted entropy of  $X^{0+}_{mn,n}|_{\gamma_r(H_{mn,n})=\gamma}$ for every $\gamma$ is less than that of $X^0_{mn,n}$. Together with the technical assumption on (c), in conclusion we have: for $\delta$ small enough, $n$ large enough, we have
\begin{eqnarray*}
\sum_{\gamma}p_{\gamma}\big(-\ln p_{\gamma}+
S(X_{mn,n}|_{\gamma_r(H_{mn,n})=\gamma})-\ln m N(n)\big)
< \frac{1}{n^2}\Big(S(X^0_{mn,n})-\ln m N^0(n)\Big)+\varepsilon.
\end{eqnarray*}
\qed
\end{proof}

Lemma~\ref{prop3.6CKP0} proves that among the fixed boundary conditions that are close to a plane, the almost planar one has almost the biggest entropy. As a corollary, we have the following theorem. Recall that for bead models with $n$ threads, $X_n^{t}(s,t)$ is the random bead configuration of toroidal boundary condition given by tilt $(s,t)$, and $X_n^{0}(s,t)$ is that of almost planar fixed boundary condition.

\begin{theorem}\label{thm-ent}
For any tilt $(s,t)\in]-\frac{1}{2},\frac{1}{2}[\times]0,+\infty[$, we have
$$\lim_{n\rightarrow\infty}\frac{S\big(X_n^{t}(s,t)\big)}{n^2}
=\lim_{n\rightarrow\infty}\frac{S\big(X_n^{0}(s,t)\big)}{n^2}=ent(s,t).$$
Moreover, for any $\delta>0$, $n\in\mathbb{N}^*$, if we consider the union of bead models with fixed boundary conditions taken in ${\mathcal{U}_{\delta,n}(s,t)}$, then the combinatorial entropy of a random bead configuration in this set normalized by $n^2$ is also equal to $ent(s,t)+o(1)$ when $n\rightarrow\infty$ and $\delta\rightarrow 0$.
\end{theorem}

If we take any two fixed boundary function of ${\mathcal{U}_{\delta,n}(s,t)}$, they do not necessarily have the same number of beads. So to define the adjusted combinatorial entropy for the union of boundary conditions in ${\mathcal{U}_{\delta,n}(s,t)}$, we need Definition~\ref{sqdsdfkufhrulgerkzzer}, which a priori furthermore asks fixing the probability that a random bead configuration has some given number of beads. However, in the proof below, we show that the choice of the probability doesn't affect the limit of the normalized adjusted entropy.

\begin{proof}
By Lemmas~\ref{Slessthan0} and \ref{Lemmaentexist}, for any tilt $(s,t)$, $\frac{S\big(X_n^{t}(s,t)\big)}{n^2}$ is bounded, so there exists a subsequence $n_k$ of $n$, and along this subsequence, for every $n_k$ we can choose an almost planar boundary conditions whose normalized entropies as sequence in $n_k$ converge.

Lemma~\ref{prop3.6CKP0} proves that among all nearby boundary conditions the almost planar one has the almost biggest normalized entropy. In particular, this implies that for any given tilt $(s,t)$ and any $n\in\mathbb{N}^*$ big enough, two almost planar boundary conditions have close entropy. Thus, the convergence along $n_k$ in the last paragraph doesn't depend on the choice of the precise almost planar boundary condition.

Moreover, this convergence is not just for a subsequence of $n_k$ but a convergence in $n$. In fact, in the following we prove that along $n_k$, the normalized adjusted entropies converge to $ent(s,t)$. As this is also true for any subsequence of $n$, we conclude that the normalized adjusted entropy converge as $n\rightarrow\infty$. Thus, without loss of generality, in the following we only consider a sequence in $n$.

Now for every $n\in\mathbb{N}^*$ fix the sequence $N(n)\propto n^2$ and consider the bead models with $n$ threads, $N(n)$ beads and with boundary conditions be any function in ${\mathcal{U}_{\delta,n}(s,t)}$. We claim that the normalized entropy of this sequence of models
converges to the same limit of $\frac{S(X_n^{0}(s,t))}{n^2}$ when $\delta\rightarrow0$ and $n\rightarrow\infty$. This claim corresponds to the second part of this theorem.

We first suppose that upper, lower and right boundaries are fixed and the left boundary boundary is free within $\delta$ neighborhood of the almost planar one. The number of beads on the left boundary is fixed (by the given upper and lower boundaries) and we denote it by $K$. For all $m$, consider the discrete version where we tile $R_{mn,n}$ by lozenges. There are $\binom{mn}{K}$ different possibilities, and by Stirling's formula
\begin{eqnarray*}
&\ &\frac{\ln \binom{mn}{K} - K\ln m}{n^2}\\
&=&\frac{1}{n^2}\Big[\ln\Big(\frac{\sqrt{2\pi mn}}{\sqrt{2\pi(mn-K)2\pi K}}\frac{(mn)^{mn}}{(mn-K)^{mn-K}K^K}\Big)+o(1)-K\ln m\Big]\\
&=&\frac{1}{n^2}\Big[\ln\sqrt{\frac{mn}{2\pi(mn-K)K}}-K\ln\frac{K}{n}+(mn-K)\ln\frac{mn}{mn-K}+o(1)\Big]\\
&=&\frac{1}{n^2}\Big[\ln\sqrt{\frac{mn}{2\pi(mn-K)K}}-K\ln\frac{K}{n}+\frac{mn-K}{mn}K+o(1)\Big]\\
&=&O(\frac{1}{n}),
\end{eqnarray*}
where $o(1)$ is for $m$ big enough. Thus, the entropy of the bead model of free left boundary will be at most $O(\frac{1}{n})$ bigger than that of the fixed almost planar one. It is not hard to show that for any of other three boundaries there is a similar result. Thus, if the number of $N(n)$ is fixed, then our claim is true.

Now we allow $N(n)$ to vary but under the constraint that the boundary functions are within ${\mathcal{U}_{\delta,n}(s,t)}$. For any $n$ and $\delta$, denote by $\mathcal{N}=\mathcal{N}(s,t,\delta,n)$ be the set of possible $N(n)$, then $|\mathcal{N}|$ is of order $O(2\delta n^2)$. According to Definition~\ref{sqdsdfkufhrulgerkzzer}, if for any $N_i\in\mathcal{N}$, the probability that $N(n)=N_i$ is given and equal to $p_{N_i}$, then the entropy of the bead configurations with boundary conditions taken in ${\mathcal{U}_{\delta,n}(s,t)}$ is equal to
\begin{eqnarray}
-\sum_{N_i\in\mathcal{N}} p_{N_i}\ln p_{N_i} +\sum_{N_i\in\mathcal{N}} p_{N_i} S_{N_i},
\end{eqnarray}
where $S_{N_i}$ is the entropy of the model whose number of beads is equal to $N_i$. Since we have proved that $\frac{S_{N_i}}{n^2}$ is at most $\frac{S(X_n^{0}(s,t))}{n^2}+o(1)$, and $\frac{-\sum_{N_i\in\mathcal{N}} p_{N_i}\ln p_{N_i}}{n^2}$ is at most of order $\frac{\ln n}{n^2}$, we have proved our claim that the normalized entropies of the bead configurations with ${\mathcal{U}_{\delta,n}(s,t)}$-boundary condition converge to the same limit of $\frac{S(X_n^{0}(s,t))}{n^2}$ when $n\rightarrow\infty$ and $\delta\rightarrow 0$.

It remains to prove the first part of this theorem.

First, by construction, we can find an almost planar boundary condition whose opposite sides matches. So we have
$$\lim_{n\rightarrow\infty}\frac{S\big(X_n^{t}(s,t)\big)}{n^2}\geq
\lim_{n\rightarrow\infty}\frac{S\big(X_n^{0}(s,t)\big)}{n^2}.$$

On the other hand, given a toroidal boundary condition (so the number of beads is fixed), consider the fixed boundary conditions of ${\mathcal{U}_{\delta,n}(s,t)}$ that yield the same number of beads. Since the bead configuration with periodic boundary condition not included in ${\mathcal{U}_{\delta,n}(s,t)}$ has a negligible contribution, by our claim proved above, we have
$$\lim_{n\rightarrow\infty}\frac{S\big(X_n^{t}(s,t)\big)}{n^2}\leq
\lim_{n\rightarrow\infty}\frac{S\big(X_n^{0}(s,t)\big)}{n^2}.$$

Finally, as for any $\delta>0$, ${\mathcal{U}_{\delta,n}(s,t)}$ also includes the almost planar boundary conditions of tilts near $(s,t)$, by applying a similar argument as above we conclude that for $(s',t')$ close to $(s,t)$, $\frac{S\big(X_n^{t}(s',t')\big)}{n^2}$ is also close to $\frac{S\big(X_n^{0}(s,t)\big)}{n^2}$. Thus, $\frac{S\big(X_n^{t}(s,t)\big)}{n^2}$ is continuous in the tilt. This allows us to use the Legendre transform in Section~\ref{surfacetensionlocalentropyfunction}, so we have proved this theorem.
\qed
\end{proof}

\begin{definition}
For any $\varepsilon>0$ and for any $(s,t)\in]-\frac{1}{2},\frac{1}{2}[\times]0,1[$, define
$$\rho^{\varepsilon}(s,t)=\sup_{\rho>0}\left\{\rho: \left(||(s',t')-(s,t)||<\rho\right)\Rightarrow \left(|ent(s,t)-ent(s',t')|<\varepsilon\right) \right\}.$$
\end{definition}

The following lemma is a direct corollary of Lemma~\ref{prop3.6CKP0} and Theorem~\ref{thm-ent}.
\begin{lemma}\label{prop3.6CKP}
Consider the unit square $D$. For any $\varepsilon>0$, and for any tilt
$$(s,t)\in]-\frac{1}{2},\frac{1}{2}[\times]0,+\infty[,$$
consider the bead model on $D$ with $n$ treads and with the fixed boundary condition fitting to a plane of tilt $(s,t)$ within $\rho^{\varepsilon}(s,t)$. Then for $n$ sufficiently large, the entropy $S$ of the bead configurations normalized by $n^2$ is at most the entropy of a bead model with an almost planar periodic boundary condition $h^{\partial,0}_n$ plus $\varepsilon+o(1)$ where $o(1)$ tends to $0$ when $n\rightarrow\infty$.
\end{lemma}

\begin{prop}\label{prop-ent}
Lemma~\ref{prop3.6CKP0} and Theorem~\ref{thm-ent} hold if the region is an isosceles right triangle instead of a square.
\end{prop}
\begin{proof}
The proof is exactly the same as that of Corollary 4.2 of \cite{CKP01}. An isosceles right triangle can be approached from interior by a union of squares, and combining two triangles gives a square. These operations naturally yield a lower bound and an upper bound of the entropy of a bead model on a isosceles right triangle, which is both equal to $ent(s,t)+o(1)$ when $n\rightarrow 0$.
\end{proof}

The following Lemma is another version of Lemma~\ref{triangulation} which we will see is related to an upper bound of the entropy $S$.
\begin{lemma}\label{triangulationprime}
For any admissible function $h$ such that $Ent(h)>-\infty$ and for any $\varepsilon_1,\varepsilon_2>0$, for $l>0$ sufficiently small, then the piecewise linear function $h'$ on the $l$-right-triangle mesh verifies the following two properties.\\
\emph{(a)} For all but a fraction of $\varepsilon_1$ of the triangles in the mesh, for every triangle, denote the tilt of $h'$ on that triangle by $(s,t)$, then the function $h$ is within $\rho^{\varepsilon_2}(s,t)l$ of $h'$.\\
\emph{(b)} $Ent(h')<Ent(h)+\varepsilon_2$.
\end{lemma}
\begin{proof}
The proof of part (a) is the same as in Lemma 2.2 of \cite{CKP01}. We now prove (b). Define the space of possible tilts as
$$V_0=[-\frac{1}{2},\frac{1}{2}]\times[0,+\infty[,$$
and for any $A>0$, $d>0$, define the following subset of $V_0$:
$$V_0^{A,d}=\{(s,t):|s-\frac{1}{2}|<d,\text{ or }|s+\frac{1}{2}|<d,\text{ or }t>A\}.$$

Since $Ent(h)>-\infty$ and $ent(.,0)=0$, we can take $A$ sufficiently large and $d$ sufficiently small so that the points $$\left\{(x,y):\left(\frac{\partial h}{\partial x},\frac{\partial h}{\partial hy}\right)(x,y)\not\in V_0^{A,d}\right\}$$
gives a contribution of absolute value less than $\frac{\varepsilon_2}{4}$ in $Ent(h)$.

Let $V_1,V_2,...,V_n$ be a open cover of ${V_0\backslash V_0^{A,d}}$ such that within each set $V_i$ the function $ent(s,t)$ changes at most by $\frac{\varepsilon_2}{4}$. For any $i\in\{1,2,...,n\}$ and for any $\eta_i\in]0,1[$, consider the set ${\bar{S}(V_i,\eta_i)}$ which is composed of possible tilts ${(\bar{s},\bar{t})}$ that there exists a probability density function (in the sense of distribution) on $V_0$ such that the average slope is equal to $(\bar{s},\bar{t})$, and a proportion bigger than $\eta_i$ is in $V_i$. This gives a family of convex subsets of $V_0$ indexed by $\eta_i$. When $\eta_i\rightarrow 1$, the set ${\bar{S}(V_i,\eta_i)}$ tends to ${V_i+\{0\}\times\mathbb{R}^+}$, where the sum of two sets is defined as the set of the sums of any pair of elements.

By the property of $ent$, for $\eta_i$ close enough to $1$ we have that for any average tilt ${(\bar{s},\bar{t})\in\bar{S}(V_i,\eta_i)}$,
$$ent(\bar{s},\bar{t})\leq\sup_{(s,t)\in V_i}ent(s,t)+\frac{\varepsilon_2}{8},$$
and
$$(1-\eta_i)\inf_{(s,t)\not\in V_0^{A,d}}ent(s,t)\geq-\frac{\varepsilon_2}{4}.$$

Now we can apply an argument of metric density from \cite{RealAndComplexAnalysis} similar to the way \cite{CKP01} uses it. For any $\varepsilon'>0$, $\eta_i>0$, if $l_i$ is sufficiently small, then for any $\delta\leq l_i$, on all but an $1-\varepsilon'$ fraction of the points $(x,y)$ such that $\left(\frac{\partial h}{\partial x},\frac{\partial h}{\partial y}\right)\in V_i$, at least a $\eta_i$ fraction of the ball centered at $(x,y)$ and of radius $\delta$ lies in $V_i$.

If there is some triangle where $h$ verifies $(a)$ for some $\varepsilon'$, the tilt of the piecewise linear function $h'$ differs from the average tilt on that triangle by at most $2\varepsilon'$. Take $\varepsilon'$ less than $\frac{\varepsilon_1}{2}$ such that for all $i$ and for all $(s,t)$ in the $2\varepsilon'$ neighborhood of $\bar{S}(V_i,\eta_i)$ we have
\begin{eqnarray}\label{dsfllsihgdfgghserezrze}
ent(s,t)\leq\sup_{(s,t)\in V_i}ent(s,t)+\frac{\varepsilon_2}{4}.
\end{eqnarray}
Also, for $\varepsilon'$ small enough, the integral of $ent\circ\nabla h$ is bigger than $-\frac{\varepsilon_2}{4}$ on any subset of $D$ whose measure is less than $2\varepsilon'$.

For all $l\leq\min_i\{l_i\}$ and less than the $l$ in (a) where we replace $\varepsilon_1$ by some $\varepsilon'$ less than $\varepsilon_1$ and verifying the conditions above, on at least a $1-\varepsilon'$ fraction of the triangles, (a) is verified.

Now compare $Ent(h)$ to $Ent(h')$ where $h'$ is the piecewise linear function on the $l$-mesh. There is at least a $1-2\varepsilon'$ fraction of triangles such that for each triangle, there exists $i$ such that in this triangle a proportion of at least $\eta_i$ of points $(x,y)$ verifies that $ent\circ\nabla h(x,y)$ is contained in $V_i$, thus the average slope of $h$ is in $\bar{S}(V_i,\eta_i)$. Meanwhile, as the tilt of $h'$ lies within $2\varepsilon'$-neighborhood of the average slope of $h$, according to (\ref{dsfllsihgdfgghserezrze}) we have that on this triangle
$$ent\circ\nabla h'\leq\sup_{(s,t)\in V_i}ent(s,t)+\frac{\varepsilon_2}{4}\leq\inf_{(s,t)\in V_i}ent(s,t)+\frac{\varepsilon_2}{2}.$$

In conclusion, we compare $Ent(h)$ and $Ent(h')$ respectively for the following two cases:
\begin{itemize}
\item On the $2\varepsilon'$ fraction of triangles and on the points in the $1-2\varepsilon'$ fraction of triangles where ${ent\circ\nabla h(x,y)\in V_0^{A,d}}$:
 \begin{itemize}
 \item the integral of $ent\circ\nabla h$ is bigger than $-\frac{\varepsilon_2}{2}$ by construction.
 \item the integral of $ent\circ\nabla h$ is less than $0$ by negativity.
 \end{itemize}
\item On the $1-2\varepsilon'$ fraction of triangles and where ${ent\circ\nabla h(x,y)\not\in V_0^{A,d}}$, for each triangle, there exists $i$ such that a proportion bigger than $\eta_i$ of points is in $V_i$. The contribution of the other $(1-\eta_i)$ proportion of points in $Ent(h)$ is most $-\frac{\varepsilon_2}{4}$ times the area. On other points,
    $$ent\circ\nabla h'<ent\circ\nabla h+\frac{\varepsilon}{4},$$
    so the contribution of these points in $Ent(h')-Ent(h)$ is at most $\frac{\varepsilon}{2}$.
 \end{itemize}
Thus we have proved the lemma.
\qed
\end{proof}

\vspace{0.5cm}

Now we can prove our main theorems of this section.

\noindent\textit{Proof of Theorem~\ref{thm4.3CKP01}.  }
We will separately prove that $Ent(h)$ is asymptotically the upper bound and lower bound of the normalized entropy on the left hand side of~(\ref{dfserfew}).

We begin by the part of lower bound. For any $\varepsilon>0$, by Lemma~\ref{lemmaapproximation}, we can find some $\tilde{h}$ such that $||\tilde{h}-h||_{L^{\infty}}<\frac{\delta}{4}$, $|Ent(\tilde{h})-Ent(h)|<\frac{\varepsilon}{4}$, and there exists some $K$ such that $ent\circ\nabla\tilde{h}>-K$ on $D$ (in other words $\tilde{h}\in \mathcal{H}^K$). By Lemma~\ref{triangulation}, for any $l$ small enough, we can construct a $l$-isosceles-right-triangle mesh and find a function $h'$ such that on every triangle of the mesh $h'$ is linear and $||h'-\tilde{h}||_{L^{\infty}}<\frac{\delta}{4}$ and $|Ent(h')-Ent(\tilde{h})|<\frac{\varepsilon}{4}$.

By Theorem~\ref{thm-ent} and Proposition~\ref{prop-ent}, on any triangular of the mesh, when $n\rightarrow\infty$, the entropy normalized by $n^2$ of the bead model with fixed almost planar boundary condition fitting the boundary of triangle converges to the contribution of this triangle in $Ent(h')$, and the configurations whose maximal height difference from $h'$ is bigger than $\frac{\delta}{4}$ is exponentially small in $n$. The fixed boundary conditions of the triangles together with the control on the maximal height difference gives a lower bound of $S(X_{n}^{\mathcal{V}_{\delta}(h)})$, so as conclusion we prove that for any $\delta$ and for $n$ small enough,
$$\frac{S(X_{n}^{\mathcal{V}_{\delta}(h)})}{n^2}\geq Ent(h)-\varepsilon.$$

Now we prove the upper bound. For any $\varepsilon>0$, since $h$ has no atom and $Ent(h)>-\infty$, there exists $\varepsilon_1$ such that for any subset of $D$ of Lebesgue measure less than $\varepsilon_1$, the integral of $ent\circ\nabla h$ on that set is bigger than $-\frac{\varepsilon}{4}$. By Lemma~\ref{triangulationprime}, for $l>0$ small enough, the piecewise linear function $h'$ on the $l$-right-triangle mesh satisfies that\\
(a) for at least a fraction of $1-\varepsilon_1$ of triangles in the mesh, on every triangle, the function $h$ is within $\rho^{\frac{\varepsilon}{2}}(s,t)l$ of $h'$ where $(s,t)$ is the tilt of that triangle.\\
(b) $Ent(h')<Ent(h)+\frac{\varepsilon}{4}$.

Lemma~\ref{Slessthan0} says that the at most $\varepsilon_1$ fraction of triangles, the entropy $S$ is at most $C$ times the area of the triangles, and Theorem~\ref{thm-ent} says that on every triangle, if the tilt $h'$ is $(s,t)$ there, then the normalized entropy of all the configurations whose height on the boundary of the triangle is within $\rho^{\frac{\varepsilon}{2}}(s,t)l$ is less than $ent(s,t)+\frac{\varepsilon}{2}+o(1)$ times the area of the triangle, $o(1)$ converging to $0$ when $n$ tends to infinity. Summing this gives an upper bound of entropy, which is less than $Ent(h)+\varepsilon+o(1)$. This finishes the proof.
\qed

The above large-deviation theorem naturally yields the following theorem about the convergence of a random bead configuration.

\begin{theorem}\label{Thm-Bead-convergenceinproba}
Given an asymptotic boundary condition function $h^{\partial}$ defined on $\partial D$ which is constant if restricted to $x=1$ or $x=0$, for any $n\in\mathbb{N}^*$, consider the bead model on $D$ with $n$ threads and with fixed boundary condition that approximates best $h^{\partial}$. Then the normalized height function $h$ converges (under the uniform norm) in probability when $n\rightarrow\infty$ to an admissible function $h_0$, which is the unique maximizer of $Ent(.)$.
\end{theorem}
\begin{proof}
Theorem~\ref{thm4.3CKP01} proves that for any admissible function $h:D\rightarrow\mathbb{R}$ such that $Ent(h)>-\infty$, for any $\delta>0$, when $n\rightarrow\infty$, $\frac{S\big(X_n^{\mathcal{V}_{\delta}(h)}\big)}{n^2}$ converges to $Ent(h)$ when $n\rightarrow\infty$. We should also take the functions that $Ent(h)=-\infty$ into consideration.

If $h\in \mathcal{H}_0$, it is easy to see that Lemma~\ref{triangulation} and the upper bound part of Theorem~\ref{thm-ent} still apply. Thus, for any $h\in \mathcal{H}_0$ such that $Ent(h)=-\infty$, we have
$$\lim_{\delta\rightarrow 0}\lim_{n\rightarrow\infty}\frac{S\big( X_n^{\mathcal{V}_{\delta}(h)}\big)}{n^2}=-\infty.$$

If $h\not\in \mathcal{H}_0$, we consider the turned space $\tilde{\mathcal{H}}$ under the uniform norm. By definition, if $h\not\in \mathcal{H}_0$, then there exists a subset of ${(\t{x},\t{y})}$ with Lebesgue measure $\mu>0$ where $\frac{\partial\t{h}}{\partial\t{y}}=1$. By the same argument of metric density used in Lemma~\ref{triangulationprime}, for all $\varepsilon>0$ small enough, there exists a subset of $\t{D}$ as a union of disjoint squares such that on every square the average vertical slope is bigger than $1-\varepsilon$ and the measure of this subset is bigger than $\mu-\varepsilon$. It is not hard to see that if we take $\varepsilon>0$ arbitrarily small, then for $\t{\delta}$ small enough, the entropy within the $\t{\delta}$-neighborhood in $\t{\mathcal{H}}$ of $\t{h}$ can be arbitrarily small.

An open set of admissible functions in the original height function space $\mathcal{H}$ is also an open set in the turned space $\t{\mathcal{H}}$, and the turned space $\t{\mathcal{H}}$ is compact under the uniform norm. Thus, from any open cover of the admissible functions we can choose a finite cover. By the definition of entropy, if we consider all the bead configurations with the same fixed boundary condition, then for any $\delta>0$, any admissible function $h$ such that $Ent(h)>-\infty$ and for $n$ large enough, the probability that a random bead configuration is in $\mathcal{V}_{\delta}(h)$, which by definition is equal to the proportion of the volume of this set with respect to the volume of the whole set of possible configurations, is proportional to $e^{Ent(h)n^2}$. When $n\rightarrow\infty$, the probability that $h$ is within the neighborhood of $h_0$ dominates the other possibilities, and we have proved the theorem.
\qed
\end{proof}

\section{Solutions of the entropy maximizing problem}\label{solution}

In this section, we will characterize $h_0$, the solution of the variational principle. The variational principle naturally yields a Euler-Lagrange equation of the limit shape $h_0$: since $ent$ is smooth, to maximize the integral of $ent$ over a region with given boundary condition, the height function should satisfies the equation:
$$\text{div}\nabla ent\circ\nabla h=0,$$
which implies
\begin{eqnarray}\label{pde}
\pi^2(1+\tan^2(\pi h_x)) h_y h_{xx}+\frac{h_{yy}}{h_{y}}+2\pi h_{xy}\tan(\pi h_x)=0.
\end{eqnarray}

However, in general it is hard to solve Equation~(\ref{pde}) directly, and we hope to have a systematical way to find the solutions. A possible option is applying directly the results of \cite{KO1} to the bead model, where the authors prove that finding the solution $h$ of the Euler-Lagrange equation can be done via finding and solving a system of algebraic equations. To do so, we prove in Theorem~\ref{thm-hmcvg2h0} that the maximizer of the bead model is a properly normalized limit of those of the dimer models.

This theorem can be summarized by a commutative diagram~(\ref{commutativediagram}) here below. For any given asymptotic fixed boundary condition $h^{\partial}$ defined on $\partial D$ and constant if restricted to $x=0$ or $x=1$, for any $n$, we consider the bead model with $n$ threads and an almost planar boundary condition $h^{\partial}_n$. Moreover, for any $m$ big enough we consider $R_{mn,n}$ as the domain constructed in Section~\ref{sqdfhyzqmurei}. We have:
\begin{eqnarray}\label{commutativediagram}
\begin{matrix}
  \text{Lozenge tiling of }R_{mn,n}, & \xrightarrow[m\rightarrow\infty]{}  & \text{Bead configuration with }n\text{ threads},\cr
  \Bigg\downarrow n\rightarrow\infty & \circlearrowright                   & \Bigg\downarrow n\rightarrow\infty \cr
  \text{Limit shape of a uniformly} & \xrightarrow[m\rightarrow\infty]{}  & \text{Limit shape of the bead model}. \cr
  \text{chosen tiling of }R_{mn,n},
\end{matrix}
\end{eqnarray}
This result seems quite natural as the bead model is a continuous scaling limit of the dimer model. However, it is not trivial since there is no theory yet that ensures the commutativity of the limit in $m$ (from dimer models to bead models) and that in $n$ (from finite cases to asymptotic limit).

\vspace{0.5cm}

For every $R_{mn,n}$, rather than considering $m\rightarrow\infty$ while keeps $n$ as when we defined the bead model, here we consider the limit $n\rightarrow\infty$ while keeping the asymptotic shape of the region. Let $R_m$ be the region $R_{mn,n}$ normalized by $n$. Define
$$\sigma:=\sup\{|y_1-y_2|:(x_1,y_1),(x_2,y_2)\in R_m\}-m.$$
In other words, the height of the region $R_m$ is equal to $m+\sigma$. Thus, if we vertically normalize $R_m$ by $m+\sigma$, then the new region, denoted by $D^m$, fits inside the unit square $D$.

The boundary condition of $D^m$ also naturally yields a boundary condition of $D$ by vertically extending the boundary height function of $D^m$, \emph{i.e.} for $x\in[0,1]$,
\begin{eqnarray*}
h(x,1)=h(x,\sup\{y:(x,y)\in D^m\}),\\
h(x,0)=h(x,\inf\{y:(x,y)\in D^m\}),
\end{eqnarray*}
while $h(0,y)$ and $h(1,y)$ are constant.

The (dimer) admissible function on $R_m$, defined as the closure of the height function $H$ of lozenge tilings normalized to $D^m$ as above (see Figure~\ref{heightfunctionlozengetiling}), forms the space of functions on $R_m$ which are horizontally $\frac{1}{2}$-Lipschitz, vertically non-decreasing and $1$-Lipschitz. If naturally extended from $D^m$ to the whole of $D$, they forms such following subspace of functions $\mathcal{H}_0$: define
\begin{eqnarray*}
\bar{\mathcal{H}}_m=\left\{h\in \mathcal{H}_0:\frac{\partial h}{\partial y}\Big|_{D\backslash D^m}=0,\ h\text{ is }(m+\sigma)\text{-Lipschitz}\right\}.
\end{eqnarray*}
It is easy to see that $(\bar{\mathcal{H}}_m)_m$ form an increasing subsequence exhausting $\mathcal{H}_0$ when $m\rightarrow\infty$.

Recall that $ent^{\diamond}$ as the local entropy function of the dimer model on the hexagon lattice. Considering Proposition~\ref{entropyexchange}, we define $ent_m$ as the normalized and adjusted local entropy function of the dimer model, \emph{i.e.},
\begin{eqnarray}\label{entropyasym}
ent_m(s,t)=(m+\sigma)\ ent^{\diamond}(s,t/(m+\sigma))-\ln(m+\sigma)t,
\end{eqnarray}
and for any $h\in\bar{\mathcal{H}}_m$ define
$$Ent_m(h)=\int_D ent_m\circ\nabla h\ dxdy.$$

Recall that Proposition~\ref{entropyexchange} says that the right side of Equation~(\ref{entropyasym}) converges to $ent(s,t)$ for any $(s,t)$ and the convergence is uniform on any compact of slopes that doesn't contain exploding points. We also remark that the concavity of $ent$ simply implies the concavity of $ent_m$.

By \cite{CKP01}, for any $m\in\mathbb{N}^*$, there exists a unique height function $\bar{h}_m\in\bar{\mathcal{H}}_m$ that maximizes $Ent_m$. The following theorem is the main result of this section.

\begin{theorem}\label{thm-hmcvg2h0}
The normalized height functions $\bar{h}_m$ converge to $h_0$ on $D$ when $m\rightarrow\infty$.
\end{theorem}
\begin{proof}
Similar to Theorem~\ref{theoremexistence}, if we consider the turned space $\tilde{\mathcal{H}}$, by compactness there is a converging subsequence of $(\tilde{\bar{h}}_m)_m$, saying $(\tilde{\bar{h}}_{m_l})_l$. Denote the limit function's preimage in $\mathcal{H}$ by $\bar{h}_0$ (it may depends on the choice of the subsequence but we will prove that this is not the case).

We prove that it is the same function as $h_0$, and we do this by showing that $Ent(\bar{h}_0)$ is equal to $Ent(h_0)$. The proof is divided into the following three parts. We first prove that
\begin{eqnarray}\label{inequality1/3}
Ent(h_0)\leq\liminf_{l\rightarrow\infty}Ent_{m_l}(\bar{h}_{m_l}),
\end{eqnarray}
then we show that $\bar{h}_0\in \mathcal{H}_0$, so we can apply Lemma~\ref{lemmaapproximation}, and finally we prove that
\begin{eqnarray}\label{inequality3/3}
Ent(\bar{h}_0)\geq\limsup_{m\rightarrow\infty}Ent_{m_l}(\bar{h}_{m_l}),
\end{eqnarray}
thus $Ent(\bar{h}_0)\geq Ent(h_0)$. By uniqueness of Theorem~\ref{theoremexistence} we prove that $\bar{h}_0=h_0$.

Finally, as we can apply this argument to any subsequence of $(\tilde{\bar{h}}_m)_m$ and prove that any subsequence has a converging subsubsequence whose limit is $h_0$, so the convergence of subsequence is in fact a convergence of the sequence $(\tilde{\bar{h}}_m)_m$ itself. Thus, without loss of generality, here below we suppose $(\tilde{\bar{h}}_m)_m$ converges so as to simplify the notation.

\vspace{0.5cm}

Begin by proving Inequality~(\ref{inequality1/3}), and without loss of generality we still take the setting of star-convexity used in Lemma~\ref{lemmaapproximation}. For any $\varepsilon>0$, by Lemma~\ref{lemmaapproximation} and Lemma~\ref{triangulation} there exist $\delta,\delta'>0$, functions $A_{\delta,\delta'}(h)$ and $h'$, such that the function $h'$ agrees with $A_{\delta,\delta'}(h)$ on $D\backslash D_{\delta}$, is piecewise linear on a $l$-triangle mesh of $D_{\delta}$, and
$$Ent(h')\geq Ent(h_0)-\frac{\varepsilon}{2}.$$

The local entropy $ent\circ\nabla h'$ is bounded, so by the same reason mentioned in the proof of Lemma~\ref{lemmaHlambda}, there exists some $M\in\mathbb{Z}^+$ such that the vertical partial derivative is less than $M$. Still by construction, on the band $[0,1]\times[1-\delta,1]$ and that $[0,1]\times[0,\delta]$, we have some frozen-like regions of shapes corresponding to the height function near the boundaries, so there exists $M'\in\mathbb{Z}^+$ such that outside $D^{M'}$ the vertical slope of $h'$ is $0$.

Thus, for all $m\geq\max\{M,M'\}$ we have $h'\in \bar{\mathcal{H}}_m$. Especially,
$$Ent_m(\bar{h}_m)\geq Ent_m(h').$$

As $h'$ is piecewise linear on $D_{\delta}$ and the number of pieces is finite, and on $D\backslash D_{\delta}$ it is taken to be the naive function in Definition and Lemma~\ref{DL1}, $\nabla h'$ only takes the values of $t=0$ together with a finite number of possible values. By Lemma~\ref{entropyexchange},
$$\lim_{m\rightarrow\infty}Ent_m(h')=Ent(h')$$
so for $m$ sufficiently large we have
$$Ent_m(h')\geq Ent(h')-\frac{\varepsilon}{2}.$$

In conclusion, we have that for $m$ sufficiently large,
$$Ent_m(\bar{h}_m)\geq Ent(h_0)-\varepsilon,$$
which proves Inequality~(\ref{inequality1/3}).

\vspace{0.5cm}

Now we prove that $\t{\bar{h}}_0\in \wt{\mathcal{H}}_0$. As in Theorem~\ref{theoremexistence}, if the set that $\frac{\partial\tbh_0}{\partial\t{y}}=1$ is positive, then for any $\varepsilon>0$, there exist a finite number of disjoint convex compacts $K_j,j=1,2,...,J$ of a positive measure independent of $\varepsilon$ and $M\in\mathbb{Z}^*$ such that on each compact the average vertical height change $av_y^{\Delta_k}(\t{\bar{h}}_m)$ is greater than $1-8\varepsilon$ if $m\geq M$. By an argument similar to that used in Theorem~\ref{theoremexistence}, it can be proved that
$$\limsup_{m\rightarrow\infty}\wt{Ent}_m(\t{\bar{h}}_m)=-\infty.$$
However, this contradicts to Inequality~(\ref{inequality1/3}) which says that
$$\wt{Ent}_m(\t{\bar{h}}_m)=Ent_m(\bar{h}_m)$$
has a lower bound, so we have proved that $\bar{h}\in \mathcal{H}_0$.

\vspace{0.5cm}

Now we are allowed to use Lemma~\ref{lemmaapproximation} to approximate $\bar{h}_0$ by a function of better regularity. For all $\varepsilon>0$, we can choose $\delta,\delta'$ small enough so that
$$Ent\big(A_{\delta,\delta'}(\bar{h}_0)\big)=(1-2\delta)^2 Ent(\bar{h}_0)+O(\delta\ln\delta),$$
so for any $\varepsilon>0$ we may choose $\delta$ and $\delta'$ so that the absolute value of the term $O(\delta\ln\delta)$ is less than $\varepsilon$.

Furthermore, by construction of the operator $A_{\delta,\delta'}$, for the same $\delta$ and $\delta'$ as above, we have that for any admissible function $h$:\\
(a) if $h\in \bar{\mathcal{H}}_m$, then the integral of local entropy function $ent_m$ of $A_{\delta,\delta'}(h)|_{D_{\delta}}$ is bigger than that of $h|_{D_{\delta}}$ (by concavity of $ent_m$).\\
(b) $A_{\delta,\delta'}(h)|_{D\backslash D_{\delta}}$ is the same function for any $h$, with two possible vertical derivative, and when the vertical derivative is non-zero, the horizontal one is bounded away from $\pm\frac{1}{2}$ by some constant of order $\delta$. Thus, the integral of ${ent_m\circ(\nabla A_{\delta,\delta'}(h))}$ on $D\backslash D_{\delta}$ converges in $m$ uniformly for all $h$ to a term of absolute value less than $\varepsilon$.

In conclusion, for any $\varepsilon>0$, there exists $\delta,\delta'>0$ and $M\in\mathbb{N}^*$ such that for any $m\geq M$ we have
$$Ent_m\big(A_{\delta,\delta'}(\bar{h}_m)\big)\geq (1-2\delta)^2 Ent_m(\bar{h}_m)-\varepsilon.$$

The boundness of $Ent_m(\bar{h}_m)$ and concavity of $ent_m$ implies the uniform boundness of
$ent_m\circ\nabla A_{\delta,\delta'}(\bar{h}_m)$ on $D$.


We also claim that $\nabla A_{\delta,\delta'}(\bar{h}_m)$ converges uniformly to $\nabla A_{\delta,\delta'}(\bar{h}_0)$ on $D$. In fact, by construction, they are all identical on $D\backslash D_{\delta}$ so have the same gradient there, and on $D_{\delta}$ we have that for any $m$,
$$\nabla A_{\delta,\delta'}(\bar{h}_m)=-\nabla U_{\delta'}*P_{0,\delta,0}\bar{h}_m.$$

According to Lemma~\ref{dini}, $P_{0,\delta,0}\bar{h}_m$ converge uniformly to $P_{0,\delta,0}\bar{h}_0$, so the convergence of $\nabla A_{\delta,\delta'}(\bar{h}_m)$ to $\nabla A_{\delta,\delta'}(\bar{h}_0)$ is uniform.

Define $\big(K_l(A_{\delta,\delta'}(\bar{h}_0))\big)_{l=1,2,...}$ as the following increasing sequence of subsets of $D$:
$$K_l(A_{\delta,\delta'}(\bar{h}_0))=\left\{(x,y)\in D:\frac{\partial A_{\delta,\delta'}(\bar{h}_0)}{\partial x}(x,y)\in[-\frac{1}{2}+\frac{1}{l},\frac{1}{2}-\frac{1}{l}]\times[0,A]\right\},$$
and the limit of this sequence is
$$K_{\infty}(A_{\delta,\delta'}(\bar{h}_0))=\left\{(x,y)\in D:\frac{\partial A_{\delta,\delta'}(\bar{h}_0)}{\partial x}(x,y)\in]-\frac{1}{2},\frac{1}{2}[\times[0,A]\right\}.$$

By the uniform convergence of $\nabla A_{\delta,\delta'}(\bar{h}_m)$ to $\nabla A_{\delta,\delta'}(\bar{h}_0)$, for all $l$, there exists $M$ such that for all $m>M$, on $K_l(A_{\delta,\delta'}(\bar{h}_0))$ we have
$$\frac{\partial A_{\delta,\delta'}(\bar{h}_m)}{\partial x}\in[-\frac{1}{2}+\frac{1}{2l},\frac{1}{2}-\frac{1}{2l}],$$

By Lemma~\ref{entropyexchange} argument (a), the convergence of $ent_m(s,t)$ to $ent(s,t)$ is uniform for any $(s,t)\in[-\frac{1}{2}+\frac{1}{2l},\frac{1}{2}-\frac{1}{2l}]\times[0,A]$, \emph{i.e.}, for any $\varepsilon>0$, there exists $M'$ such that for all $m'>M'$ and all $(s,t)\in[-\frac{1}{2}+\frac{1}{2l},\frac{1}{2}-\frac{1}{2l}]\times[0,A]$, we have
\begin{eqnarray}
|ent_m(s,t)-ent(s,t)|<\frac{\varepsilon}{2}.
\end{eqnarray}
The uniform convergence also implies that the space
$$\{ent_m(.,.),m\geq M\}\cup\{ent(.,.)\}$$
viewed as a subspace of continuous functions on the compact set $$(s,t)\in[-\frac{1}{2}+\frac{1}{2l},\frac{1}{2}-\frac{1}{2l}]\times[0,A]$$
is compact. Especially, by Arzela-Ascoli, they are equicontinuous: for the same $\varepsilon$, there exists $\varepsilon'>0$ such that for any $(s,t)$ and $(s',t')$ in $[-\frac{1}{2}+\frac{1}{2l},\frac{1}{2}-\frac{1}{2l}]\times[0,A]$ and for any $m>M$,
$$||(s,t)-(s',t')||<\varepsilon'\Rightarrow|ent_m(s,t)+ent_m(s',t')|<\frac{\varepsilon}{2}.$$

Again by the uniform convergence of $\nabla A_{\delta,\delta'}(\bar{h}_m)$ to $\nabla A_{\delta,\delta'}(\bar{h}_0)$, there exists $M''>M$ such that for all $m''\geq M''$,
$$\sup_{(x,y)\in K_l}||\nabla A_{\delta,\delta'}(\bar{h}_{m''})-\nabla A_{\delta,\delta'}(\bar{h}_{0})||<\varepsilon'.$$

Thus for all $m'>M'$, $m''>M''$, $(x,y)\in K_l$, we have
\begin{eqnarray*}
&\ &|ent_{m'}\circ\nabla A_{\delta,\delta'}(\bar{h}_{m''})(x,y)-ent\circ\nabla A_{\delta,\delta'}(\bar{h}_0)(x,y)|\\
&\leq&|ent_{m'}\circ\nabla A_{\delta,\delta'}(\bar{h}_{m''})(x,y)-ent_{m'}\circ\nabla A_{\delta,\delta'}(\bar{h}_0)(x,y)|\\
&\ &+|ent_{m'}\circ\nabla A_{\delta,\delta'}(\bar{h}_{0})(x,y)-ent\circ\nabla A_{\delta,\delta'}(\bar{h}_0)(x,y)|\\
&\leq&\varepsilon.
\end{eqnarray*}

Thus, on $K_l$, we have the following uniform convergence on $m'$ and $m''$:
$$\lim_{m'\rightarrow\infty,m''\rightarrow\infty}ent_{m'}\circ\nabla A_{\delta,\delta'}(\bar{h}_{m''})(x,y)=
ent\circ\nabla A_{\delta,\delta'}(\bar{h}_0)(x,y),$$
so
\begin{eqnarray}\label{sdfqsrgtfergtezgt}
\lim_{m\rightarrow\infty}Ent_m^{K_l}(A_{\delta,\delta'}(\bar{h}_m))=Ent^{K_l}(A_{\delta,\delta'}(\bar{h}_0)).
\end{eqnarray}

When $l$ tends to infinity, by bounded convergence, the right hand side of~(\ref{sdfqsrgtfergtezgt}) tends to $Ent^{K_{\infty}}(A_{\delta,\delta'}(\bar{h}_0))$ which is equal to $Ent(A_{\delta,\delta'}(\bar{h}_0))$. For the left hand side, the difference between $Ent_m^{K_l}(A_{\delta,\delta'}(\bar{h}_m))$ and $Ent_m^{K_{\infty}}(A_{\delta,\delta'}(\bar{h}_m))$ also converges to $0$ by bounded convergence, and the difference between $Ent_m^{K_{\infty}}(A_{\delta,\delta'}(\bar{h}_m))$ and $Ent_m(A_{\delta,\delta'}(\bar{h}_m))$ is equal to
$$\iint_{D\backslash K_{\infty}}ent_m(A_{\delta,\delta'}(\bar{h}_m))dxdy,$$
and by Lemma~\ref{entropyexchange} (b), for $l$ sufficiently large, then for $m$ large enough, the term above will be uniformly bounded from above by $\varepsilon$.

In conclusion, for $m$ large enough, we have
$$Ent\big(A_{\delta,\delta'}(\bar{h}_0)\big)\geq Ent_m\big(A_{\delta,\delta'}(\bar{h}_m)\big)-2\varepsilon,$$
and let $\varepsilon\rightarrow0$ we get
$$Ent(\bar{h}_0)\geq\limsup_{m\rightarrow\infty} Ent_m(\bar{h}_m),$$
which is Inequality~(\ref{inequality3/3}).
\qed
\end{proof}

For the explicit example of $h^{\partial}$ (\ref{sqlilfjrgeg}) considered in Section 1, then it is the limit case of tiling an hexagonal domain lozenges, a particular case studied in \cite{CLP}. By taking limit in their explicit forms, we get $h_0(x,y)=\frac{1}{2\pi}H(2x-1,2y-1)$, where
\begin{eqnarray}\label{shfmlzfuezr}
H(x,y)=
\begin{cases}
\arctan\frac{y}{\sqrt{1-x^2-y^2}}-x\arctan\frac{xy}{\sqrt{1-x^2-y^2}}\text{ if } \ x^2+y^2\leq 1\\
\frac{\pi}{2}(1-|x|)\text{ if } \ x^2+y^2>1, y>\frac{1}{2}\\
\frac{\pi}{2}(|x|-1)\text{ if } \ x^2+y^2>1, y<\frac{1}{2}.
\end{cases}
\end{eqnarray}
Readers can verify that (\ref{shfmlzfuezr}) is a particular solution of the Euler-Lagrange equation~(\ref{pde}). More general, in cases like this where $h^{\partial}$ is piecewise linear, we can observe frozen boundaries, which are algebraic curves described by \cite{KO1}. Authors of \cite{BeadFinitization} prove the frozen boundaries for the bead model corresponding to $abc$-hexagons.

\begin{figure}[H]
\centering
\includegraphics[clip, trim=4cm 10cm 0.5cm 10cm, width=0.9\textwidth]{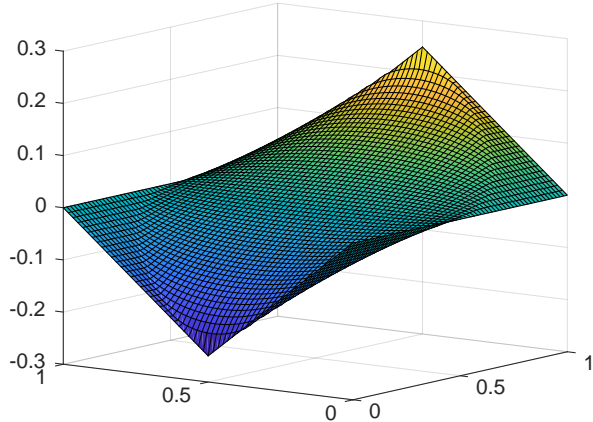}
\caption{The function $h_0$ for this boundary condition.}
\end{figure}


\section{Limit shape of standard Young tableaux}\label{sect6.4}

\subsection{Limit shape of standard (skew) Young tableaux}

In this section use the map from uniform bead configurations to standard Young tableaux (which can be skew) to study the limiting behavior of the later one. More precisely, we fix an arbitrarily chosen (skew) shape of Young diagram $\lambda$, given by two $\frac{1}{2}$-Lipschitz function
$$h(.,0),h(.,1):[0,1]\rightarrow\mathbb{R}$$
such that $h(x,1)-h(x,0)>0$ on $]0,1[$, and $\lambda$ is given by
$$\lambda=\{(x,z):2h(x,0)\leq z\leq 2h(x,1)\}.$$
Without loss of generality we can suppose that $h(0,0)=h(0,1)=0$. Readers can compare this to~(\ref{QKJFQSDKFQ}), page \pageref{QKJFQSDKFQ}.

For any $n\in\mathbb{N}^*$, define $\lambda_n$ as the normalized (skew) diagram that approximates $\lambda$ to an order of $O(\frac{1}{n})$, and the diagram is made of boxes of edge length $\frac{\sqrt{2}}{n}$ and written under the Russian convention. We use the $x-z$ coordinates.

Recall that $\mathcal{T}_{\lambda_n}$ is the set of standard tableaux of diagram $\lambda_n$. We can view a random tableau $T\in\mathcal{T}_{\lambda_n}$ as a random piecewise constant function on $\lambda_n$.

Consider $\Omega$ as a probability space, and consider $\mathbf{B}_n=\mathbf{B}_n(\omega)$ be a random bead configuration for the bead model corresponding to $\lambda_n$. By the map $\mathcal{Y}$ constructed in Section~\ref{sect6.4.1map} from bead configurations to the standard Young tableaux, we define the following random surface
\begin{eqnarray*}
\tau_n            :\ &\mathbb{R}^2\times\Omega &\rightarrow \ [0,1],\\
                     &(x,z;\omega)             &\mapsto     \ \mathds{1}_{(x,z)\in\lambda_n}\frac{\mathcal{Y}(\mathbf{B}_n(\omega))(x,z)}{|\lambda_n|},
\end{eqnarray*}
where we extend the function to the whole $\mathbb{R}^2$ plane and outside $\lambda_n$ we take $0$ by default. We have the following theorem.


\begin{theorem}\label{convergenceYoungSurface}
For a sequence of (skew) Young diagram $\lambda_n$ with an asymptotic shape $\lambda$, when $n\rightarrow\infty$, the random surfaces $\tau_n$ converge on any compact subset of the interior of $\lambda$ in probability and under uniform metric to a surface $\mathcal{S}$ supported on $\lambda$. The surface $\mathcal{S}$ is explicitly determined by the unique function $h_0\in \mathcal{H}$ that maximizes $Ent(.)$ with a boundary condition corresponding to $\lambda$. If we define for any $x\in[0,1]$
\begin{eqnarray*}
z_-(x)=\inf\{z:(x,z)\in\lambda\},\\
z_+(x)=\sup\{z:(x,z)\in\lambda\},
\end{eqnarray*}
and for any value $e\in[h(x,0),h(x,1)]$, define
$$h^{-1}_{0,x}(e)=\inf\{y\in[0,1]:h_0(x,y)\geq e\},$$
then the surface is given by
\begin{eqnarray*}
\mathcal{S}(x,z)=
\begin{cases}
h^{-1}_{0,x}\left(\frac{h_0(x,0)(z_+(x)-z)+h_0(x,1)(z-z_-(x))}{z_+(x)-z_-(x)}\right)&\ \text{if }(x,z)\in\lambda,\\
0 &\ \text{otherwise.}
\end{cases}
\end{eqnarray*}
\end{theorem}
\begin{proof}
Consider the corresponding sequence of bead models, which by construction has an asymptotic boundary condition $h^{\partial}$ determined by $\lambda$. By Theorem~\ref{Thm-Bead-convergenceinproba}, the normalized height function $h$ converges in probability to $h_0$ under the uniform metric.

For any such compact $K$ in the interior of $\lambda$, there exists $N(K)\in\mathbb{N}^*$ such that for any $n>N(K)$ we have $K\subset\lambda_n$. To prove that $\tau_n$ converges to $\mathcal{S}$ on $K$, we define another random function $\eta_n$. Recall that $y_{i,j}=y_{i,j}(\omega)$ is the random vertical coordinate of the $j^{th}$ bead on the $i^{th}$ thread (page \pageref{sqdflfhriluhreiugtruetgrkjfjfdfv}). For any bead configuration $\mathbf{B}_n$ with $n$ threads, $n>N(K)$, consider
\begin{eqnarray*}
\eta_n:\ &K &\rightarrow \ [0,1],\\
         &(x,z;\omega)     &\mapsto     \ y_{\lfloor xn \rfloor, \lfloor (z-z_-(x))n \rfloor}(\omega),
\end{eqnarray*}
\emph{i.e.}, for all $n$ we associate the box containing the point $(x,z)$
to a value equal to the $y$-coordinate of the bead corresponding to that box.

When $n\rightarrow\infty$, the random function $\eta_n$ converges in probability to $\mathcal{S}(x,z)$ on $K$. In fact, restricted to every $x$, $\eta_n(x,.,\omega)$ viewed as a stepwise constant function of $z$ is roughly the inverse (which can be well defined by using $\inf$ and $\sup$) of the normalized height function $h$ as a stepwise constant function of $y$. Meanwhile, still restricted to $x$, the surface $\mathcal{S}$ viewed as a function of $z$ is merely the inverse function of $h_0$ as a function of $y$, while the degenerating case (where the inversion fails) only happens in the frozen region of $h_0$, and by construction this doesn't matter. Thus the fact that a random surface $h$ converges to $h_0$ implies that $\eta_n$ converges to $\mathcal{S}$.

Now consider the difference between $\tau_n$ and $\eta_n$. For any bead configuration, conditioning to any ordering of $y_{i,j}$, the difference of $\tau_n$ and $\eta_n$ on a box ${i,j}$ is just equal to the difference of $y_{i,j}$ and its rank normalized by $|\lambda_n|$. So
\begin{eqnarray*}
\sup_{(x,z)\in K} |\tau_n-\eta_n|(x,z)\leq
\sup_{k=1,2,...,|\lambda_n|}\left|y_{i_k,j_k}-\frac{k}{|\lambda_n|}\right|,
\end{eqnarray*}
where the right hand side converges to $0$ in probability since the array $(y_{i_k,j_k})_{k=1,2,...,|\lambda_n|}$ is of the same law than a random ordered $|\lambda_n|$-dimensional array under the uniform measure on $[0,1]$.
\qed
\end{proof}

We remark that the random surface $\eta_n$ in the proof can be viewed as the limit of a normalized plane partition, which also gives the bead model. So the convergence of $\eta_n$ when $n\rightarrow\infty$ is nothing different from the convergence of bead configurations in Theorem~\ref{Thm-Bead-convergenceinproba}.

\vspace{0.5cm}

We call Theorem~\ref{convergenceYoungSurface} the ``surface version" convergence of a random (skew) Young tableau. It will be interesting to recover for a general skew case the results of \cite{PR} and \cite{Sni}, which we call as the ``contour curve version" convergence.

For any (skew) Young diagram $\lambda_n$, any standard tableau $T\in\mathcal{T}_{\lambda_n}$, and for any $\alpha\in]0,1[$, define ${Y_{\alpha,n}(T)}$ as the set composed of the boxes whose entries are less than $\alpha|\lambda_n|$, \emph{i.e.} a sub (skew) diagram of $\lambda_n$ given by
\begin{eqnarray}\label{Yan}
Y_{\alpha,n}(T):=\left\{\left(\frac{i}{n},\frac{j}{n}\right):T(i,j)\leq\alpha|\lambda_n|\right\}.
\end{eqnarray}
If we consider $T$ as a random standard tableau, then $Y_{\alpha,n}$ such defined is a random subdiagram of $\lambda_n$. We have

\begin{theorem}\label{boundarymeasureconvergence}
When $n\rightarrow\infty$, the Dirac measures of the upper boundary of $Y_{\alpha,n}$ converges to a Dirac measure on a curve determined explicitly by $h_0$. The curve is the contour line of height $\alpha$ of the surface $\mathcal{S}$ in Theorem~\ref{convergenceYoungSurface}.
\end{theorem}
\begin{proof}
Consider a boundary condition on $D$ of the bead model corresponding to $\lambda$ and a line segment $y=\alpha$ in $D$. Then for any $n\in\mathbb{N}^*$, define the following random subdiagram of $\lambda_n$ formed by the boxes corresponding to the beads under the line $y=\alpha$, \emph{i.e.},
\begin{eqnarray*}
Y_{\alpha,n}'=\left\{\left(\frac{i}{n},\frac{j}{n}\right):y_{i,j}\leq\alpha|\lambda_n|\right\}.
\end{eqnarray*}

The upper boundary of $Y_{\alpha,n}'$ normalized by $n$ is a $1$-Lipschitz function on $[0,1]$. By the same reason than in Theorem~\ref{convergenceYoungSurface}, this curve converges to a limiting curve determined by $h_0$.

Consider the difference of the diagrams $Y_{\alpha,n}'$ and  $Y_{\alpha,n}$. If $\alpha |\lambda_n|\leq y_{i_{\lfloor\alpha |\lambda_n|\rfloor},j_{\lfloor\alpha |\lambda_n|\rfloor}}$, then $Y_{\alpha,n}'\leq Y_{\alpha,n}$ and if $\alpha |\lambda_n|\geq y_{i_{\lfloor\alpha |\lambda_n|\rfloor},j_{\lfloor\alpha |\lambda_n|\rfloor}}$ then
$Y_{\alpha,n}'\geq Y_{\alpha,n}$. In either case, their difference is a random skew diagram, and the number of boxes in this diagram is
$$\big|\alpha |\lambda_n| -y_{i_{\lfloor\alpha n\rfloor},j_{\lfloor\alpha n\rfloor}}|\lambda_n|+O(1)\big|.$$

Thus, the area of this diagram normalized into $\lambda$ is equal to the difference of $\alpha$ and the $\lfloor\alpha|\lambda_n|\rfloor^{th}$ biggest element in a random array uniformly taking $|\lambda_n|$ points in $[0,1]$, which converges to $0$ in probability when $n\rightarrow\infty$. By the Lipshictz condition of the upper boundary of a Young diagram under the Russian convention, the norm sup of these upper boundaries converges to $0$ in probability. Thus the Dirac measure of the upper boundary of $Y_{\alpha,n}$ converges to the same limit than that of $Y_{\alpha,n}'$.
\qed
\end{proof}



The example (\ref{sqlilfjrgeg}) in Section 1 corresponds to the shape of a random square standard Young tableau. For a sequence of odd positive integers $n$, let $\lambda_n$ be a sequence of squares $\frac{n+1}{2}\times\frac{n+1}{2}$, so the corresponding bead model has $n$ threads, and $h_0$ is given by (\ref{shfmlzfuezr}). For any $\alpha\in]0,1[$, if we write the square diagram under the Russian convention and let the scale be $[0,1]\times[0,1]$, define $z_{\alpha}(x)$ as the limiting upper boundary of the first $\alpha$ proportion of boxes. This corresponds to a level line $y=\alpha$, and the difference of $z_{\alpha}(x)$ and the lower boundary of the diagram (\emph{i.e.} $z=|x-\frac{1}{2}|$) corresponds to the number of beads on the thread $x$ and between $y=\alpha$ and $y=0$. Since the total area is $\frac{1}{2}$, we have that
\begin{eqnarray*}
&\ &z_{\alpha}(x)=2\big(h_0(x,\alpha)-h_0(x,0)\big)+|x-\frac{1}{2}|\\
&=&
\begin{cases}
\frac{1}{\pi}\left(\arctan\frac{1-2\alpha}{\sqrt{1-(1-2x)^2-(1-2\alpha)^2}}-(1-2x)
\arctan\frac{(1-2\alpha)(1-2x)}{\sqrt{1-(1-2x)^2-(1-2\alpha)^2}}\right)+\frac{1}{2}\\
\hfill \text{ if }1-(1-2x)^2-(1-2\alpha)^2\geq 0,\\
|x-\frac{1}{2}| \hfill \text{  if }1-(1-2x)^2-(1-2\alpha)^2 < 0,\ \alpha<\frac{1}{2},\\
1-|x-\frac{1}{2}|\hfill \text{ if }1-(1-2x)^2-(1-2\alpha)^2 < 0,\ \alpha>\frac{1}{2}.\\
\end{cases}
\end{eqnarray*}

Thus we recover the result in \cite{PR}.
\begin{figure}[H]
\centering
\subfigure{
\includegraphics[width=0.35\textwidth]{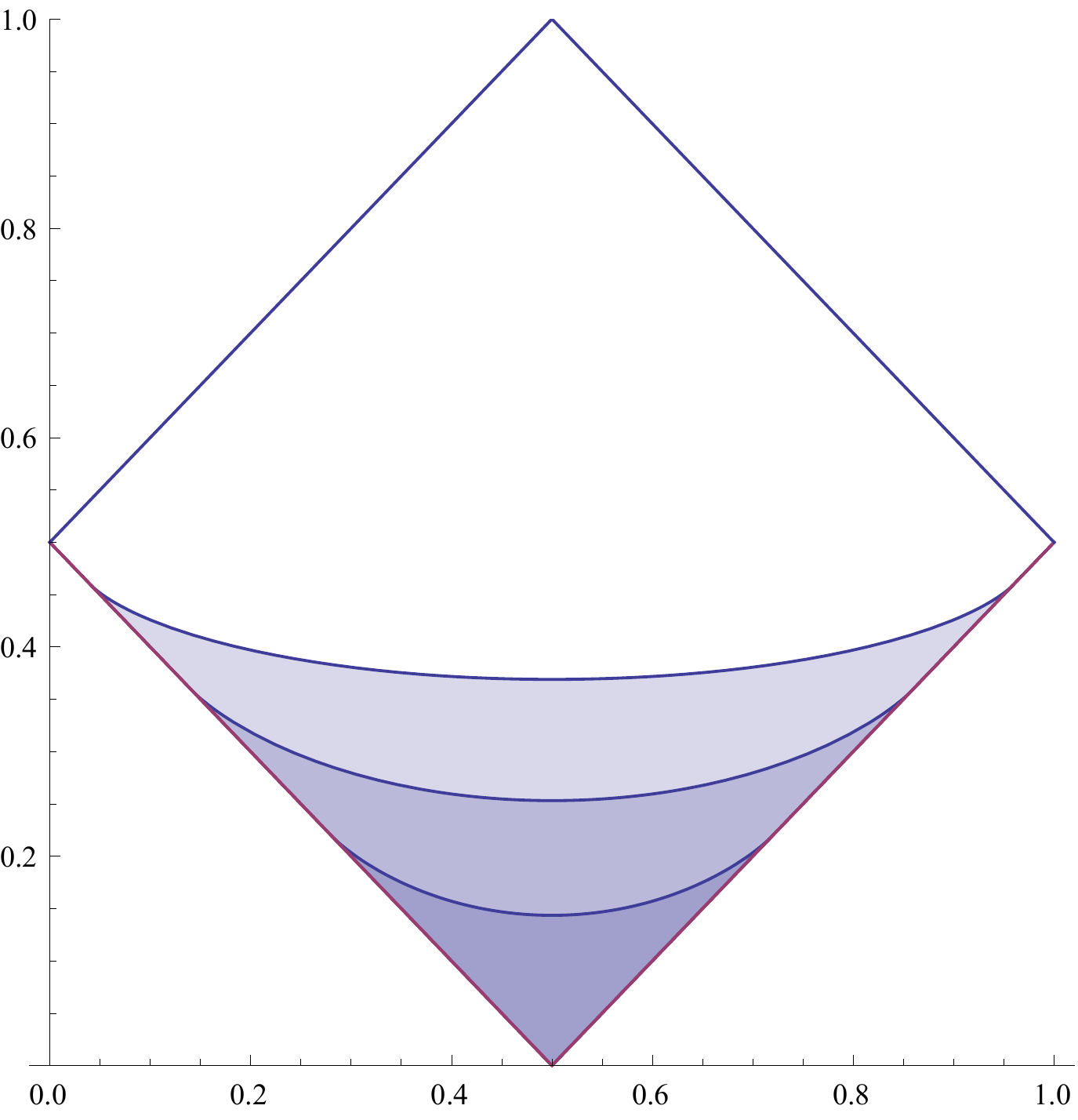}}
\hspace{0.1\textwidth}
\subfigure{
\includegraphics[width=0.35\textwidth]{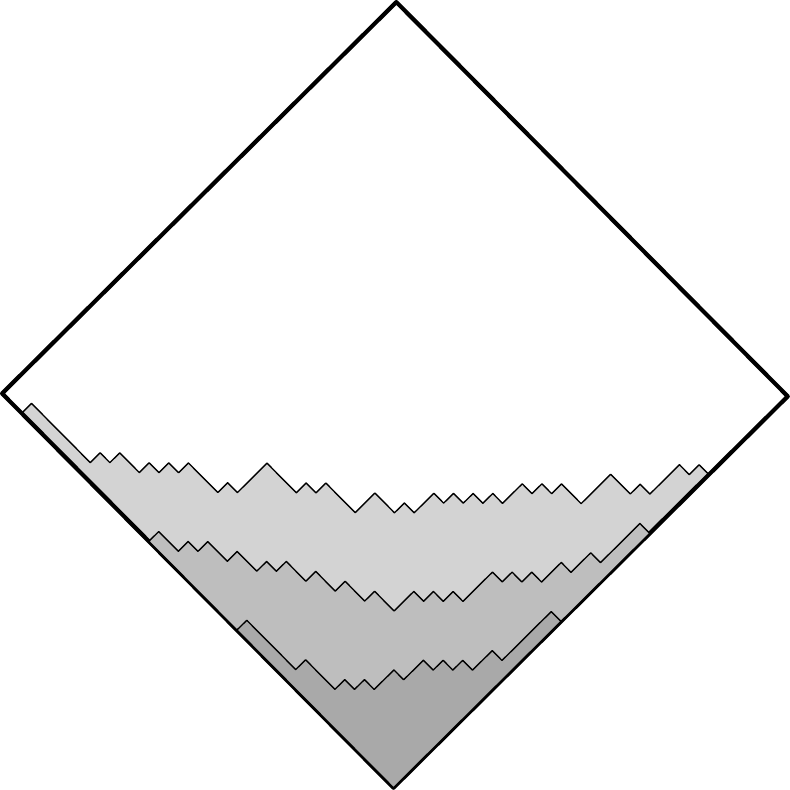}}
\caption{Level line of a standard square tableau, for $\alpha=0.05$, $0.15$ and $0.3$. Simulation for $|\lambda|=1600$.}\label{LevelLine}
\end{figure}

\newpage
\appendix
\noindent{\LARGE\textbf{Appendix}}
\section{Proof of Propositions~\ref{prop1} and~\ref{prop2}}
Fix the parameters $\alpha$ and $\gamma$ so we can simply write $Z_{mn,n}$ and $\tilde{Z}_n$ (this notation is only limited to this proof because the notation $Z_{mn,n}$ is already used for the partition function of lozenge tilings of $R_{mn,n}$ or $T_{mn,n}$). The characteristic polynomial is
\begin{eqnarray}
\det\hat{K}(z,w)=a^2z-(b+cw)^2/w=(\alpha/m)^2z-(e^{\alpha\gamma/m}+w)^2/w.\label{3}
\end{eqnarray}

The dimer partition function $Z_{mn,n}$ can be written as a linear combination of 4 terms ${Z_{mn,n}^{(\theta,\tau)}}$ \cite{Kas63,GalluccioLoebl,Tes00,CR07}, where ${\theta,\tau\in\{0,1\}}$, and the term ${Z_{mn,n}^{(\theta,\tau)}}$ is defined by
\begin{eqnarray*}
Z_{mn,n}^{(\theta,\tau)}=\prod_{\substack{z^n=(-1)^{\theta},\\w^{mn}=(-1)^{\tau}}}\det\hat{K}(z,w)=\prod_{\substack{z^n=(-1)^{\theta}, \\u^{n}=(-1)^{\tau}}}\prod_{w^m=u}\big((\alpha /m)^2z-(e^{\alpha\gamma /m}+w)^2/w\big).
\end{eqnarray*}
In the linear combination, the coefficient for every term is either $\frac{1}{2}$ or $-\frac{1}{2}$, where three terms have positive signs and one has negative sign.

Taking logarithm, we have
\begin{eqnarray}\label{srwkhraedsf}
\ln Z^{(\theta\tau)}_{mn,n}=\sum_{\substack{z^n=(-1)^{\theta},\\u^{n}=(-1)^{\tau}}}\sum_{w^m=u}\ln\big((\alpha/m)^2z-(e^{\alpha\gamma /m}+w)^2/w\big).
\end{eqnarray}

We rewrite~(\ref{3}) as
\begin{eqnarray}
-(w-w_1)(w-w_2)/w \label{decomposition1},
\end{eqnarray}
where $w_1$ and $w_2$ are two roots of the polynomial ~(\ref{3}) given by
\begin{eqnarray}
w_{1,2}=-1+\frac{\alpha}{m}(-\gamma\pm\sqrt{-z})+o\left(\frac{1}{m}\right)\label{roots},
\end{eqnarray}
which are close to $-1$ when $m$ is large. 

For given parameters $\alpha$ and $\gamma$, parity $(\theta,\tau)\in\{0,1\}^2$, and $n\in\mathbb{Z}$, for fixed $z\in S^1$, we first calculate the sum over $m$. As in the dimer model, we hope to approximate this sum by an integral. However, as we have stated, now the sum is of order $1$ so the integral (which approximates the sum divided by $m$) is of order $1/m$. So rather than to compare the sum divided by $m$ to the integral as usual, we compare the sum to $m$ times the integral. As a result, the difference between them is something a priori not negligible and should be determined precisely.

For every $(\theta,\tau)$, the term $Z_{mn,n}^{(\theta,\tau)}$ is a real number, so $\ln Z_{mn,n}^{(\theta,\tau)}$ is either real or purely imaginary. Since the partition function $Z_{mn,n}$ satisfies
$$Z_{mn,n}^{(\theta,\tau)}\leq Z_{mn,n}\leq 2\max_{(\theta,\tau)}Z_{mn,n}^{(\theta,\tau)},$$
we can just consider the case where $\ln Z_{mn,n}^{(\theta,\tau)}$ is real. Especially, we can just consider it real part.

For every $(\theta,\tau)$, by~(\ref{srwkhraedsf}), the real part of the logarithm $\ln Z_{mn,n}^{(\theta,\tau)}$ can be written as a  sum over $z$, $u$ and $w$. Consider first the sum over $w$, which by~(\ref{decomposition1}) can be rewritten as
\begin{eqnarray*}
\sum_{w^m=u}\big(\mathfrak{Re}\ln(-1)+\mathfrak{Re}\ln(w-w_1)+\mathfrak{Re}\ln(w-w_2)-\mathfrak{Re}\ln w\big).
\end{eqnarray*}
Since $\ln(-1)$ and $\ln w$ are purely imaginary, the above sum is always equal to
\begin{eqnarray}
\sum_{w^m=u}\big(\mathfrak{Re}\ln(w-w_1)+\mathfrak{Re}\ln(w-w_2)\big).\label{sum}
\end{eqnarray}

We need to compare this to the following value, which is an integral over the unit circle $S^1=\{w\in\mathbb{C}:|w|=1\}$:
\begin{eqnarray}
\mathfrak{Re}\Big(m\int_{S^1}\big(\ln(w-w_1)+\ln(w-w_2)\big)\frac{dw}{(2\pi i)w}\Big). \label{int}
\end{eqnarray}

We calculate the integral~(\ref{int}) first. Its value depends on whether the root $w_1$ and $w_2$ are inside or outside of the unit circle $S^1$. If a root is inside the unit circle $S^1$, we denote this root by $w_{in}$, and by the fact that $\ln w$ is purely imaginary we have
\begin{eqnarray*}
\mathfrak{Re}\Big(m\int_{S^1}\ln(w-w_{in})\frac{dw}{(2\pi i)w}\Big)
=\mathfrak{Re}\Big(m\int_{S^1}\ln\big(1-\frac{w_{in}}{w}\big)\frac{dw}{(2\pi i)w}\Big).
\end{eqnarray*}
Since $\left|\frac{w_{in}}{w}\right|<1$, we can develop $\ln\big(1-\frac{w_{in}}{w}\big)$ into a power series of $\frac{w_{in}}{w}$, whose powers in $w$ are not bigger than $-1$. The contour integral of any term in this series times $\frac{dw}{(2\pi i)w}$ around $S^1$ is $0$, so for an root inside $S^1$ we have
$$\mathfrak{Re}\Big(m\int_{S^1}\ln(w-w_{in})\frac{dw}{(2\pi i)w}\Big)=0.$$

If a root is outside $S^1$, denote it by $w_{out}$, we have
\begin{eqnarray*}
\mathfrak{Re}\Big(m\int_{S^1}\ln(w-w_{out})\frac{dw}{(2\pi i)w}\Big)=\mathfrak{Re}\Big(m\ln w_{out}+m\int_{S^1}\ln\big(1-\frac{w}{w_{out}}\big)\frac{dw}{(2\pi i)w}\Big).
\end{eqnarray*}
Again, we develop the logarithm $\ln\big(1-\frac{w}{w_{out}}\big)$ into a power series of $\frac{w}{w_{out}}$ with powers in $w$ bigger than $1$, so the contour integral is $0$.

In conclusion, if we use the indicator function $\mathds{1}_{out}$ to tell whether a root $w_{1,2}$ is outside $S^1$, then the integral~(\ref{int}) is equal to
\begin{eqnarray}\label{sqlidrazr}
m\sum_{j=1,2}\mathfrak{Re}\big(\mathds{1}_{out}\ln(w_{j})\big).
\end{eqnarray}

When $m$ is large, the roots $w_1$ and $w_2$ are both close to $-1$, so whether a root is inside or outside the unit circle mainly depends on its real part. When $m\rightarrow\infty$, we just need to check whether $\mathfrak{Re}(-\alpha\gamma\pm\alpha\sqrt{-z})$ is positive or negative, and when it is negative, the root $w_j$ is outside $S^1$, and in the logarithm of $\ln(w_{j})$, the only term of order $\frac{1}{m}$ is $\alpha\gamma\mp\alpha\mathfrak{Re}\sqrt{-z}$. Thus, when $m\rightarrow\infty$,~(\ref{int}) tends to
\begin{eqnarray}\label{integral2}
\sum_{+,-}\left(\alpha(\gamma\mp\mathfrak{Re}\sqrt{-z})\right)_+,
\end{eqnarray}
where $(x)_+$ is defined to be $\max\{x,0\}$.

Summing this term for $u\in S^1$, $u^{n}=(-1)^{\tau}$ just multiply it by $n$. Summing this for $z\in S^1$, $z^{n}=(-1)^{\theta}$ and divided by $n$ can be approximated by an integral over $S^1$:

\begin{eqnarray}
&\ &\int_{S^1}\big(\sum_{+,-}\alpha\left(\gamma\mp\mathfrak{Re}\sqrt{-z}\right)_+\big)\frac{dz}{(2\pi i)z}\nonumber\\
&=&\int_{0}^{2\pi}\frac{\alpha}{2\pi}\left(\Big(\gamma+\cos\big(-\theta/2\big)\Big)_+ +\Big(\gamma-\cos\big(-\theta/2\big)\Big)_+\right)d \theta\nonumber\\
&=&\int_0^{2\pi}\frac{\alpha}{\pi}(\gamma+\cos(\theta))_+ d\theta
=\frac{2}{\pi}\big(\alpha\gamma\arccos(-\gamma)+\alpha\sqrt{1-\gamma^2}\big),\label{partition}
\end{eqnarray}
and the error term between the sum over $u$ and $z$ and the integral~(\ref{partition}) is negligible.

Now we consider the difference between~(\ref{sum}) and~(\ref{int}). It suffices to consider the difference between terms of $w_1$, and the argument for $w_2$ is similar. We use Euler-Maclaurin formula, and to simplify the notation we denote by $f$ the function
$$f(x;w_1)=\ln(e^{i(\frac{2\pi x+\Arg u}{m})}-w_1),$$
then our problem is reduced to estimating the real part of
\begin{eqnarray}
\sum_{k=1}^{m}f(k;w_1)-\int_{0}^{m}f(x;w_1)dx.\label{diff}
\end{eqnarray}

A little remark is that here the function $f(.;w_1)$ is taken in the class $C^{\infty}$. Both the sum and the integral in~(\ref{diff}) differ from the original ones but only by imaginary constants, which causes no effect.

Apply the Euler-Maclaurin formula to~(\ref{diff}) to order two, then we have
\begin{eqnarray*}
&\ &\sum_{k=1}^{m}f(k;w_1)-\int_{0}^{m}f(x;w_1)dx\\
&=&\frac{1}{2}\big(f(m;w_1)-f(0;w_1)\big)+1/12\big(f'(m;w_1)-f'(0;w_1)\big) -\int_0^m\frac{1}{2}f''(x;w_1)B_2(x-\lfloor x\rfloor)dx\\
&=&\mathds{1}_{w_1\in D}\pi i-R(w_1),
\end{eqnarray*}
where the first term is imaginary and the second term $R(w_1)$ is the remainder term of the Euler-Maclaurin formula,
\begin{eqnarray}
R(w_1)=\int_0^m\frac{1}{2}f''(x)B_2(x-\lfloor x\rfloor)dx, \label{rest}
\end{eqnarray}
where $B_2$ is the Bernoulli polynomial of order 2 and $f''$ is equal to
$$f''(x;w_1)=\frac{e^{i(\frac{2\pi x+\Arg u}{m})}w_1}{\big(e^{i(\frac{2\pi x+\Arg u}{m})}-w_1\big)^2}\left(\frac{2\pi}{m}\right)^2.$$

The remainder term $R(w_1)$ is a priori not negligible. We will split the unit circle $S^1$ into the following three parts and respectively consider~(\ref{diff}) there. If we let $w=e^{i(\frac{2\pi x+\Arg u}{m})}$, then considering the following partition of $S_1$ is equivalent to considering a partition of ${x\in[0,m]}$:
\begin{eqnarray*}
S_{\RN{1}} &:=&\left\{w\in S_1:|w+1|>C_1\frac{\ln m}{\sqrt{m}}\right\}\\
S_{\RN{2}} &:=&\left\{w\in S_1:C_1\frac{\ln m}{\sqrt{m}}>|w+1|>C_2 \frac{\ln m}{m^{3/4}}\right\}\\
S_{\RN{3}} &:=&\left\{w\in S_1:|w+1|<C_2\frac{\ln m}{m^{3/4}}\right\}.\\
\end{eqnarray*}

Here $C_1$ is an arbitrary positive real number and $C_2$ is a positive real number small enough. When $m$ is sufficiently large, on $S_{\RN{1}}$, $f''=O(\frac{1}{m(\ln m)^2})$, so its contribution in the remainder term $R(w_1)$ (an integral over $S_{\RN{1}}$) tends to $0$ when $m\rightarrow\infty$. On $S_{\RN{2}}$, $f''=O(\frac{1}{\sqrt{m}(\ln m)^2})$. The length of $S_{\RN{2}}$ is $\frac{\ln m}{\sqrt{m}}$ so the terms in total is of order $\sqrt{m}\ln m$, so its contribution in $R(w_1)$ also tends to $0$.

To calculate the difference on $S_{\RN{3}}$, we approximate the sum and integral on the arc $S_{\RN{3}}$ respectively by the sum and integral on a line segment passing $w_1$ orthogonal to the $x-$axis whose length is of order $C_2\frac{\ln m}{m^{3/4}}$. In fact, if $w_1\in D$, this is a part of a chord passing $w_1$, whose length is $O\left(\frac{1}{\sqrt{m}}\right)$.

Without loss of generality we suppose that $m$ is even, let $x'=x-m/2$, so near $-1$ we have
$$e^{i(\frac{2\pi x+\Arg u}{m})}=-e^{i(\frac{2\pi x'+\Arg u}{m})}=-1-\frac{2\pi x'+\Arg u}{m} i+\frac{1}{2}\frac{(2\pi x'+\Arg u)^2}{m^2}+o\left(\frac{1}{m^2}\right),$$
where $|x'|<C_2m^{1/4}\ln m$. We have
\begin{eqnarray*}
\ln(-e^{i(\frac{2\pi x'+\Arg u}{m})}-w_1)-\ln (-1-\frac{2\pi x'+\Arg u}{m} i-w_1)=O\left(\frac{\ln m}{m^{3/4}}\right).
\end{eqnarray*}
This difference, if summed over $\{k\in\mathbb{Z},\ |k|<C_2m^{1/4}\ln m\}$ or integrated over $\{|x|<C_2m^{1/4}\ln m\}$, tends to zero when $m\rightarrow\infty$. Consider the sum
\begin{eqnarray*}
\sum_{k=-A}^{A-1}\ln (-1-\frac{2\pi k+\Arg u}{m} i-w_1)
=\sum_{k=-A}^{A-1}\Big(-\ln m+\ln \big(\alpha\gamma\mp\alpha\sqrt{-z}-(2\pi k+\Arg u)i\big)\Big),
\end{eqnarray*}
and the integral
\begin{eqnarray*}
\int_{-A}^{A}\ln (-1-\frac{2\pi x'+\Arg u}{m} i-w_1)dx'=\int_{-A}^{A}\Big(-\ln m+\ln \big(\alpha\gamma\mp\alpha\sqrt{-z}-(2\pi x'+\Arg u)i\big)\Big)dx'.\label{sumlocal}
\end{eqnarray*}

Their difference is independent of $m$. Let
$$g(x')=\ln \big(\alpha\gamma\mp\alpha\sqrt{-z}-(2\pi x'+\Arg u) i\big),$$
and apply the Euler-Maclaurin formula to $g$, we get
\begin{eqnarray*}
& &\sum_{k=-A}^{A-1} g(k)-\int_{-A}^A g(x')dx'\\
&=& \frac{1}{2}\big(g(A)-g(-A)\big)+\frac{1}{12}\big(g'(A)-g'(-A)\big)\\
&-&\frac{1}{2}\int_{-A}^A\frac{-4\pi^2}{(\alpha\gamma\mp\sqrt{-z}-(2\pi x'+\Arg u)i)^2}B(x'-\lfloor x'\rfloor)dx'.
\end{eqnarray*}

For $A$ large, the first term is close to $\frac{\pi}{2}i$ so its real part is close to $0$, and the second term is close to $0$. Considering the property of $\int\frac{dx}{(z+x i)^2}$, it is clear that the third part (the remainder term of Euler Maclaurin formula) is converging. Moreover, $\forall \varepsilon >0$, there exits $C(\varepsilon)\in\mathbb{N}^*$ such that outside $[-C(\varepsilon),C(\varepsilon)]$, uniformly on $u$ and $z$, the remainder term is less than $\varepsilon$. Note that we have proved the convergence of~(\ref{diff}) when $m\rightarrow\infty$, and we still need to prove that it can be arbitrarily small when $n\rightarrow\infty$.

For any given $\varepsilon$, we consider the difference of the finite sum $\sum_{k=-C(\varepsilon)}^{C(\varepsilon)-1} g(k)$ and the integral $\int_{-C(\varepsilon)}^{C(\varepsilon)} g(x')dx'$.

We calculate the sum first. Fix $k\in\mathbb{Z}\cap[-C(\varepsilon),C(\varepsilon)-1]$, the sum of $g(k)$ over $z:z^n=1$ divided by $n$ is approximated by
$$\ln\alpha+\int_{S^1}\ln \big(\gamma\mp\sqrt{-z}-\frac{2\pi k+\Arg u}{\alpha} i\big)\frac{dz}{(2\pi i)z}$$
within $o(1)$ as function of $n$. Let $s=s(k,u)=\gamma-\frac{2\pi k+\Arg u}{\alpha}i$. We consider the sum of those corresponding to $w_1$ and $w_2$,
\begin{eqnarray}
\int_{S^1}\ln(s+\sqrt{-z})\frac{dz}{(2\pi i)z}+\int_{S^1}\ln(s-\sqrt{-z})\frac{dz}{(2\pi i)z},\label{2sum}
\end{eqnarray}
and by taking $\sqrt{-z}=Z$, $Z\in S^1$, $\Arg Z$ from $\frac{\pi}{2}$ to $\frac{3\pi}{2}$,~(\ref{2sum}) is equal to
$$\int_{S^1}\ln(s+Z)\frac{2dZ}{(2\pi i)Z}+C,$$
where $C$ is imaginary so we don't need to consider. By a power expansion similar to what we used for $w$, we have
\begin{eqnarray*}
\mathfrak{Re}\Big(\int_{S^1}\ln(s+Z)\frac{2dZ}{(2\pi i)Z}\Big)=2\mathfrak{Re}(\mathds{1}_{s\not\in D}\ln s).
\end{eqnarray*}

We now approximate the sum over $u$ divided by $n$. Again this term can be approximated by an integral over $u$. Note that in fact if we glue the interval of $2\pi k-\Arg u$ for every $k\in[-C(\varepsilon),C(\varepsilon)-1]$, this gives exactly a continuous interval of
$[-2\pi C(\varepsilon), 2\pi C(\varepsilon)]$, and the double sum of $g$ over $u$ and $z$ is equal to
$$2\mathfrak{Re}\Big(\int^{2\pi C(\varepsilon)}_{-2\pi C(\varepsilon)}\mathds{1}_{(\gamma-\frac{y}{\alpha}i)\not\in D}\ln(\gamma-\frac{y}{\alpha}i) dy\Big).$$

If we do the same thing for the double integral of $g$ over $u$ and $z$, this gives exactly the same form. We see that integral over $z$ and $u$ makes disappear the difference between the sum and the integral of $g$.

By taking $\varepsilon\rightarrow 0$, we have proved the proposition.

\bibliographystyle{alpha}
\bibliography{mybib}
\end{document}